\documentclass[11pt,preprint]{imsart}

\usepackage{bm}
\usepackage[utf8]{inputenc}
\usepackage{mathrsfs}
\usepackage[T1]{fontenc}

\usepackage{comment}
\usepackage{listings}
\usepackage{algorithmicx}
\usepackage[ruled]{algorithm}
\usepackage{algpseudocode}
\usepackage{algpascal}
\usepackage{algc}

\usepackage{bigints}
\usepackage{amsmath,amsfonts,amssymb}
\usepackage{MnSymbol}

\usepackage{graphicx}
\usepackage[left=3cm,right=3cm,top=2cm,bottom=2cm]{geometry}

\usepackage{hyperref} 
\hypersetup{
    colorlinks=true,
    linkcolor=blue,
    filecolor=magenta,
    urlcolor=cyan,
}

\usepackage{ulem}

\usepackage{amsthm}

\usepackage{cleveref} 
\crefformat{equation}{(#2#1#3)}
\crefformat{line}{Line #1}
\crefformat{table}{Table #2#1#3}
\usepackage{multirow}
\crefformat{lemma}{Lemma #2#1#3}
\crefformat{proposition}{Proposition #2#1#3}
\crefformat{theorem}{Theorem #2#1#3}
\crefformat{corollary}{Corollary #2#1#3}
\crefformat{question}{Question #2#1#3}

\theoremstyle{definition}

\theoremstyle{plain}
\newtheorem{theorem}{Theorem}[section]
\newtheorem{corollary}[theorem]{Corollary}

\newtheorem{proposition}[theorem]{Proposition}

\newtheorem{lemma}[theorem]{Lemma}

\theoremstyle{remark}

\usepackage[english]{babel}
\usepackage[hang,small]{caption}
\usepackage{boxedminipage}
\usepackage{dsfont}

\usepackage[usenames,dvipsnames]{xcolor}
\usepackage{colortbl}

\usepackage[customcolors]{hf-tikz}

\tikzset{style green/.style={
    set fill color=green!50!lime!60,
    set border color=white,
  },
  style cyan/.style={
    set fill color=cyan!90!blue!60,
    set border color=white,
  },
  style orange/.style={
    set fill color=orange!80!red!60,
    set border color=white,
  },
  hor/.style={
    above left offset={-0.15,0.31},
    below right offset={0.15,-0.125},
    #1
  },
  ver/.style={
    above left offset={-0.1,0.3},
    below right offset={0.15,-0.15},
    #1
  }
}
\usepackage[textwidth=1.5cm, textsize=scriptsize]{todonotes}
\usepackage{pdfpages}

\newcommand{\ok}[1]{{\color{darkgreen}{ok}}}
\definecolor{darkgreen}{rgb}{0,0.6,0}

\DeclareMathOperator{\SN}{SN}

\DeclareMathOperator{\Tr}{Tr}

\def\beq{\begin{equation}}
  \def\eeq{\end{equation}}
  \def\beqn{\begin{eqnarray*}}
  \def\eeqn{\end{eqnarray*}}
  \def\bitem{\begin{itemize}}
  \def\eitem{\end{itemize}}
  \def\benum{\begin{enumerate}}
  \def\eenum{\end{enumerate}}
  \def\bmult{\begin{multline*}}
  \def\emult{\end{multline*}}
  \def\bcenter{\begin{center}}
  \def\ecenter{\end{center}}

\newcolumntype{P}[1]{>{\centering\arraybackslash}p{#1}}
\DeclareMathOperator*{\argmax}{arg\, max}
\newcommand{\argmin}{\mathop{\mathrm{arg\,min}}}

\DeclareMathOperator{\rank}{rank}

\DeclareMathOperator{\tr}{tr}

\def\cC{\mathcal{C}}
\def\cD{\mathcal{D}}

\def\cG{\mathcal{G}}
\def\cH{\mathcal{H}}
\def\cI{\mathcal{I}}

\def\cL{\mathcal{L}}

\def\cN{\mathcal{N}}

\def\cP{\mathcal{P}}
\def\cQ{\mathcal{Q}}
\def\cR{\mathcal{R}}

\def\cU{\mathcal{U}}

\def\cW{\mathcal{W}}

\def\cY{\mathcal{Y}}

\def\br{\mathbf{r}}

\def\1{{\mathbf 1}}
\def\0{{\mathbf 0}}

\newcommand{\bDelta}{{\boldsymbol\Delta}}

\def\bbC{\mathbb{C}}

\def\bbE{\mathbb{E}}

\def\bbN{\mathbb{N}}

\def\bbP{\mathbb{P}}

\def\bbR{\mathbb{R}}
\def\bbS{\mathbb{S}}

\def\bbZ{\mathbb{Z}}

\DeclareMathAlphabet\mathbfcal{OMS}{cmsy}{b}{n}

\newcommand{\E}{\operatorname{\mathbb{E}}}
\renewcommand{\P}{\operatorname{\mathbb{P}}}

\newcommand{\floor}[1]{\left\lfloor#1\right\rfloor} 
\newcommand{\ceil}[1]{\left\lceil#1\right\rceil}
\newcommand{\abs}[1]{\left\lvert#1\right\rvert}

\def\proscal<#1,#2>{\langle #1,#2\rangle}
\newcommand{\poubelle}[1]{}

\newcommand{\loc}{{\mathrm{loc}}}

\newcommand{\op}{\mathrm{op}}

\newcommand{\dya}{\mathrm{dya}}
\newcommand{\lone}{\mathrm{L}_1}
\newcommand{\noise}{E}

\newcommand{\algoPrincipal}{\mathbf{ISR}}
\newcommand{\algoRefineLocally}{\mathbf{SLR}}

\newcommand{\DAG}{\mathbf{DAG}}
\newcommand{\rk}{\mathbf{rk}}
\newcommand{\pth}{\mathbf{path}}

\newcommand{\iso}{\mathrm{iso}}
\newcommand{\biso}{\mathrm{biso}}

\newcommand{\reco}{\mathrm{reco}}
\newcommand{\perm}{\mathrm{perm}}

\newcommand{\spaceAnd}{\quad \text{ and } \quad}

\def\bG{\mathbf{G}}

\def\bI{\mathbf{I}}

\def\bP{\mathbf{P}}

\def\br{\mathbf{r}}

\def\1{{\mathbf 1}}
\def\0{{\mathbf 0}}

\def\bbC{\mathbb{C}}

\def\bbE{\mathbb{E}}

\def\bbN{\mathbb{N}}

\def\bbP{\mathbb{P}}

\def\bbR{\mathbb{R}}
\def\bbS{\mathbb{S}}

\def\bbZ{\mathbb{Z}}

\DeclareMathAlphabet\mathbfcal{OMS}{cmsy}{b}{n}

\alglanguage{pseudocode}

\begin{document}

\begin{frontmatter}

	\title{Optimal rates for ranking a permuted isotonic matrix in polynomial time}
	\runtitle{Optimal ranking}
	\begin{aug}
	\author{Emmanuel Pilliat, Alexandra Carpentier, and Nicolas Verzelen}
	\runauthor{Pilliat et al.}
	\end{aug}

	\begin{abstract}
        We consider a ranking problem where we have noisy observations from a matrix with isotonic columns whose rows have been permuted by some permutation $\pi^*$. This encompasses many models, including crowd-labeling and ranking in tournaments by pair-wise comparisons. In this work, we provide an optimal and polynomial-time procedure for recovering $\pi^*$, settling an open problem in~\cite{flammarion2019optimal}. As a byproduct, our procedure is used to improve the state-of-the art for ranking problems in the stochastically transitive model (SST). 
        Our approach is based on iterative pairwise comparisons by suitable data-driven weighted means of the columns. These weights are built using a combination of spectral methods with new dimension-reduction techniques. In order to deal with the important case of missing data, we establish a new concentration inequality for sparse and centered rectangular Wishart-type matrices. 
  %
	\end{abstract}

\end{frontmatter}

\maketitle

\section{Introduction}

Ranking problems have recently spurred a lot of interest in the statistical and computer science literature. This includes a variety of problems ranging from ranking experts/workers in crowd-sourced data, ranking players in a tournament or equivalently sorting objects based on pairwise comparisons.

To fix ideas, let us consider a problem where we have noisy partial observations from an unknown matrix $M\in [0,1]^{n\times d}$. In crowdsourcing problems, $n$ stands for the number of experts (or workers), $d$ stands for the number of questions (or tasks) and $M_{i,k}$ for the probability that expert $i$ answers question $k$ correctly. For tournament problems, we have $n=d$ players (or objects) and $M_{i,k}$ stands for the probability that player $i$ wins against player $k$. Based on these noisy data, the general goal is to provide a full ranking of the experts or of the players. 

Originally, these problems were tackled using parametric model for the matrix $M$. Notably, this includes the noisy sorting model \cite{braverman2008noisy} or Bradley-Luce-Terry model \cite{bradley1952rank}. Still, it has been observed that these simple models are often unrealistic and do not tend to fit well. 

This has spurred a recent line of literature where strong parametric assumptions are replaced by non-parametric assumptions \cite{shah2015estimation,shah2016stochastically,shah2019feeling,shah2020permutation, mao2020towards,mao2018breaking,liu2020better,flammarion2019optimal,bengs2021preference,saad2023active}. In particular, for tournament problems, the strong stochastically transitive (SST) model presumes that the square matrix $M$ is, up to a common permutation $\pi^*$ of the rows and of the columns, bi-isotonic and satisfies the skew symmetry condition $M_{i,k}+M_{k,i}=1$. Although optimal rates for estimation of the permutation $\pi^*$ have been pinpointed in the earlier paper of Shah et al.~\cite{shah2016stochastically}, there remains a large gap between these optimal rates and the best known performances of polynomial-time algorithms. This has led to  conjecture the existence of a statistical-computational gap~\cite{mao2020towards, liu2020better}.

For crowdsourcing data, the counterpart of the SST model is the so-called bi-isotonic model, where the rectangular matrix $M$ is bi-isotonic, up to an unknown permutation $\pi^*$ of its rows and an unknown permutation $\eta^*$ of its columns. 
This model turns out to be really similar to the SST model and the existence of a statistical-computational gap has also been conjectured~\cite{mao2020towards}. 

\medskip 

In this paper, we tackle a slightly different route and we consider the arguably more general isotonic model~\cite{flammarion2019optimal}. The only assumption is that all the columns of $M$ are nondecreasing up to an unknown permutation of the rows, making the isotonic model more flexible than the bi-isotonic and SST models. It is in fact the most general model under which an unambiguous ranking of the experts is well-defined. In this model as well, there is a gap between the (statistical) optimal rates, and the rate obtained by the (polynomial-time) algorithm in~\cite{flammarion2019optimal}.

Our main contributions are as follows. For the isotonic model, we establish the optimal rate for recovering the permutation, and we introduce a polynomial-time procedure achieving this rate, thereby settling the absence of any computational gap in this model. Besides, our procedure and results have important consequences when applied to the SST and bi-isotonic model. More specifically, we achieve the best known guarantees  in these two models~\cite{liu2020better,mao2018breaking} and even improve them in some regimes.

\subsection{Problem formulation}\label{ss:pf}


Let us further introduce our model.  A bounded matrix $A\in [0,1]^{n\times d}$ is said to be isotonic if its columns are nondecreasing, that is $A_{i,k}\leq A_{i+1,k}$ for any $i\in [n-1]$ and $k \in [d]$. Henceforth, we write $\mathbb{C}_{\iso}$ for the collection of all $n\times d$  isotonic matrices taking values in $[0,1]$. In our model, we recall that we  assume that the signal matrix $M$ is isotonic up to an unknown permutation of its rows. In other words, there exists a permutation $\pi^*$ of $[n]$ such that the matrix $M_{\pi^{*-1}}$ defined by $(M_{\pi^{*-1}})_{i,k}= (M_{\pi^{*-1}(i),k})$ has nondecreasing columns, that is
\begin{equation}\label{eq:assumption_isotonic}
	M_{\pi^{*-1}(i),k} \leq M_{\pi^{*-1}(i+1),k} \enspace ,
\end{equation} 
for any $i \in \{1, \dots, n-1\}$ and $k \in \{1, \dots, d\}$, or equivalently $M_{\pi^{*-1}} \in \bbC_{\iso}$.
Henceforth, $\pi^*$ is called an oracle permutation. Using the terminology of crowdsourcing, we refer to  $i^{\mathrm{th}}$ row of $M$ as \emph{expert} $i$ and to $k^{\mathrm{th}}$ column as  \emph{question} $k$.

In this work, we have $N$ partial and noisy observations of the matrix $M$ of the form $(x_t,y_t)$ where 
\begin{equation}\label{eq:model_partial}
y_t = M_{x_t} + \varepsilon_t \quad t=1,\ldots, N \enspace .
\end{equation}
For each $t$, the position $x_t\in [n]\times [d]$ is sampled uniformly. The noise variables $\varepsilon_t$'s are  independent and their distributions only 
depend on the position $x_t$. We only assume that all these distributions are centered and are subGaussian with a subGaussian norm of at most $1$ -- see e.g. \cite{wainwright2019high}. In particular, this encompasses the typical case where the $y_t$'s follow Bernoulli distributions with parameters $M_{x_t}$.

As usual in the literature e.g.~\cite{pilliat2022optimal,liu2020better, mao2020towards}, we use, for technical convenience, the Poissonization trick which amounts to assuming that the number $N$ of observations has been sampled according to a Poisson distribution with parameter $\lambda n d$. We refer to $\lambda>0$ as the sampling effort. When $\lambda > 1$, we have, in expectation, several independent observations per entry $(i,j)$ - and $\lambda=1$ means that there is on average one observation per entry. In this paper, we are especially interested in the sparse case where $\lambda$ is much smaller than one, i.e.~the case where we have missing observations for some entries. We refer to $\lambda=1$ as the full observation regime at it bears some similarity to the case often considered in the literature --e.g.~\cite{shah2016stochastically,flammarion2019optimal}, where we have a full observation of the matrix,
\beq\label{eq:model_0}
Y= M + E' \in \mathbb{R}^{n\times d}\ . 
\eeq
The entries of the noise matrix $E'$  are independent, centered, and $1$-subGaussian. 





In this work, we are primarily interested in estimating the permutation $\pi^*$. Given an estimator $\hat{\pi}$, we use the square Frobenius norm $\|M_{\hat{\pi}^{-1}} - M_{\pi^{*-1}} \|_F^2$ as the loss. This loss quantifies the distance between the matrix $M$ reordered according to the estimator $\hat{\pi}$ and the matrix $M$ sorted according to the oracle permutation $\pi^*$. This loss is explicitly used in~\cite{liu2020better,pilliat2022optimal} and is implicit in earlier works --see e.g.~\cite{shah2016stochastically}.

We define the associated optimal risk of permutation recovery as a function of the number $n$ of experts, the number $d$ of question and the sampling effort $\lambda$,
\beq\label{eq:minimax_risk_perm}
\cR^*_{\perm}(n,d,\lambda)= \inf_{\hat \pi}  \sup_{\stackrel{\pi^*\in \Pi_n}{M:\,  M_{\pi^{*-1}}\in \mathbb{C}_{\iso}}} \mathbb E_{(\pi^*, M)} [\|M_{\hat \pi^{-1}} - M_{\pi^{*-1}}\|_F^2] \enspace , 
\eeq
where the infimum is taken over all estimators. Here, $\Pi_n$ stands for the collection of all permutations of $[n]$. If the main focus is not only to estimate $\pi^*$, but also to reconstruct the unknown matrix $M$, we also consider the optimal reconstruction rate
\begin{equation}\label{eq:minimax_risk_est}
	\cR^*_{\reco}(n,d,\lambda)= \inf_{\hat M}\sup_{\stackrel{\pi^*\in \Pi_n}{M:\,  M_{\pi^{*-1}}\in \mathbb{C}_{\iso}}} \E\left[\|\hat M - M\|_F^2\right]\ \enspace .
\end{equation}
It turns out that reconstructing the matrix $M$ is more challenging than estimating the permutation $\pi^*$. Considering both risks allows to disentangle the reconstruction of the matrix $M$: looking at both enables to distinguish the error that is due to estimating the permutation, from the error that comes from estimating an isotonic matrix. 


\subsection{Past results on the isotonic model and our contributions}\label{ss:contributions}

In, the specific case where $d=1$ (a single column), our model is equivalent to uncoupled isotonic regression and is motivated by optimal transport. Rigollet and Niles-Weed~\cite{rigollet2019uncoupled} have established that the reconstruction error of $M$ is of the order of $n (\tfrac{\log\log(n)}{\log(n)})^2$. 

For the general case $d\geq 1$, Flammarion et al.~\cite{flammarion2019optimal} have shown\footnote[1]{The authors consider the isotonic model as a subcase of a seriation model, where each columns of $M_{\pi^{*-1}}$ is only assumed to be unimodal.} that the optimal reconstruction error in the full observation model~\eqref{eq:model_0} is of the order of $n^{1/3}d+n$. However, the corresponding procedure is not efficient. They also introduce an efficient procedure that first estimates $\pi^*$ using a score based on row comparisons on $Y$. Unfortunately, this method only achieves a reconstruction error of the order of $n^{1/3}d + n\sqrt{d}$ which is significantly slower than the optimal one. Whether or not there is a statistical-computationnal gap was therefore an open problem.

We prove in this work that there is no computational statistical gap in this model. More precisely, we introduce estimators that are both polynomial-time and minimax optimal up to some polylog factors.
To that end, we characterize the optimal risks $\cR^*_{\perm}(n,d,\lambda)$ and $\cR^*_{\reco}(n,d,\lambda)$ of permutation estimation and matrix reconstruction, for all possible number of experts $n\geq 1$, number of questions  $d\geq 1$ and all sampling efforts $\lambda$, up to some polylog factors in $nd$. \Cref{tab:rates} summarizes our findings in the arguably most interesting cases\footnote[2]{We are indeed mostly interested in the more realistic sparse observation regime (meaning $\lambda  \leq 1$). The case $\lambda \leq 1/d$ leads to the trivial minimax bound of order $nd$ for both reconstruction and estimation, as in this case we have less than one observations per expert on average. As for the case $\lambda > 1/d$ but $\lambda \leq 1/n$, we have less than one observation per question on average, and this leads to a minimax risk of order $n\sqrt{d/\lambda}$ for permutation estimation and of order $nd$ for matrix recontruction.} $\lambda \in [1/(n\land d), 1]$. 


\begin{table}[h]
\begin{center}
	\begin{tabular}{| c | c | c |}
		\noalign{\global\arrayrulewidth=0.3mm}
		\arrayrulecolor{gray}
		  \hline
		  ~ & $n \leq d^{3/2}\sqrt{\lambda}$ & $d^{3/2}\sqrt{\lambda} \leq n$\\ \hline
		  $\mathcal R^*_{\perm}$ & {$n^{2/3}\sqrt{d}\lambda^{-5/6}$} & {$n/\lambda$}\\ \hline
		  $\mathcal R^*_{\reco}$  & $n^{1/3}d\lambda^{-2/3}$ & $n/\lambda$\\
		  \hline
		\end{tabular} \enspace ,
\end{center}
\caption{Optimal rates in our model, for all possible values of $n,d$ and $\lambda \in [1/(n\land d), 1]$, up to a polylogarithmic factor in $nd$. These rates are achieved by polynomial-time estimators.}
\label{tab:rates}
\end{table}

\subsection{Implication for other models and connection to the literature}\label{subsec:literature}



As discussed earlier, the isotonic model is quite general and encompasses both the bi-isotonic model for crowdsourcing problems as well the SST model for tournament problems. 

Let us first focus on the SST model which corresponds to the case where $n=d$ together with a bi-isotonicity and a skew-symmetry assumption. In the full observation scheme (related to the case $\lambda=1$) where one observes the noisy matrix $n\times n$, Shah et al.~\cite{shah2016stochastically} have established that the optimal rates for estimating $\pi^*$ and reconstructing the matrix $M$ are of the order of $n$. In contrast, their efficient procedure which estimates $\pi^*$ according to the row sums of $Y$ only achieves the rate of $n^{3/2}$. In more recent years, there has been a lot of effort dedicated to improving this $\sqrt{n}$ statistical-computational gap. The SST model was also generalized to partial observations by  \cite{chatterjee2019estimation}, which corresponds to  $\lambda \leq 1$. They introduced an efficient procedure that targets a specific sub-class of the SST model, and that achieves a rate of order $n^{3/2}\lambda^{-1/2}$ in the worst case for matrix reconstruction.

Recently, a few important contributions tackling both the bi-isotonic model and the SST model made important steps towards better  understanding the statistical-computational gap. We first explain how their results translate in the SST model. Mao et al.~\cite{mao2020towards, mao2018breaking} introduced a polynomial-time procedure handling partial observation, achieving a rate of order $n^{5/4}\lambda^{-3/4}$ for matrix reconstruction. Nonetheless, \cite{mao2020towards} failed to exploit global information shared between the players/experts -- as they only compare players/experts two by two -- as pointed out by \cite{liu2020better}. Building upon this remark, \cite{liu2020better} managed to get the better rate $n^{7/6+o(1)}$ with a polynomial-time method in the case $\lambda = n^{o(1)}$. 

Let us turn to the more general bi-isotonic model. Here, the rectangular matrix $M \in \bbR^{n \times d}$ is bi-isotonic up an unknown permutation $\pi^*$ of the rows and an unknown permutation $\eta^*$ of the columns. Since $M$ is not necessarily square, this model can be used in more general crowd-sourcing problems. The optimal rate for reconstruction in this model with partial observation has been established in \cite{mao2020towards} to be of order $\nu(n,d,\lambda):= (n\lor d)/\lambda + \sqrt{nd/\lambda}\land n^{1/3}d\lambda^{-2/3}\land d^{1/3}n\lambda^{-2/3}$ up to polylog factors, in the non-trivial regime where $\lambda \in [1/(n\land d), 1]$. However, the polynomial-time estimator provided by Mao et al. \cite{mao2020towards} only achieves the rate $n^{5/4}\lambda^{-3/4} + \nu(\lambda, n, d)$. In a nutshell, Mao et al. first compute column sums to give a first estimator of the permutation of the questions. Then, they compare the experts on aggregated blocks of questions, and finally compare the questions on aggregated blocks of experts. As explained in the previous paragraph for SST models, Liu and Moitra~\cite{liu2020better}  improved this rate to $n^{7/6 + o(1)}$ in the square case $(n=d)$, with a subpolynomial number of observations per entry $(\lambda = n^{o(1)})$. 	
Their estimators of the permutations $\pi^*$, $\eta^*$ were based on hierarchical clustering and on local aggregation of high variation areas. Both \cite{liu2020better,mao2020towards} made heavily use of the bi-isotonicity structure of $M$ by alternatively sorting the columns and rows.
As mentioned for the SST model, the order of magnitude $n^{7/6+o(1)}$ remains nevertheless suboptimal, and whether there exists an efficient algorithm achieving the optimal rate in this bi-isotonic model remains an open problem. 

We now discuss the implications of our work concerning the bi-isotonic model and SST model. First, in the full observation setting $(\lambda=1)$ and square case for the bi-isotonic model $(n=d)$, we reach in polynomial-time the upper bound $n^{7/6}$ up to polylog factors, for both permutation estimation and matrix reconstruction. In particular, we improve the rate in \cite{liu2020better} by a subpolynomial factor in $n$, and we do not need a subpolynomial number of observation per entry. Moreover, our procedure being primarily designed for the isotonic model, it does not require any shape constraint on the rows in contrast to \cite{liu2020better,mao2020towards}.
Beyond the full observation regimes, we provide guarantees on our estimator of $\pi^*$ for different values of $\lambda$. In particular, in \Cref{cor:ub_reco_biso}, we derive an estimator of the matrix $M$ that achieves a maximum reconstruction risk $\sup_{\pi^*, \eta^*, M} \E\left[\|\hat M - M_{\pi^{*-1}\eta^{*-1}}\|_F^2\right]$ of order less than $n^{7/6}\lambda^{-5/6}$ up to polylogs, thereby improving the state-of-the-art polynomial-time methods in partial observation \cite{mao2020towards}. Lastly, we perform our analysis in the general rectangular case, giving guarantees for general values of $d$.

The optimal risks and the known polynomial-time upper bounds for the isotonic, bi-isotonic with two permutations and SST models are summarized in \Cref{tab:comparison}. For the sake of simplicity, we focus in the table to the specific case case $n=d$ and $\lambda \in [1/n, 1]$.

\begin{table}[h]
\begin{center}
	\begin{tabular}{|P{2cm} | P{1cm} | P{3.1cm} | P{3.5cm} | P{3.5cm} |}
		\noalign{\global\arrayrulewidth=0.3mm}
		\arrayrulecolor{gray}
		  \hline
		  \multicolumn{2}{|c|}{\multirow{2}{*}{\vspace{-1cm}$\begin{array}{cc} 
			\text{Different models, with}\\
			\text{$M \in \bbR^{n\times n}$}\\
		\end{array}$}} & {\bf Isotonic} & Bi-isotonic($\pi^*, \eta^*)$ & SST \\ 
		  \multicolumn{2}{|c|}{}&{\bf$M_{\pi^{*-1}}$ has nondecreasing columns } &$M_{\pi^{*-1} \eta^{*-1}}$ has nondecreasing columns and rows& $M_{\pi^{*-1} \pi^{*-1}}$ has nondecreasing columns and rows, and $M_{ik} + M_{ki} = 1$\\
		  \hline
		  Permutation estimation & Poly. Time & {\bf $n^{7/6}\lambda^{-5/6}$ {\bf [Th~\ref{th:UB}]}} &
		  $\begin{array}{cc} 
			n^{7/6+o(1)}  & \text{\cite{liu2020better}}(\lambda=n^{o(1)})\\
			n^{7/6}\lambda^{-5/6} & \textbf{[Th~\ref{th:UB}]}\\
		\end{array}$
		& $\begin{array}{cc} 
			n^{7/6+o(1)}  & \text{\cite{liu2020better}}(\lambda=n^{o(1)})\\
			n^{7/6}\lambda^{-5/6} & \textbf{[Th~\ref{th:UB}]}\\
		\end{array}$\\
		\cline{2-5}
		~ & optimal rate& {\bf $n^{7/6}\lambda^{-5/6}$ [Th~\ref{th:lower_bound_poisson}]} &
		  $n/\lambda$  \cite{mao2020towards} & $n/\lambda$ \cite{mao2020towards}\\
		  \hline
		  Matrix \newline reconstruction &Poly. Time & 
		  $\begin{array}{cc} 
			n^{3/2} & (\lambda=1)\text{\cite{flammarion2019optimal}}\\
			n^{4/3}\lambda^{-2/3} & \textbf{[Cor~\ref{cor:ub_reco_biso}]}\\
		\end{array}$
		& $\begin{array}{cc} 
			n^{7/6+o(1)}  & \text{\cite{liu2020better}}(\lambda=n^{o(1)})\\
			n^{5/4}\lambda^{-3/4} & \text{\cite{mao2020towards}}\\
			n^{7/6}\lambda^{-5/6} & \textbf{[Cor~\ref{cor:ub_reco_biso}]}\\
		\end{array}$ & $\begin{array}{cc} 
			n^{7/6+o(1)}  & \text{\cite{liu2020better}}(\lambda=n^{o(1)})\\
			n^{5/4}\lambda^{-3/4} & \text{\cite{mao2020towards}}\\
			n^{7/6}\lambda^{-5/6} & \textbf{[Cor~\ref{cor:ub_reco_biso}]}\\
		\end{array}$\\
		\cline{2-5}
		  ~ & optimal rate & $n^{4/3}\lambda^{-2/3}$ \cite{flammarion2019optimal}\newline (also [Prop~\ref{prop:LB_reconstruction}]) &
		  $n/\lambda$  \cite{mao2020towards} & $n/\lambda$ \cite{mao2020towards}\\
		  \hline
		\end{tabular} \enspace 
\end{center}
\caption{For the isotonic model, the optimal rate for permutation estimation (resp. matrix reconstruction) corresponds to $\cR^*_{\perm}$ (resp. $\cR^*_{\reco})$. For the two other columns, the optimal rates are similarly defined as minimax risk over the corresponding models. The Poly. Time rows correspond to state-of-the art rates achieved by polynomial-time methods. All the rates are given up to polylogarithmic factors in $n$.}
\label{tab:comparison}
\end{table}

Finally, we mention the even more  specific model where the matrix $M$ is bi-isotonic up to a single  permutation $\pi^*$ acting on the rows. This corresponds to the case where $\eta^*$ is known in the previous paragraph~\cite{mao2020towards,pilliat2022optimal,liu2020better}. Equivalently, this also corresponds to our isotonic model~\Cref{eq:model_partial} with the additional assumption that all the  rows are nondecreasing, that is $M_{i,k} \leq M_{i,k+1}$. 
For this model, it is possible to leverage the shape constrains on the rows to build efficient and optimal estimators, this for all $n$, $d$, and $\lambda$ -- see \cite{pilliat2022optimal}.

\subsection{Overview of our techniques}

In this work, we introduce the iterative soft ranking $(\algoPrincipal)$ procedure, which gives an estimator $\hat \pi$ based on the observations. 
Informally, this method iteratively updates a weighted directed graph between experts, where the weight between any two experts quantifies the significance of their comparison. The procedure increases the weights at each step. After it stops, the final estimator is an arbitrary permutation $\hat \pi$ that agrees as well as possible  with the final weighted directed graph.

As mentioned in \cite{liu2020better}, it is hopeless to use only local information between pairs of experts to obtain a rate of order $n^{7/6}$ up to polylogs, and we must exploit global information. Still, we do it in a completely different way of Liu and Moitra~\cite{liu2020better} who were building upon the bi-isotonicity of the matrix. 

One first main ingredient of our procedure is a new dimension reduction technique. At a high level, suppose that we have partially ranked the rows in such a way that, for a given triplet ($P$, $O$, $I$) of subsets of $[n]$, we are already quite confident that experts in $P$ are below those in $I$ and above those in $O$. Relying on the shape constraint of the matrix $M$, it is therefore possible to build a high-probability confidence regions for rows in $P$ based on the rows in $O$ and the rows in $I$. If, for a question $j$, the confidence region is really narrow, this implies that all experts in $P$ take almost the same value on this column. As a consequence, this question is almost irrelevant for further comparing the experts in $P$.
In summary, our dimension reduction technique selects the set of questions for which the confidence region of $P$ is wide enough, and in that way reduces the dimension of the problem while keeping most of the relevant information.

The second main ingredient, once the dimension is reduced, is to use a spectral method to capture some global information shared between experts.
That is why our procedure makes significant use of spectral methods to compute the updates of the weighted graph. Although this spectral scheme already appears  in recent works~\cite{pilliat2022optimal,liu2020better},  those are used here for updating the weight of the comparison graph rather than performing a clustering as in~\cite{liu2020better}. Moreover, the analysis of the spectral step in the partial observation regime $(\lambda \ll 1)$ leads to technical difficulties -- see the discussion in \Cref{subsec:comments}.

Related to the latter problem, we need to establish a new tail bound on sparse rectangular matrices. More specifically, for a rectangular matrix $X$ with centered independent entries that satisfy a Bernstein type condition, we provide a high-probability control of the operator norm of $XX^T -\mathbb{E}[XX^T]$. This result, based on non-commutative matrix Bernstein concentration inequality,  may be  of independent interest e.g. for controlling the spectral properties of a sparse bipartite random graph. We state it in \Cref{sec:concentration}, independently of the rest of the paper.

\subsection{Notation}

Given a vector $u$ and $p\in [1,\infty]$, we write $\|u\|_{p}$ for its $l_p$ norm. For a matrix $A$, $\|A\|_F$ and $\|A\|_{\op}$ stand for its Frobenius and its operator norm. We write $\lfloor x\rfloor$ (resp. $\lceil x\rceil$) for the largest (resp. smallest) integer smaller than (resp. larger than) or equal to $x$. 
Although $M$ stands for an $n\times d$ matrix, we extend it sometimes in an infinite matrix defined for all $i \in \bbN, k \in \{1, \dots, d\}$ by setting $M_{ik} = 0$ when $i \leq 0$ and $M_{ik} = 1$ when $i \geq n+1$. The corresponding infinite matrix $M_{\pi^{*(-1)}}$ which is obtained by permuting the $n$ original rows is still isotonic and takes values in $[0,1]$. We shall often work with submatrices $M(P, Q)$ of $M$ that are restricted to a subset $P\subset [n]$ and $Q\subset [d]$ of rows and columns. If $A$ is any matrix in $\bbR^{P \times Q}$, we write $\overline A$ for the matrix whose rows are all equal to the average row of $A$, namely $\overline A_{ik} = \frac{1}{|P|}\sum_{j \in P}A_{jk}$.

\section{Results}

In this section, we first establish the statistical limit with a lower bound on $\cR^*_{\perm}(n,d,\lambda)$. Then, we state the existence of a polynomial-time estimator that is minimax optimal up to polylog factors. More precisely, we prove that for all integers $n,d$ and $\lambda \in [1/d, 8n^2]$, the optimal rate of permutation estimation $\cR^*_{\perm}$ is of the order of
\begin{equation}\label{eq:rho_perm} 
\rho_{\perm}(n,d, \lambda):=\frac{n^{2/3}\sqrt{d}}{\lambda^{5/6}}\bigwedge n\sqrt{\frac{d}{\lambda}}+\frac{n}{\lambda} \enspace ,
\end{equation}
up to some polylog factors. As a corollary, we then establish that the optimal rate of matrix reconstruction $\cR^*_{\reco}$ is of order
\begin{equation}\label{eq:rho_est} 
\rho_{\reco}(n,d, \lambda):=\frac{n^{1/3}d}{\lambda^{2/3}}+\frac{n}{\lambda} \enspace ,
\end{equation}
up to polylog factors. We therefore establish that these two problems do not exhibit a computational-statistical gap.

\subsection{Minimax lower bound for permutation estimation}\label{subsec:lower_bound}

Assume that $\lambda \in [1/d, 8n^2]$ is fixed and that we are given $N=Poi(\lambda nd)$ independent observations under model \Cref{eq:model_partial}. Namely, we observe $(x_t,y_t)_{t=1, \dots, N}$ where $x_t$ is sampled uniformly in $[n]\times [d]$ and $y_t=M_{x_t}+ \varepsilon_t$ conditionally to $x_t$. The following theorem states that $\rho_{\perm}$ is a lower bound on the maximum risk of permutation estimation for all $n,d, \lambda \in [1/d, 8n^2]$, up to some numerical constant.
\begin{theorem}\label{th:lower_bound_poisson}
	There exists a universal constant $c>0$ such that, for any $n\geq 2$,  $d\geq 1$, and $\lambda \in [1/d, 8n^2]$, we have
	 \begin{align}
		\cR^*_{\perm}(n,d,\lambda)
	\geq  c\rho_{\perm}(n,d,\lambda)\ . \label{eq:lower_bound_partia_observation}
	\end{align}
\end{theorem}

In the proof, we show a slightly stronger result that also covers the cases $\lambda < 1/d$ and $\lambda > 8n^2$, where $\cR^*_{\perm}(n,d,\lambda)$ is in fact respectively lower bounded by a quantity of order $nd$ and $n \sqrt{d}/\lambda$. For the sake of readability, we chose to omit these arguably less interesting cases in the statement of \Cref{th:lower_bound_poisson} and of \Cref{th:UB}.

\subsection{Optimal  permutation estimation}\label{sec:minimax_tpper}

Let us fix a quantity $\delta\in (0,1)$ that will correspond to a small probability. We need to introduce some notation. We write 
\begin{equation}\label{eq:definition_phil1}
	\phi_{\lone} = 10^4\log\left(\frac{10^2nd}{\delta}\right) \enspace .
\end{equation}
Our procedure depends on a sequence of tuning parameters. For this reason, we introduce a subset $\Gamma \subset \bbR^+$, henceforth called a grid. The grid $\Gamma$ is said to be valid
if it contains a sequence $\gamma_0 \geq \dots \geq \gamma_{2\floor{\log_2(n)}+2}$ of length $2\floor{\log_2(n)} + 3$ such that that for all $u$,
\begin{equation}\label{eq:valid_gamma}
	\gamma_{u} - \gamma_{u+1} \geq \gamma_{2\floor{\log_2(n)}+2} + \phi_{\lone} \spaceAnd \gamma_{2\floor{\log_2(n)}+2} \geq \phi_{\lone}\enspace .
\end{equation}
In light of this definition, we could simply choose the valid sequence $\Gamma=\{\phi_{\lone}, 2\phi_{\lone},\ldots, (2\floor{\log_2(n)}+3)\phi_{\lone}\}$ with a corresponding $\gamma_0$ that is polylogarithmic. Still, for practical purpose, we consider general grids; examples of such gris are discussed in more details in \Cref{subsect:valid_grid}.

For any valid subset $\Gamma$, we define $\bar \gamma$ as the smallest possible value of $\gamma_0$ over all sequences that satisfy~\eqref{eq:valid_gamma}.
\begin{equation}\label{eq:def_gamma}
	\bar \gamma = \min \{\gamma ~:~ \exists (\gamma_u) \text{ satisfying \Cref{eq:valid_gamma} s.t. } \gamma_{0} = \gamma\} \enspace .
\end{equation}

Our main procedure $\algoPrincipal$, for iterative soft ranking, will be described in detail in \Cref{sec:algo_sketch}. The only tuning parameters are the the number of steps $T$  and the valid grid $\Gamma$. 

\begin{theorem}\label{th:UB}
	There exists $C>0$ such that the following holds. Let $\lambda \in [1/d, 8n^2]$ and $\delta>0$. Assume that $\Gamma$ is a valid grid and that $T \geq 4\bar \gamma^6$ with $\bar\gamma$ defined in \Cref{eq:def_gamma}. For any permutation $\pi^*\in \Pi_n$ and any matrix $M$ such that $M_{\pi^{*-1}}\in \mathbb{C}_{\iso}$, the estimator $\hat \pi$ from Algorithm $\algoPrincipal(T, \Gamma)$ defined in the next section satisfies
	$$ \|M_{\hat \pi^{-1}} - M_{\pi^{*-1}} \|_F^2 \leq CT\bar\gamma^{6}\rho_{\perm}(n,d,\lambda) \enspace ,$$
	with probability at least $1- 10T\delta$.
\end{theorem}
In particular, if we suitably choose $\Gamma$ (as discussed above) and $T= 4\lceil \bar \gamma^6 \rceil$ and $\delta= 1/(nd)^2$, we deduce from 
\Cref{th:UB} that 
\begin{equation*}
	\cR^*_{\perm}(n,d,\lambda) \leq C' \log^{C'}(nd) \rho_{\perm}(n,d,\lambda) \enspace ,
\end{equation*}
for some numerical constant $C'>0$.
In the case where $\lambda = n^{o(1)}$ and $n=d$, this bound achieves the order of magnitude $n^{7/6}$, which aligns with the result presented in Theorem 2 of Liu and Moitra \cite{liu2020better}. However, it is important to note that the analysis made in \cite{liu2020better} focuses on the statistically easier bi-isotonic model, and their procedure heavily relies on the isotonicity structure imposed on the questions. 

\subsection{Optimal reconstruction of the matrix $M$}\label{ss:recM}

We now turn to the problem of estimating the signal matrix $M$. Obviously, the reconstruction of the matrix $M$ from the observation of model in\Cref{eq:model_partial} is at least as hard as if we knew the permutation $\pi^*$. In this favorable situation, estimating $M$ amounts to estimating $d$ isotonic vectors from partial and noisy observations $Y_{ik} = \frac{1}{\lambda}\sum_ty_t\1_{x_t=(ik)}$. The isotonic regression problem is already well understood, and we state the following lower bound without proof since it directly follows from \cite{mao2020towards} (see in particular Theorem 3.1 therein). We recall that $\rho_{\reco}(n,d,\lambda)$ is defined in~\eqref{eq:rho_est}.

\begin{proposition}\label{prop:LB_reconstruction}
	There exists a universal constant $c>0$ such that, for any $n\geq 2$, any $d\geq 1$, and any $\lambda >0$, we have
	 \begin{align}
		\cR^*_{\reco}(n,d,\lambda)
	\geq  c\rho_{\reco}(n,d,\lambda)\ . \label{eq:lower_bound_partia_observation}
	\end{align}
\end{proposition}

In particular, since $\rho_{\perm}(n,d,\lambda) \ll \rho_{\reco}(n,d,\lambda)$ in many regimes in $n$, $d$, $\lambda$, this proposition implies that the reconstruction of a permuted isotonic matrix is harder than the estimation of the permutation, namely that $\cR^*_{\perm} \ll \cR^*_{\reco}$.

To build an optimal estimator of $M$, we compute the estimated permutation $\hat \pi$ of \Cref{th:UB} and estimate an isotononic matrix based on this ordering. This approach is similar to what is done in~\cite{mao2020towards,pilliat2022optimal}, for related problems where a bi-isotonic assumption is done.
For simplicity, set  the tuning parameters $T$, $\Gamma$ for \Cref{alg:principal} so that $T = 4\ceil{\overline \gamma^{6}}$ and $\bar \gamma^6 \leq C'\log^{C'}(nd/\delta)$. We  split  the samples $y_t$ defined in \Cref{eq:model_partial} into two independent sequences of samples $(y_t^{(1)})$, $(y_t^{(2)})$. First, we compute the estimator $\hat \pi$ of $\pi^*$ with the first sub-samples $(y_t^{(1)})$. Then, we define $\hat M_{\iso}$ as the projection of $Y^{(2)}_{\hat \pi}$ onto the convex set of isotonic matrices, where $Y^{(2)}$ is the matrix defined by $Y^{(2)}_{ik}= \frac{1}{\lambda}\sum_ty_t^{(2)}\1_{x_t^{(2)}=(i,k)}$. More precisely, set
\begin{equation*}
	\hat M_{\iso} = \argmin_{\tilde M \in \bbC_{\iso}} \| \tilde M - Y^{(2)}_{\hat \pi^{-1}}\|_2^2 \enspace .
\end{equation*}

The following corollary controls the risk of $\hat M_{\iso}$.
\begin{corollary}\label{cor:reconstruction_UB}
	Assume that $\lambda \in [1/d, 8n^2]$. There exists a universal constant $C''$ such that the following holds for any permutation $\pi^*\in \Pi_n$ and any matrix $M\in \bbC_{\iso}$.
	\begin{equation*}
\bbE[\|(\hat M_{\iso})_{\hat \pi} - M\|_F^2] \leq C'' \log^{C''}(nd)\rho_{\reco}(n,d,\lambda) \enspace .
	\end{equation*}
\end{corollary}
As a consequence, the polynomial-time estimator $\hat M_{\iso}$ achieves the optimal risk for all values of $n$ and $d$. For $\lambda = 1$, the optimal risk $\rho_{\reco}(n,d,1)$ is of the order of $n^{1/3}d + n$. 
In particular, our risk bound strictly improves over the one of Flammarion et al.~\cite{flammarion2019optimal} - e.g.~their procedure achieves the estimation error $n\sqrt{d}$ for $n\geq d^{1/3}$. Their slower convergence rates are mainly due to the fact that their estimator of the permutation $\pi^*$ is suboptimal in this regime. 


\subsection{Polynomial-time reconstruction in the bi-isotonic model}\label{subsec:reco_biso}

We now turn our attention to the problem of estimating the matrix $M$ when $M$ satisfies the additional assumption of being bi-isotonic up to unknown permutations $\pi^*$ and $\eta^*$ of its rows and columns respectively. In other words, the matrix $M_{\pi^{*-1} \eta^{*-1}}$ has non-decreasing entries. As explained in the introduction, this model has attracted a lot of attention in the last decade and encompasses the SST model for tournament problems. 

To simplify the exposition, we focus in this section on the case $n=d$ and $\lambda \in [\tfrac{1}{n}, 1]$. Since the bi-isotonic model is a specific case of the isotonic model, we could rely on the estimator $\widehat{M}_{\iso}$ introduced in the previous subsection. In fact, we can improve this estimation rate by relying on the bi-isotonicity of the matrix $M_{\pi^{*-1} \eta^{*-1}}$.

As previously, we choose the tuning parameters of \Cref{alg:principal} in such a way that $T =4 \ceil{\overline \gamma^{6}}$ and $\bar \gamma^6 \leq C'\log^{C'}(nd/\delta)$. Then, we use the following procedure:
\begin{enumerate}
	\item Subsample the data into $3$ independent samples $(y^{(1)}_t)$, $(y^{(2)}_t)$, $(y^{(3)}_t)$.
	\item Run our procedure \Cref{alg:principal} to obtain an estimator $\hat \pi$ of the permutation $\pi^*$ of the rows, using the first sample.
	\item Run again \Cref{alg:principal} to obtain an estimator $\hat \eta$ of the permutation $\eta^*$ of the columns, using the second sample.
	\item Compute the least-square estimator $\hat M_{\biso} = \argmin_{\tilde M \in \bbC_{\biso}} \| \tilde M - Y^{(3)}_{\hat \pi^{-1} \hat \eta^{-1}}\|_2^2$, where $\bbC_{\biso}$ is the set of all bi-isotonic matrices with entries in $[0,1]$ and $Y^{(3)}_{ik}= \frac{1}{\lambda}\sum_ty_t^{(3)}\1_{x_t^{(3)}=(i,k)}$.
\end{enumerate}

The following corollary states that $\hat M_{\biso}$ achieves a reconstruction rate of order $n^{7/6}\lambda^{-5/6}$ in the bi-isotonic model.
\begin{corollary}\label{cor:ub_reco_biso} Assume that $\lambda \in [1/n, 8n^2]$.
	There exists a universal constant $C''$ such that 
	\begin{equation*}
		\sup_{\stackrel{\pi^*, \eta^* \in \Pi_n}{M:\,  M_{\pi^{*-1} \eta^{*-1}}\in \mathbb{C}_{\mathrm{biso}}}} \E\left[\|(\hat M_{\biso})_{\hat \pi \hat \eta}- M\|_F^2\right] \leq C''\log^{C''}(n)n^{7/6}\lambda^{-5/6} \enspace .
		\end{equation*}
\end{corollary}

Here, we have fixed $n=d$ to simplify the exposition but we could extend the analysis to general $n$ and $d$. Our risk bound improves over the rate $n^{5/4}\lambda^{-3/4}$ of Mao et al.~\cite{mao2020towards}. In~\cite{liu2020better}, Liu and Moitra have introduced a procedure achieving the rate $n^{7/6}$ in the specific case where $\lambda = n^{o(1)}$. In some way, our procedure generalizes their results for general $\lambda$, while being applicable to the more general isotonic models.

Still, we recall that the optimal risk (without computational constraints) for estimating the matrix 
$M$ is of the order $n/\lambda$ -- see e.g.~\cite{shah2016stochastically,mao2020towards}. This remains an open problem to establish the existence of a computational-statistical gap or to construct a polynomial-time procedure achieving this risk on SST and bi-isotonic models. 

\section{Description of the $\algoPrincipal$ procedure}\label{sec:algo_sketch}

\subsection{Weighted directed graph $\cW$ and estimator $\hat \pi$}\label{subsec:weigthed_graph_permutation}

Our approach involves the iterative construction of a weighted directed graph $\cW$, represented by an antisymmetric matrix in $\bbR^{n \times n}$. More formally, for any experts $i$, $j$ in $[n]$, we have $\cW(i, j) = -\cW(j, i)$. In a nutshell, $\cW(i, j)$ quantifies our evidence of the comparisons between expert $i$ and expert $j$. If $\cW(i, j)$ is large and positive (resp. negative), we are confident that the expert $i$ is above (below) the expert $j$. Most of the procedure is dedicated to the construction of $\cW$. Before this, let us explain how we deduce our estimator $\hat \pi$ from $\cW$.

For a given weighted directed graph $\cW$, we define its corresponding directed graph at threshold $\gamma > 0$ as
\begin{equation}\label{eq:graph_from_weighted}
	\cG(\cW, \gamma) = \{(i,j) \in [n]^2 ~:~ \cW(i,j) > \gamma \} \enspace .
\end{equation}
For any thresholds $\gamma < \gamma'$, it holds that $\cG(\cW, \gamma) \subset \cG(\cW, \gamma')$. In other words, the function $\gamma \to \cG(\cW, \gamma)$ is nondecreasing. When $\gamma \geq \max_{i,j}|\cW(i,j)|$, $\cG(\cW, \gamma) = \emptyset$ is the trivial graph with no edges.
Let $\hat \gamma$ be the smallest threshold $\gamma$ such that $\cG(\cW, \gamma)$ is a directed acyclic graph ($\DAG$). By monotonicity, $\cG(\cW, \hat \gamma)$ is also the largest $\DAG$ among $\{\cG(\cW, \gamma), \gamma \geq \hat \gamma \}$. We then build the estimator $\hat \pi$ by picking any  permutation that is consistent with the graph $\widehat \cG := \cG(\cW, \hat \gamma)$, that is if $(i,j) \in \widehat \cG \cap [n]^2$ then $\hat \pi(i) \geq \hat \pi(j)$.
To put it another way, the general idea of our procedure can be summarized into these three components:
\begin{enumerate}
	\item Construct a weighted directed graph $\cW$ between the experts.
	\item Compute the largest directed acyclic graph $\widehat \cG$ of $\cW$.
	\item Take any arbitrary permutation $\hat \pi$ that is consistent with $\widehat \cG$.
\end{enumerate}
The construction of $\cW$ is at the core of this paper, and the computation of $\widehat \cG$ and $\hat \pi$ will be discussed in \Cref{subsec:further_discussion}. Still, we already point out that the third point can be dealt in polynomial time using Mirsky's algorithm~\cite{mirsky1971dual}.

\subsection{Construction of $\cW$ with $\algoPrincipal$}

\subsubsection{Description of the subsampling}\label{subsec:subsample}
Let us now describe the construction of the weighted directed graph $\cW$. Let $T \geq 1$ be an arbitrary integer, representing the number of steps of our procedure. In what follows, we explain how we subsample the data from \Cref{eq:model_partial} into $5T$ independent matrices $(Y^{(s)})_{s=1\dots5T}$.
Recall that we are given $N$ observations $(x_t, y_t)$, where $N$ follows a Poisson distribution $\cP(\lambda nd)$. Let us divide the observations into $5T$ batches $(N^{(s)})_{s = 0, \dots, 5T-1}$, aggregated into matrices of averaged observations $Y^{(s)}$. To that end, we let $S_u$ be i.i.d. uniform random variables in $\{0, \dots, 5T - 1\}$ representing a random batch for observation $u$ and we define
\begin{equation}\label{eq:batches}
	N^{(s)} = \{u \in \{1 \dots, N\} ~:~ S_u = s \} \spaceAnd Y^{(s)}_{ik} = \sum_{t \in N^{(s)}} \tfrac{y_t}{\br^{(s)}_{ik}\lor 1}\1\{x_t = (i,k)\} \enspace ,
\end{equation}
where, for any $(i,k) \in [n] \times [d]$,~ $\br^{(s)}_{ik} = \sum_{t \in N^{(s)}}\1\{x_{t} = (i,k)\}$ is the number of times the coefficient position $(i,k)$ is observed in batch $s$. $Y^{(s)}_{ik}$ is equal to $0$ if $(i,k)$ is not observed in batch $s$ and it is equal to the average of the observations $y_t$ for which $x_t = (i,k)$ otherwise. 
We also define the mask matrix $B^{(s)}$ as being equal to $0$ at location $(i,k)$ if the value is missing from batch $s$, and to $1$ otherwise.  
\begin{equation}\label{eq:def_matrix_B}
	B^{(s)}_{ik} = \1\{\br^{(s)}_{ik} \geq 1\}\enspace .
\end{equation}
Define $\lambda_0 = \lambda/5T$.
In our sampling scheme, where the data is divided into $5T$ samples, each coefficient $B^{(s)}_{ik}$ has a probability of $1 - e^{-\lambda_0}$ of being equal to one.
It is worth mentioning that a different subsampling scheme was performed in \cite{pilliat2022optimal}, consisting in aggregating consecutive columns. However, such a scheme is not applicable in our case as we do not assume the rows of $M$ to be nondecreasing, unlike in \cite{pilliat2022optimal}.

\subsubsection{Neighborhoods in comparison graphs}

At each step $t = 0, \dots, T-1$ of the procedure, we aim to enrich our knowledge of the order of the experts, which we formally do by nondecreasing the weights of $\cW$ in absolute value. At $T=0$, we start with the weights $\cW_{ij}$ all being equal to zero. A meaningful update of $\cW$ around a reference expert $i$ can be done when we restrict ourselves to experts that are in a neighborhood of $i$. Broadly speaking, a neighborhood of $i$ is a set made of all the experts $j$ that are not comparable to $i$ with respect to a given partial order.

More precisely, for any directed graph $\cG$ and any experts $i,j \in \{1, \dots, n \}$, we say that $i$ and $j$ are $\cG$-comparable if there is a path from $i$ to $j$ or from $j$ to $i$ in $\cG$.  The neighborhood $\cN(\cG, i)$ of $i$ in $\cG$ can then naturally be defined as the set of experts $j$ that are not $\cG$-comparable with $i$. Equipped with the concept of neighborhood, our overall strategy involves iterating over all possible thresholds $\gamma \in \Gamma$ such that $\cG(\cW, \gamma)$ is acyclic, as well as all possible experts $i$. At each iteration, we apply the soft local ranking procedure \Cref{alg:refine_locally_sketch} described in the next subsection. \Cref{alg:refine_locally_sketch}  updates the weights between $i$ and any expert $j$ in the neighborhood $\cN(\cG(\cW, \gamma), i)$ of $i$.  Our approach can be summarized as follows:
\begin{enumerate}
	\item Subsample the data - see \Cref{subsec:subsample}.
	\item Initialize $\cW$ to be the directed graph with all weights set to $0$.
	\item For all $t= 0, \dots, T-1$ and $\gamma \in \Gamma$ such that $\cG(\cW, \gamma)$ is acyclic and all $i \in [n]$, update  $\cW$ with the soft local ranking procedure \Cref{alg:refine_locally_sketch}.
\end{enumerate}

\medskip

\begin{algorithm}[H]
	\caption{\label{alg:principal} $\algoPrincipal(T, \Gamma)$}
	\begin{algorithmic}[1]
		\Require $N$ and observations $(x_t, y_t)_{t = 1, \dots, N}$ according to \Cref{eq:model_partial}, a number of steps $T$ and a valid grid $\Gamma$ as in \Cref{eq:valid_gamma}
		\Ensure A weighted graph $\cW$ and an estimator $\hat{\pi}$
		\State Aggregate the observation into $5T$ matrices of observation $(Y^{(s)})$ as in \Cref{eq:batches}
		\State Initialize $\cW(i,j) = 0$ for all $(i,j) \in [n]^2$, and $\hat \gamma = 0$
		\For{$t = 0, \dots, T-1$}
		\For{$\gamma \in \Gamma \cap [\hat \gamma, +\infty)$}
		\State Compute $\cG = \cG(\cW, \gamma)$ the directed graph at threshold $\gamma$ of $\cW$ as in \Cref{eq:graph_from_weighted} and set $P = \cN(\cG, i)$.\label{line:compute_directed_1} 
		\State Take $5$ samples $\cY = (Y^{(5t)}, \dots, Y^{(5t+4)})$
		\For{$i \in [n]$} \label{line:loop}
		\State Apply $\algoRefineLocally(\cY, \cW, \gamma, i, \cG, P)$ to update $\cW$ \label{line:soft_cluster}
		\EndFor
		\EndFor
		\State Set $\hat \gamma$ as the smallest $\gamma$ such that $\cG(\cW, \gamma)$ is acyclic \label{line:acyclicite}
		\EndFor
		\State Set $\widehat \cG = \cG(\cW, \hat \gamma)$ be the largest acyclic $\DAG$ (see \Cref{eq:graph_from_weighted})\label{line:largest_dag}
		\State Set $\hat{\pi}$ to be any arbitrary permutation that is consistent with $\widehat \cG$\label{line:dag_to_estimator}
		\State \Return{$\cW$ and $\hat{\pi}$}
	\end{algorithmic}
\end{algorithm}

The main \Cref{line:soft_cluster} of \Cref{alg:principal} aims to provide a soft ranking of the neighborhood $P$ of $i$ by setting positive (resp. negative) weights $\cW_{ij}$ to experts $j\in P$ that are significantly below (resp. above) $i$. \Cref{line:acyclicite} together with restricting $\gamma \geq \hat \gamma$ simply guarantees that all the considered graph $\cG$ are acyclic. Finally, Lines \ref{line:largest_dag} and \ref{line:dag_to_estimator} simply correspond to the construction of the final permutation, described in the second and third points of \Cref{subsec:weigthed_graph_permutation}.
\subsection{Description of the updating procedure}\label{subsec:description_procedure_trisection}

\subsubsection{Local weighted sums}
Let us describe the process of updating a given weighted graph $\cW$, which will be used twice at each call of the soft local ranking \Cref{alg:refine_locally_sketch}. Let us fix a weighted graph $\cW$, an element $s \in \{0, \dots, 5T-1\}$ and $Y:=Y^{(s)}$ the matrix defined in \Cref{eq:batches}. We also let $i \in [n]$ be an arbitrary expert corresponding to \Cref{line:loop} of \Cref{alg:principal}, and $\gamma$ be any threshold in the grid $\Gamma$. We write $P := \cN(\cG(\cW, \gamma),i) \subset [n]$ for the neighborhood of $i$ in $\cG(\cW, \gamma)$, echoing the notation of the sets that are trisected in \cite{pilliat2022optimal}. 

Since the matrix $M$ is, up to a row-permutation, a column-wise isotonic matrix, it follows that, if the expert $i$ is above $j$, then for any vector $w\in \mathbb{R}_+^d$, we have $\sum_{k=1}^d w_{ik}M_{ik}\geq \sum_{k=1}^d w_{jk}M_{jk}$. As a consequence, the crux of the algorithm is to find suitable data-driven weights $w$ that allow to discriminate the experts. As explained in the introduction, earlier works focused on uniform weights $w=\1_{[d]}$~\cite{shah2016stochastically} which, unfortunately leads to suboptimal results.
Before discussing the choice of the weights $w$ in the following subsections, let us first formalize how we leverage on $w$ to compare the experts and update the graph $\cW$.

Given a subset $Q \subset [d]$ of columns and a non-zero vector $w \in \bbR_+^{Q}$, we first check whether the following condition is satisfied:
\begin{equation}\label{eq:condition_w}
	\lambda_0\|w\|_2^2 \geq \|w\|^2_{\infty} \enspace ,
\end{equation}
where we recall that $\lambda_0 = \lambda/5T$.
This condition is always verified when $\lambda_0 \geq 1$, and it is equivalent to $\lambda_0 |Q| \geq 1$ when $w = \1_Q$. Condition~\eqref{eq:condition_w} ensures that $w$ is not too sparse which could be harmful when many observations are lacking ($\lambda_0$  small).

If this condition is not satisfied, then we leave the weights of $\cW$ unchanged. Otherwise, we define the $(Y, P, w)$-updating weights $\cU := \cU(Y, P, w)$ around $i$ as
\begin{equation}\label{eq:updating_edges}
	\cU_{i j} = \frac{1}{\sqrt{\frac{1}{\lambda_0} \land \lambda_0}} \cdot \proscal<Y_{i \cdot} - Y_{j \cdot}, \frac{w}{\|w\|_2}> \enspace ,\\
\end{equation}
where, for all $w' \in \bbR^Q$ and $a \in \bbR^{d}$, we write $\proscal<a, w'>= \sum_{k \in Q}a_kw'_k$.
We can then update the weighted directed graph around $i$ by setting, for all $i \in P$ such that $|\cU_{i j}| \geq |\cW_{i j}|$,
\begin{equation}\label{eq:update_W}
	\cW_{i j} = \cU_{i j} \spaceAnd \cW_{j i} = -\cU_{i j}\enspace .
\end{equation}
As explained above, if we replace $Y_{i \cdot}$ and $Y_{j \cdot}$ by $M_{i \cdot}$ and $M_{j \cdot}$ respectively in~\eqref{eq:updating_edges}, then the corresponding value of the statistic is non-negative if expert $i$ is above $j$. Hence, a large value for $\cU_{i j}$ provides evidence that $i$ is above $j$.

Computing $\cU(Y, P, w)$ for suitable directions $w$ is the basic brick or our procedure, since it is through the update \Cref{eq:update_W} that we iteratively increase the weights of $\cW$. This update shares some similarities to the pivoting algorithm introduced in \cite{liu2020better} and also used in \cite{pilliat2022optimal}, in the sense that while we are fixing an arbitrary reference expert $i$ to compute pairwise comparisons, they fix a set $P$ and compute a pivot expert $i_0$ that would correspond to a quantile of the set $\{\proscal<Y_{j \cdot}, \tfrac{w}{\|w\|_2}>, j \in P\}$ in the case $\lambda_0 = 1$.

Note that the orientation of a given weighted edge $(i,j)$ can change during the procedure if it turns out that $|\cU_{i j}| \geq |\cW_{i j}|$ and that $\cU_{i j}\cW_{i j} \leq 0$. This simply means that if the direction $w$ leads to a more significant weight between some experts $i$ and $j$, then we are more confident to use the vector $w$ and to revise the order between $i$ and $j$.

For $Q \subset [d]$, choosing $w = \1_Q$ in \Cref{eq:updating_edges} amounts to compute the average of the observations over all questions in $Q$. 
We now explain in the main sections how we iteratively build adaptive weights $w$ that allow to improve over the naive global average given by $w = \1_{[d]}$.

\subsubsection{Definitions of a rank in a $\DAG$}

We first introduce a few definitions on directed acyclic graphs $\cG$, which we formally define as a set of directed edges $(i, j) \in [n]^2$ for which there is no cycle. We denote $\pth(i,j) = \{(k_1, \dots, k_L) ~:~ L > 0 \mbox{ and } (i,k_1), \dots, (k_L, j) \in \cG\}$ as the set of all possible paths from $i$ to $j$, and we write $|s|$ for the length of any path $s$. We say that $i$ and $j$ are $\cG$-comparable if $\pth(i,j) \cup \pth(j,i) \neq \emptyset$, and we write $\cN(i, \cG)$ for the set of all experts that are not $\cG$-comparable with $i$. If $i$, $j$ are $\cG$-comparable, it either holds that $\pth(i,j) = \emptyset$ or $\pth(j, i) = \emptyset$. We say in the first case that $i$ is $\cG$-below $j$ and that $i$ is $\cG$-above $j$ in the second case. we also define the relative rank from $i$ according to $\cG$ as the length of the longest path in $\cG$ from $i$ to $j$, or minus the longest past from $j$ to $i$ depending on wether $i$ is $\cG$-above or $\cG$-below $j$:

\begin{equation}\label{eq:rank}
	\begin{aligned}
	\rk_{\cG, i}(j) &= \max \{|s| ~:~ s \in \pth(i,j) \} - \max \{|s| ~:~ s \in \pth(j,i)\} \enspace .
	\end{aligned}
\end{equation}
Here, we use the convention $\max \emptyset = 0$. With this definition, the neighborhood of a given expert $i$ is equal to the set of experts whose relative rank is equal to $0$, that is $\cN(\cG, i) = \rk_{\cG, i}^{-1}(0)$.  Moreover, an expert $j\in [n]$ is $\cG$-above (resp. $\cG$-below) $i$ if and only if $\rk_{\cG,i}(j) \geq 1$ (resp. $\rk_{\cG,i}(j) \leq -1)$. Although $\cG$ stands for a finite set of edges with endpoints in $[n]$, we extend it to a set of edges with endpoints in $\bbZ^2$ by putting in $\cG$ every $(i,j) \in \bbZ^2$ such that $i > j$ and $j \leq 0$ or $i \geq n + 1$.


\subsubsection{Description of the soft local ranking algorithm}\label{sec:soft_local}

To update the weighted directed graph $\cW$ in \Cref{line:soft_cluster} of \Cref{alg:principal}, we apply the soft local ranking procedure $\algoRefineLocally$ to all experts $i \in [n]$ and all thresholds $\gamma$. To define our soft local ranking procedure, let us fix $\cW$, an expert $i$ and a threshold $\gamma$ such that $\cG(\cW, \gamma)$ is acyclic. As a shorthand, we write $\cG$ and $P$ respectively for the thresholded graph $\cG(\cW, \gamma)$ and the neighborhood $\cN(\cG,i)$ of $i$ in $\cG$.

We write $\cD$ for the set of all dyadic numbers: $\cD = \{ 2^k ~:~ k \in \bbZ \}$ and we define the set $\cH = \cD \cap \left[\frac{1}{nd},1\right]$.
We denote $\overline y(P)$ as the mean of the vectors $Y_{j \cdot}$ over all $j \in P$, that is $\overline y_k(P) = \tfrac{1}{|P|}\sum_{j \in P} Y_{j k}$, for any $k \in [d]$.
$\algoRefineLocally$ 
relies on the following steps repeated over all height $h \in \cH$.  It is also described in Algorithm~\ref{alg:refine_locally_sketch}.
\begin{enumerate}
    \item {\bf Dimension reduction.} Using the first sample $Y^{(1)}$, we first reduce the dimension by selecting a subset $\widehat{Q}^h\subset [d]$ corresponding to wide confidence regions. Recall that $\rk_{\cG, i}$ is the relative rank to $i$ defined in \Cref{eq:rank}. For any $a > 0$, define the sets $\cN_a := \cN_a(\cG, i)$ (resp. $\cN_{-a}:=\cN_{-a}(\cG, i)$) of experts $j$ which are $\cG$-above (resp. $\cG$-below) all the experts of $P$ and whose relative rank to any $i' \in P$ is at most $a$ in absolute value:
	\begin{equation}\label{eq:neighborhood}
		\cN_a = \bigcap_{i' \in P} \rk_{\cG, i'}^{-1}([1, a]) \spaceAnd \cN_{-a} = \bigcap_{i' \in P} \rk_{\cG, i'}^{-1}([- 1, - a]) \enspace .
	\end{equation}

	Secondly, we define for any question $k \in [d]$ and $a \geq 1$ the width statistic $\widehat \bDelta_k$ as the difference between the mean of the experts in $\cN_a$ and the mean of the experts in $\cN_{-a}$. Then, $\hat a_{k}$ is set to be the first value of $a \geq 1$ such that any $a' \geq a$ has a corresponding width statistic of at least $(\lambda_0 \land 1)h$:
	\begin{equation}\label{eq:stat_delta}
		\widehat \bDelta_{k}(a) = \overline y_{k}(\cN_a) - \overline y_{k}(\cN_{-a}) \spaceAnd \hat a_k(h) = \max \left\{a \geq 1 ~:~ \frac{1}{\lambda_0 \land 1}\widehat \bDelta_{k}(a) < h \right\} +1\enspace .
	\end{equation}
	Finally, we define $\widehat Q^{h} := \widehat Q^{h}(\cG, i)$ as the set of indices $k$ such that $\hat a_k(h)$ is relatively small.
	\begin{equation}\label{eq:set_Q}
		\widehat Q^{h} = \{k \in [d] ~:~ |\cN_{\hat a_k(h)}|\land |\cN_{-\hat a_k(h)}|\leq \frac{1}{\lambda_0 h^2} \} \enspace .
	\end{equation}
    Intuitively, if the experts above and below $i$ vary by more than $h$ on a specific question $k$, then this question should belong to $\widehat Q^{h}$. Conversely, if the experts below and above $i$ are nearly equal on the question $k$, than $\hat a_k(h)$ will be large and $k$ will not be selected in $\widehat Q^{h}$.
    
    \item {\bf Average-based weighted sums.} Still using the first sample $Y^{(1)}$, we examine the corresponding submatrix  $Y^{(1)}(P, \widehat Q)$ restricted to questions in $\widehat Q$. If the row sums of $Y$ are larger than the current edges, we update the weighted edges. More formally, we 
	compute the $(Y^{(1)}, P, \1_{\widehat Q})$-updating weighted edges  $(\cU_{\widehat Q})$ around $i$ as defined in \Cref{eq:updating_edges} and update $\cW$ as in \Cref{eq:update_W}. We then also update $\cG = \cG(\cW, \gamma)$ and $P = \cN(\cG, i)$.
\item {\bf PCA-based weighted sums.} Relying on the samples $Y^{(2)}$, $Y^{(3)}$, $Y^{(4)}$, $Y^{(5)}$, we do a slight abuse of notation and write $Y^{(s)}$ for the restriction of $Y^{(s)}$ to the subset $P, \widehat Q^{h}$ for $s=2,3,4,5$. Ideally, we would get an informative direction $w$ from the largest right singular vector of $\mathbb{E}[Y^{(2)}-\overline{Y}^{(2)}] \in \bbR^{P \times \widehat Q^{h}}$. Indeed, it is known (see the proofs for more details) that the entries of the first right singular vector of an isotonic matrix all share the same sign and are most informative to compare the experts. However, computing directly the empirical right-singular vector of $Y^{(2)}-\overline{Y}^{(2)}$ does not lead to the desired bounds because (i) this matrix is perhaps highly rectangular (ii) the noise is possibly heteroskedastic and (iii) this matrix is perhaps sparse because of the many missing observations when $\lambda_0$ is small. Here, we use a workaround which is reminiscent of that of~\cite{pilliat2022optimal} and discussed later. First, we compute $\hat v$ as a proxy for the first left singular vector of $\mathbb{E}[Y^{(2)}-\overline{Y}^{(2)}]$.
\begin{equation}\label{eq:ACP}
	\hat v := \hat v( P, \widehat Q^h) = \argmax_{v \in \bbR^{P}:~\| v \|_2 \leq 1} \Big[ \|v^T(Y^{(2)} - \overline{Y}^{(2)})\|_2^2 - \frac{1}{2}\| v^T(Y^{(2)} - \overline{Y}^{(2)} - Y^{(3)} + \overline{Y}^{(3)})\|_2^2\Big] \enspace .
\end{equation}
The right-hand side term in~\eqref{eq:ACP} deals with the heteroskedasticity of the noise matrix $E$ in~\Cref{eq:model_0}. 
$\hat v$ in~\eqref{eq:ACP} can be computed efficiently since it corresponds to the leading eigenvector of a symetric matrix. For technical reasons occurring in the sparse observation regime (i.e. when $\lambda_0$ is small), we then threshold the largest absolute values of the coefficients of $\hat v$ at $\sqrt{\lambda_0}$ and define $(\hat v_-)_i = \hat v_i \1\{|\hat v_i| \leq \sqrt{\lambda_0}\}$. After having calculated $\hat v_-$, we consider as in \cite{pilliat2022optimal} the image $\hat z = \hat v_-^T (Y^{(4)}-\overline Y^{(4)}) \in \bbR^{\widehat Q}$  of $\hat v_{-}$. We then threshold the smallest values of $\hat z$ and take the absolute values of the components. Thus, we get $\hat w^+\in \mathbb{R}^{\widehat Q}$ defined by $(\hat w^+)_l= |\hat z_l|\1\{|\hat z_l|\geq \gamma\sqrt{\lambda_0 \land \frac{1}{\lambda_0}}\}$ for any $l  \in \widehat Q$. 

Finally, we consider the last submatrix $Y^{(5)} = Y^{(5)}(P, \widehat Q)$. We apply these weights $\hat w^+$ to compute the row-wise weighted sums of $Y^{(5)}$ and update the weighted edges. More formally, we compute the $(Y^{(5)}, P, \hat w^+)$-updating weighted edges $\cU(Y^{(5)}, P, \hat w)$ around $i$ as defined in \Cref{eq:updating_edges}. We finally update the weighted directed graph $\cW$ with $\cU(Y^{(5)}, P, \hat w^+)$ as in \Cref{eq:update_W}.
\end{enumerate}

\begin{algorithm}[H]
	\caption{$\algoRefineLocally((Y^{(s)})_{s= 1, \dots, 5},\cW, \gamma, i, \cG, P)$ \label{alg:refine_locally_sketch}}
	\begin{algorithmic}[1]
		\Require $6$ samples $(Y^{(s)})_{s=1,\dots,5}$, a weighted directed graph $\cW$, a threshold $\gamma$ such that $\cG(\cW, \gamma)$ is acyclic and an expert $i \in [n]$. $\cG$ and $P$ are shorthands for the thresholded graph $\cG(\cW, \gamma)$ and the neighborhood $\cN(\cG,i)$.
		\Ensure An update of $\cW$
		\Statex
		\For{$h \in \cH$}
		\State \label{line:set_Q}Compute $\widehat Q^{h} := \widehat Q(\cG, i)$ as in \cref{eq:set_Q} using sample $Y^{(1)}$
		\State \label{line:first_stat_l1}Let $\cU_{\widehat Q^{h}}$ be the $(Y^{(1)}, P, \1_{\widehat Q^{h}})$-updating weighted edges around $i$ as in \Cref{eq:updating_edges}, using again sample $Y^{(1)}$ \label{line:update_1}
		\State Update $\cW$ with $\cU(\widehat Q^{h})$ as in \Cref{eq:update_W} and update $\cG = \cG(\cW, \gamma)$, $P = \cN(\cG, i)$
		\State Restrict the samples $(Y^{(s)})_{s=2,\dots,5}$ to $P, \widehat Q^h$ in the following remaining steps
		\State \label{line:ACP}Compute the PCA-like direction $\hat v := \hat v(P, \widehat Q^{h})$ as in~\eqref{eq:ACP} and define $(\hat v_-)_i = \hat v_i \1\{|\hat v_i| \leq \sqrt{\lambda_0}\}$
		\State \label{line:direction}Compute $\hat z = \hat v_-^T (Y^{(4)}-\overline Y^{(4)})$ and define $\hat w^+$ by $(\hat w^+)_l= |\hat z_l|\1\{|\hat z_l|\geq \gamma\sqrt{\lambda_0 \land \frac{1}{\lambda_0}}\}$ for any $l  \in \widehat Q^{h}$
		\State \label{line:second_stat_l1}Let $\cU(Y^{(5)},\hat w^+)$ be the $(Y^{(5)}, P,\hat w^+)$-updating weighted edges around $i$ as in \Cref{eq:updating_edges}
		\State Update $\cW$ with $\cU(Y^{(5)},\hat w^+)$ as in \Cref{eq:update_W} \label{line:update_2}
		\EndFor
	\end{algorithmic}
\end{algorithm}

\subsection{Toy example illustrating \Cref{alg:refine_locally_sketch}}

To understand why the steps described in \Cref{alg:refine_locally_sketch} are relevant, assume that $\pi^* = \mathrm{id}$ and consider the following simple example where $n = 204$, $d = 10$, and where the isotonic matrix $M_{\pi^{*-1}}$ can be decomposed into three blocks of rows as
\begin{equation*}
	M_{\pi^{*-1}}= \alpha + \frac{h}{2} \left(\begin{array}{cccccccccc}
		\0&\0&\tikzmarkin[ver=style cyan]{1}\1&\1&\0&\tikzmarkin[ver=style cyan]{2}\1&\1&\0&\tikzmarkin[ver=style cyan]{3}\1&\1\\
	  \hline
	  0&0&0&1&0&1&0&0&1&1 \\
	  0&0&0&1&0&1&0&0&1&1 \\
	  0&0&0&-1&0&-1&0&0&-1&-1 \\
	  0&0&0&-1&0&-1&0&0&-1&-1\\
	  \hline
	  \0&\0&-\1&-\1\tikzmarkend{1}&\0&-\1&-\1\tikzmarkend{2}&\0&-\1&-\1\tikzmarkend{3}
	\end{array}\right) \enspace .
  \end{equation*}

In the above matrix, $\alpha$ is any number in $(h, 1-h)$, and $\0,\1$ are the columns in $\bbR^{100}$ whose coefficients are respectively all equal to $0$ and $1$. Assume that the statistician already knows that the first and the third blocks are made of experts that are respectively above and below the second block. If $\cW, P,\gamma$ are the parameters fixed in \Cref{alg:refine_locally_sketch}, the three blocks correspond respectively to the subsets $\cN_{1} \cup \cN_2$, $P$ and $\cN_{-1}\cup \cN_{-2}$ in our example. Provided that $\cN_{-2}$ and $\cN_{2}$ are large enough, the set $\widehat Q^{h}$ only keeps  columns corresponding to indices $k$ where $\widehat \bDelta_k(1)$ is large -- those are highlighted in blue.

Then, we can work on the reduced subset $\widehat Q^h$ of columns highlighted in blue. As one may check, $\widehat Q^h$ contains all the relevant columns to decipher the experts in the block $P$. Besides, the expected matrix of observations restricted to the block $P$ and to $\widehat Q^h$ is of rank one:
\begin{equation*}
	\bbE[Y - \overline Y]= \frac{h}{2}\left(\begin{array}{cccccccccc}
	  0&1&1&0&1&1 \\
	  0&1&1&0&1&1 \\
	  0&-1&-1&0&-1&-1 \\
	  0&-1&-1&0&-1&-1\\
	\end{array}\right) \enspace .
  \end{equation*}
In particular, the right singular vector of this matrix is of the form $(0,1,1,0,1,1)$ and provides suitable weights to decipher the two largest experts from the two lowest experts in the above matrix. The PCA-based weighted sums steps above precisely aims at estimating these weights.  

\subsection{Comments on the procedure and relation to the literature}\label{subsec:comments}

Finding confidence regions $\widehat Q$ before computing weighted sums on the corresponding columns is at the core of our procedure. This idea generalizes the RankScore procedure of \cite{flammarion2019optimal} which rather computes averages on the subsets $[d]$ or on the singletons $\{1\}, \dots, \{d\}$. As mentioned in the introduction, only using the subsets of the RankScore method in \cite{flammarion2019optimal} does not allow to reach the optimal rate for permutation estimation or matrix reconstruction.

In \Cref{alg:refine_locally_sketch}, the computation of subsets $\widehat Q^{h}$ is reminiscent of some aspects of the non oblivious trisection procedure used in \cite{pilliat2022optimal} for the bi-isotonic model. In fact, the statistic $\widehat \bDelta_k$ corresponds to the statistic $\widehat \bDelta^{(\mathrm{ext})}_{k,1}$ in \cite{pilliat2022optimal}. Apart from that, the selection of subsets of questions was quite different in~\cite{pilliat2022optimal} as it mostly involved change-point detection ideas as introduced in~\cite{liu2020better}. However, those ideas are irrelevant in our setting because the rows do not exhibit any specific structure in the isotonic model.

The high-level sorting method in \cite{pilliat2022optimal} is based on a hierarchical sorting tree with memory. In contrast, our new algorithm is based on an iterative refinement of a weighted comparison graph. This new algorithm is more natural and benefits from the fact that it is almost free of any tuning parameter. Indeed, at the end of \Cref{alg:principal}, we simply use the threshold $\hat \gamma$ corresponding to the largest acyclic $\hat \cG$ graph in $\cW$. No significant threshold needs to be chosen, since any permutation that is consistent with $\hat \cG$ is also necessarily consistent with $\cW$ thresholded at values larger than $\hat \gamma$.


The spectral step in \cite{pilliat2022optimal} is quite similar to the third step of our procedure described in~\cref{sec:soft_local}, except for the first thresholding of $\hat v$ to obtain $\hat v_-$. In \cite{pilliat2022optimal}, this workaround was not needed mainly because in the bi-isotonic model, it is possible to aggregate sparse observations by merging consecutive columns -- see~\cite{pilliat2022optimal} for further details. This is however not possible here. 

As mentioned in the introduction, Liu and Moitra~\cite{liu2020better} obtain an upper bound of the permutation loss of the order of $n^{7/6}$ for the estimation of two unknown permutations in the case where $M \in \bbR^{n\times n}$ is bi-isotonic. Broadly speaking, their method involves iterating a clustering method called block-sorting over groups of rows or columns that are close with each other. Using this sorting method based on block-sorting, their whole approach alternates between row sorting and column sorting for a subpolynomial number of time.
Besides, their procedure makes heavily use of bi-isotonicity of the matrix. 
It turns out that \Cref{alg:refine_locally_sketch} reaches the same rate in this bi-isotonic model by running only once on the rows, and once on the columns, as described  in \Cref{subsec:reco_biso}. Otherwise said, if the problem is to estimate only $\pi^*$ in the bi-isotonic model, we proved that only the isotonicity of the columns is necessary to achieve the state-of-the-art polynomial-time upper bound of order $n^{7/6}$.



\subsection{Examples of valid grids $\Gamma$}\label{subsect:valid_grid}
Remark that the simple set $\{(u+1)\cdot\phi_{\lone}, u \in \{0, \dots, 2\floor{\log_2(n)} + 2\}\}$ is a valid grid of logarithmic size with $\bar \gamma \leq (2\log_2(n) + 3)\phi_{\lone}$. This set is the smallest valid grid achieving the smallest possible value of $\bar \gamma$. However, it depends on the  quantity $\phi_{\lone}$ which is perhaps a bit pessimistic in practice.

An other choice can be to take $\bbR^+$ itself, albeit infinite. Indeed, the set $\{\cG(\cW, \gamma), \gamma \geq 0 \}$ is made of at most $n^2$ possible directed graphs for any $\cW$ during the whole procedure. Choosing $\bbR^+$ is convenient since it does not depend on the constants in $\phi_{\lone}$ that are likely to be overestimated. The drawback of choosing $\bbR^+$ though is that the number of tested $\gamma$ in \Cref{alg:refine_locally_sketch} becomes quadratic in $n$.

Finally, a good compromise is to take the set 
$\{(1+ \frac{1}{\log_2(n)})^{u'}, ~u' \in \bbZ \}$. It is easy to check that it contains a sequence satisfying \Cref{eq:valid_gamma} whose length is at least $2\floor{\log_2(n)}+3$ and whose maximum $\bar \gamma$ is a polylogarithmic function in $nd/\delta$.

\subsection{Discussion on the computation of $\widehat \cG$ and $\hat \pi$}\label{subsec:further_discussion}

Once we have suitable weighted graph $\cW$, it remains to construct the permutation $\hat \pi$, as in the second and third point of \Cref{subsec:weigthed_graph_permutation}.

For the second point, checking that a given directed graph is acyclic can be done through depth first search with a computational complexity less than $n$, so that computing $\hat \gamma$ can be done with less than $|\Gamma|n$ operations. As discussed in \Cref{subsec:description_procedure_trisection}, it is possible to choose $\Gamma$ to be of size of order less than $\log(n)$. If $\Gamma$ is bounded and is such that any different thresholds $\gamma, \gamma'$ in $\Gamma$ satisfy $|\gamma - \gamma'|\geq \eta$ for some $\eta > 0$, the computation of $\hat \gamma$ can always be done with complexity of order less than $n\log(\max (\Gamma)/\eta)$. 

Regarding the third point, a permutation $\hat \pi$ can be computed in polynomial time from the directed acyclic graph $\hat \cG$ using Mirsky's algorithm~\cite{mirsky1971dual} -- see also~\cite{pananjady2022isotonic}. It simply consists in finding the minimal experts $i$ in $\hat \cG$, removing them and repeat this process. This construction is in fact equivalent to ranking the experts according to the index $\rk_{\hat \cG, 0}$ as defined in \Cref{eq:rank}.

\section{Concentration inequality for rectangular matrices}\label{sec:concentration}

In this section, we state a concentration inequality for rectangular random matrices with independent entries satisfying a Bernstein-type condition. 
This section can be read independently of the rest of the paper. Let $p$ and $q$ be two positive integers and $X \in \bbR^{p \times q}$ be a random matrix with independent and mean zero coefficients. Assume that there exists $\sigma > 0$ and $K \geq 1$ such that for any $i = 1,\dots, p$ and $k = 1, \dots, q$,
\begin{equation}\label{eq:condition_moments_bernstein}
	\forall u \geq 1, ~~~~\bbE[(X_{ik})^{2u}] \leq \frac{1}{2}u!\sigma^2 K^{2(u-1)} \enspace .
\end{equation}

This Bernstein-type condition \Cref{eq:condition_moments_bernstein} is exactly the same as Assumption 1 in \cite{bellec2019concentration} -- see~\cite{bellec2019concentration} for a discussion. Let $\Lambda \in \bbR^{p\times p}$ be any orthogonal projection matrix, i.e. $\Lambda = \Lambda^T$ and $\Lambda^2 = \Lambda$. We write $r_{\Lambda}$ for the rank of $\Lambda$.

\begin{proposition}\label{prop:concentration_bernstein_op}
	There exists a positive numerical constant $\kappa$ such that the following holds for any $\delta > 0$.
	\begin{equation}
		\|\Lambda(XX^T - \bbE[XX^T])\Lambda\|_{\op} \leq \kappa \left[\sqrt{(\sigma^4pq + \sigma^2 q)\log(p/\delta)} + (\sigma^2r_{\Lambda} + K^2\log(q))\log(p/\delta)\right] \enspace .
	\end{equation}
\end{proposition}

For the sake of the discussion, consider the particular case where $X_{ik} = B_{ik} E_{ik}$, with $B_{ik}$ and $E_{ik}$ being respectively independent Bernoulli random variable of parameter $\sigma^2$ and centered Gaussian random variable with variance $1$.  By a simple computation done e.g. in \Cref{eq:moment_bernstein}, $X_{ik}$ satisfies condition \Cref{eq:condition_moments_bernstein} with $K$ being of the order of a constant.
Hence, if $K^2\log(q) \leq \sigma^2 p$, applying \Cref{prop:concentration_bernstein_op} with the identity matrix $\Lambda$ gives 
\begin{equation}\label{eq:concentration_bernstein_simple}
	\|XX^T - \bbE[XX^T]\|_{\op} \leq 2\kappa\sigma^2\left[\sqrt{pq \log(p/\delta)} + p\log(p/\delta)\right]\enspace ,
\end{equation}
with probability at least $1- \delta$.

\medskip

Up to our knowledge, the inequality \Cref{eq:concentration_bernstein_simple} is tighter than state-of-the-art result  random rectangular sparse matrices in the regime where $q \gg p$ and $\sigma^2 \ll 1$. In fact, most of the results in the literature concerning random matrices state concentration inequalities for the non centered operator norm $\|XX^T\|_{\op}$ -- see the survey of Tropp \cite{2015arXiv150101571T}.

More specifically, Bandeira and Van Handel \cite{bandeira2016sharp} provide tight non-asymptotic bounds for the spectral norm of a square symmetric random matrices with independent Gaussian entries, and derive tail bounds for the operator norm of $XX^T$. For instance, Corollary 3.11 in \cite{bandeira2016sharp}, implies that, for some numerical constant $c$,  $\bbE[\|XX^T\|^2_{\op}] \leq c(\sigma^2 (p \lor q) + \log(p \lor q))$. Together with a triangular inequality, Bandeira and Van Handel imply $\|XX^T - \bbE[XX^T]\|^2_{\op}\leq c\sigma^2( (p \lor q) + \log(\tfrac{p \lor q}{\delta})) $ with probability higher than $1-\delta$. 

While the order of magnitude $\sigma^2(p\lor q)$ is tight for controlling the operator norm $\|XX^T\|^2_{\op}$ of the non-centered Gram matrix with high probability, \Cref{eq:concentration_bernstein_simple} implies that the right bound for 
$\|XX^T - \bbE[XX^T]\|^2_{\op}$ is rather $\sigma^2\sqrt{pq}$ which is significantly smaller in the regime $p \ll q$ and $\sigma^2 \ll 1$.

In the proof of \Cref{th:UB}, we could have used those previous results for controlling the matrices of the form $\|XX^T - \bbE[XX^T]\|^2_{\op}$. However, we would have then achieved a suboptimal risk upper bound. Indeed, \Cref{prop:concentration_bernstein_op} plays critical role in the proof of \Cref{th:UB}, when we need to handle matrices with partial observations that are possibly highly rectangular in the spectal step of the procedure \Cref{eq:ACP}.

The  proof of \Cref{prop:concentration_bernstein_op} relies on the observation that the matrix $XX^T - \mathbb{E}[XX^T]$ is the sum of $q$ centered rank 1 random matrices. This allows us to apply Matrix Bernstein-type concentration inequalities for controlling  the operator norm of this sum -- see \cite{2015arXiv150101571T} or Section 6 of \cite{wainwright2019high}. 


\appendix

\section{Proof of \Cref{th:UB}}\label{sec:general_analysis}

\subsection{Notation and signal-noise decomposition}
We first introduce some notation, and in particular the noise matrices on which we will apply concentration inequalities.
In what follows, we define for any matrix $A \in \bbR^{n \times d}$, and any vector $w \in \bbR^d$: 
\begin{equation}\label{eq:abuse_notation}
	\proscal<A_{i \cdot}, w> = \sum_{k =1}^d A_{ik}w_k \enspace .
\end{equation}
If $w$ belongs to $\bbR^Q$ where $Q$ is some subset of $[d]$, we also write $<A_{i \cdot}, w>= \sum_{k \in Q}^d A_{ik}w_k$. The same notation stands for the scalar product on matrices, namely $\proscal<A,A'> = \Tr(A^TA')$ if $A' \in \bbR^{n\times d}$. If $A$ and $A'$ are two matrices in $\bbR^{n \times d}$, then we write the coordinate-wise product $(A \odot A')_{ik} = A_{ik}A'_{ik}$. In what follows, we assume that $\pi^* = \mathrm{id}$. We make this assumption without loss of generality since we can reindex each expert $i$ with $i'=\pi^{*-1}(i)$. Recalling that $B$ is defined in \Cref{eq:def_matrix_B} we define 
\begin{equation}\label{eq:def_lambda_1}
	\lambda_1 := \bbP(B^{(s)}_{ik} = 1) = 1- e^{-\lambda_0}\enspace .
\end{equation}

If $\lambda_0 \leq 1$, we have $\lambda_0 \geq \lambda_1 \geq (1-\tfrac{1}{e})\lambda_0$. We assume in what follows that $\lambda_0 \leq 1$, which corresponds to the case where there are potentially many unobserved coefficients. The case $\lambda_0 \geq 1$ will be treated in \Cref{sec:full}. For an observation matrix $Y^{(s)}$ defined in \Cref{eq:batches}, we make the difference between $\bbE[Y^{(s)}] = \lambda_1 M$, which is the unconditional expectation of $Y^{(s)}$, and $\bbE[Y^{(s)}|B^{(s)}] = B^{(s)} \odot M$, which is the expectation of $Y^{(s)}$ conditionally to the matrix $B$. We write the noise matrix

\begin{equation}\label{eq:two_possible_noises}
	\noise^{(s)} = Y^{(s)} - \lambda_1 M \spaceAnd \widetilde E^{(s)} = Y^{(s)} - B^{(s)}\odot M \enspace .
\end{equation}

Recall that $\varepsilon_t = y_t - M_{x_t}$ is the subGaussian noise part in model \Cref{eq:model_partial}, and that $N_s$ is defined in \Cref{eq:batches}. Each coefficient $\widetilde E^{(s)}_{ik}$ can be rewritten as the average of the noise $\varepsilon_t$. that are present in $N^{(s)}$ and that correspond to coefficient $x_t=(i,k)$.
\begin{equation}\label{eq:aggregated_noise}
	\widetilde E^{(s)}_{ik} = \sum_{t \in N^{(s)}} \tfrac{\varepsilon_t}{\br^{(s)}_{ik}\lor 1}\1\{x_t = (i,k)\} \enspace .
\end{equation}
From now on, we often omit the dependence in $s$.  We will extensively use the decomposition $Y = \lambda_1M + \noise$, where $\lambda_1$ is defined in \Cref{eq:def_lambda_1} and $\noise$ in \Cref{eq:two_possible_noises}. Recalling that $B_{ik} = \1\{r_{ik} \geq 1\}$, we often rewrite $E$ as the sum of two centered random variables:
\begin{equation*}
	\noise_{ik} = (B_{ik} - \lambda_1)M + B_{ik}\widetilde E_{ik} \enspace .
\end{equation*}

Handling the concentration of the noise is more challenging in the case $\lambda_0 \leq 1$ than in the full observation regime $\lambda_0 \geq 1$ discussed in \Cref{sec:full}. Indeed, while subGaussian concentration inequalities are effective in the full observation regime $\lambda_0 \geq 1$, they lead to slower estimation rate in the case $\lambda_0 \leq 1$, for instance in \Cref{lem:concentration_l1}. Indeed, it turns out that the variance of a coefficient $\varepsilon_{ik}$ is of order $\lambda_0 \leq 1$, while the hoeffding inequality only implies that $B_{ik} - \lambda_1$, and in particular $\varepsilon_{ik}$ are $c$-subGaussian for some numerical constant $c$.
To overcome this issue, one of the main ideas is to use Bernstein-type bounds on the coefficients of $\noise$ and on the random matrix $EE^T - \bbE[EE^T]$- see \Cref{lem:concentration_proscal_bernstein} and \Cref{prop:concentration_bernstein_op}.

\subsection{General property on $\cW$}

Recall that we assume that $\lambda_0 \leq 1$, so that $\tfrac{1}{\lambda_0} \land \lambda_0= \lambda_0$ in \Cref{eq:update_W}, and that $\phi_{\lone}$ is defined in \Cref{eq:definition_phil1} by $\phi_{\lone} := 10^4\log(10^2nd/\delta)$. In the following, we let $\xi$ be the event on which the noise concentrates well for all the pairs $(Q,w)$ considered during the whole procedure. More precisely, we say that we are under event $\xi$, if for any $s=0, \dots, 5T-1$ and for any pair $(Q,w)$ that is used to compute a refinement as in \Cref{eq:updating_edges} we have

\begin{equation}\label{eq:condition_phil1}
	\abs{\proscal<\noise^{(s)}_{i \cdot} - \noise^{(s)}_{j \cdot}, w>} \leq \tfrac{1}{3}\phi_{\lone}\sqrt{\lambda_0} \quad  \text{ for any } (i,j) \in [n]^2 \enspace .
\end{equation}

\begin{lemma}\label{lem:concentration_l1}
	The event $\xi$ holds true with probability at least $1-2T\delta$.
\end{lemma}

The idea of \Cref{lem:concentration_l1} is to apply a bernstein-type inequality and a union bound on all the possible dot products $\proscal<\noise^{(s)}_{i \cdot}, w>$, for all the $5T$ possible $s$ and the at most $2T$ possible $w$. The upper bound is of the order of the square of the variance of $\noise_{ik}$ up to the polylogarithm factor $\phi_{\lone}$. The crucial point is that if $\proscal<\noise^{(s)}_{i \cdot}, w>$ is not $\lambda_0$-subGaussian, it satisfies the Bernstein's Condition [ 2.15 of \cite{massart2007concentration}] with variance $\nu=\lambda_0$ and scaling factor $b=\|w\|_{\infty}$. We then obtain an upper bound of order $\sqrt{\lambda_0}$ since any $w$ considered in the update step \Cref{eq:update_W} must satisfy \Cref{eq:condition_w}.Recall that $\bar \gamma$ is defined in \Cref{eq:def_gamma}. We fix in what follows a sequence $\overline \gamma = \gamma_0 > \gamma_1 > \gamma_2 > \dots > \gamma_{\floor{2\log_2(n)}} = \gamma_{\min}$ in $\Gamma$ satisfying property \Cref{eq:valid_gamma}. We say that $u$ is the level of the corresponding threshold $\gamma_u$. We say $\cW$ and $(\gamma_u)$ satisfies the property  $\cC(\cW, (\gamma_u))$ if the following holds
\begin{enumerate}
	\item {\bf consistency:} For any $(i,j) \in \cG(\cW, \gamma_{\min})$ it holds that $\pi^*(i) > \pi^*(j)$.
	\item {\bf weak-transitivity:} Fix any $u \in \{0, \dots, \floor{2\log_2(n)} -1\}$. For any experts $i, j$, $k$, if $i$ is $\cG(\cW, \gamma_u)$-above $j$ and $k \in \cN(\cG(\cW, \gamma_{u+1}), j)$, then any $i'\geq i$ is also $\cG(\cW, \gamma_{\min})$-above $k$.
\end{enumerate}

The first point of the above property means that at threshold $\gamma_{\min}$, there is no mistake in the directed graph $\cG(\cW, \gamma_{\min})$, meaning that if there is an edge from $i$ to $j$ in $\cG(\cW, \gamma_{\min})$, then $i$ is truly above $j$. Moreover, we only state the consistency property of the graph $\cG(\cW, \gamma_{\min})$, but this property also implies that, for any more conservative threshold $\gamma \geq \gamma_{\min}$, any $(i,j) \in \cG(\cW, \gamma)$ satisfies $\pi^*(i) > \pi^*(j)$. This is due to the fact that $\cG(\cW, \gamma) \subset \cG(\cW, \gamma_{\min})$.
The weak transitivity property states in particular that if there is a path from $i$ to $j$ in the more conservative graph $\cG(\cW, \gamma_u)$, then there is a path from $i$ to any $k$ in the neighborhood of $j$ at the less conservative threshold $\gamma_{\min}$.
The following lemma states that the above property remains true for the weighted graph $\cW'$, after any update \Cref{eq:update_W} of the whole procedure.

\begin{lemma}\label{lem:consistency}
	Under $\xi$, the property $\cC(\cW', (\gamma_u))$ holds true for any directed weighted graph $\cW'$ 
obtained at any stage of \Cref{alg:principal} and \Cref{alg:refine_locally_sketch}. 
\end{lemma}

We denote in the following $\cW_t$ for the directed weighted graph at the begining of step $t$. For any $u \in [0, \floor{2\log_2(n)}]$, we also write as a short hand $\cG_{t, u} = \cG(\cW_t, \gamma_u)$ for the directed graph at begining of step $t$ and level $u$ and $P_{t,u}(i) = \cN(\cG_{t, u}, i)$ for the set of experts that are not comparable with $i$ according to $\cG_{t,u}$. For any sequence of experts $I$, we write $\cP_{t, u}(I)$ for the sequence of subsets $(P_{t,u}(i))_{i \in I}$.
Let us now divide the $T$ steps of the algorithm into $\tau_{\max} = \floor{\log_2(n)}+1$ epochs of $K= \floor{T/\tau_{\max}}$ steps. For any $\tau \in [0,\tau_{\max}]$, we also write $\cG^K_{\tau, u} = \cG_{\tau K, u}$, $P^K_{\tau,u}(i) = P_{\tau K, u}(i)$ and $\cP^K_{\tau, u}(I) = \cP^K_{\tau, u}(I)$. Now we consider for each epoch $\tau$ 
a sequence of experts $I(\tau) = (i_1(\tau), \dots, i_{L_\tau}(\tau))$ defined by induction:
\begin{itemize}
	\item $I(0)$ is the empty sequence
	\item For $\tau \geq 0$, let $(i_1, \dots, i_L)$ be the sequence ordered according to $\pi^*$ and corresponding to the union of the already constructed sequences $\bigcup_{\tau' \leq \tau} I(\tau')$ , and $i = 0$, $i_{L + 1} = n+1$. For any $l \in [0, L]$, let $A_l$ be the set of experts that are $\cG^K_{\tau+1, 2\tau+1}$-below $i_{l+1}$ but $\cG^K_{\tau+1, 2\tau+1}$-above $i_l$. For  all $l$ such that $A_l$ is not empty, we define $i'_l$ as the expert of $A_l$ which is any expert closest to the median $\floor{(i_l + i_{l+1})/2}$, and the new sequence $I(\tau + 1) := (i'_l)$.
\end{itemize}

By definition, remark that $I(1)$ is equal to $(\floor{(n+1)/2})$.
The induction step aims at building a sequence $I(\tau+1)$ that is disjoint from $\cup_{\tau' \leq \tau} I(\tau')$, and that cuts each set $A_l$ of experts that are above $i_l$ and below $i_{l+1}$ according to the graph at epoch $\tau + 1$ and level $2\tau+1$. Given the already constructed collections of perfectly ordered experts $I(\tau')$ for $\tau' \leq \tau$, the idea of $I(\tau+1)$ is that it tends to fill the gaps between the neighborhoods in $\cG_{\tau+1,2\tau+1}$ of any two successive experts in $\cup_{\tau' \leq \tau}I(\tau')$. 

\medskip

By monotonicity, it holds that for any expert $i$, epoch $\tau$ and level $u$ that 
$P^K_{\tau+1, u+1}(i) \subset P^K_{\tau+1,u}(i) \subset P^K_{\tau,u}(i)$. We say that the sets $P^K_{\tau,2\tau}(i)$ and $P^K_{\tau, 2\tau+1}(i)$ are the neighborhoods of $i$ at the beginning of epoch $\tau$ and that the sets $P^K_{\tau+1, 2\tau}(i), P^K_{\tau+1, 2\tau+1}(i)$ are the neighborhood of $i$ at the end of epoch $\tau$. The neighborhoods at the end of a given epoch $\tau$ are obtained from the neighborhoods at the beginning the of epoch $\tau$ after $K$ steps of the \Cref{alg:principal}. On the other hand, we say that the sets $P^K_{\tau, 2\tau}, P^K_{\tau+1, 2\tau}$ are the conservative subsets at epoch $\tau$, since they correspond to a more conservative directed graph with threshold $\gamma_{2\tau} \geq \gamma_{2\tau+1}$. The following lemma states that, at any epoch $\tau$, the conservative subsets at the beginning of epoch $\tau$ are well separated according to the true order $\pi^* = \mathrm{id}$:
\begin{lemma}\label{lem:packing_general}
	Under event $\xi$, for any $\tau \in [0, \tau_{\max}]$, letting $(i_1, \dots, i_L) = I(\tau)$, we have 
	\begin{align*}
		P^K_{\tau,2\tau}(i_1) < \dots <  P^K_{\tau,2\tau}(i_{L}).
	\end{align*}
\end{lemma}
In other words, \Cref{lem:packing_general} implies that, for any $l < l'$, any expert in $P^K_{\tau,2\tau}(i_l)$ is $\pi^*$-below any expert in $P^K_{\tau,2\tau}(i_{l'})$. 
As a consequence, it holds that for any $l \in [1, L_{\tau} - 2]$,
\begin{equation}\label{eq:packing_stronger}
	P^K_{\tau,2\tau}(i_l) \overset{\cG^K_{\tau, 2\tau}}{\prec} P^K_{\tau,2\tau}(i_{l+2})\enspace .
\end{equation}
Namely, any expert in $P^K_{\tau,2\tau}(i_l)$ is $\cG^K_{\tau, 2\tau}$-below any expert in $P^K_{\tau,2\tau}(i_{l+2})$. Indeed, \Cref{lem:packing_general} and first point of event $\xi$ imply that any expert $j$ in $P^K_{\tau,2\tau}(i_l)$ is $\cG^K_{\tau, 2\tau}$-below $i_{l+1}$, since $j$ cannot be in $P^K_{\tau,2\tau}(i_l)$. On the other hand, $i_{l+1}$ is itself $\cG^K_{\tau, 2\tau}$-below any expert of $P^K_{\tau,2\tau}(i_{l+2})$ for the same reason.
The following lemma states that the ending less conservative subsets are covering the set of all experts.
\begin{lemma}\label{lem:covering_general}
	Under event $\xi$, it holds that $$[n] = \bigcup_{\tau = 0}^{\tau_{\max} - 1}\bigcup_{i\in I(\tau)} P^K_{\tau+1, 2\tau+1}(i) \enspace .$$
\end{lemma}

Let $\hat \pi$ be the estimator obtained from the final weighted directed graph $\cW$ at the end of the procedure, that is any permutation on $[n]$ that is consistent with the largest acyclic graph of the form $\cG(\cW, \gamma)$ for all $\gamma > 0$. For any sequence of subsets $\cP = (P_1, \dots, P_L)$ we define
\begin{equation}\label{eq:def_square_norm}
	\SN(\cP) = \sum_{P \in \cP} \|M(P) - \overline M(P) \|_F^2 \enspace.
\end{equation}
The following proposition that we can control the $L_2$ error of $\hat \pi$ by the maximum over all epoch $\tau$ of the sum over $\tau$ of the square norms of the groups in $\cP^K_{\tau+1, 2\tau+1}$.

\begin{proposition}\label{prop:UB_by_square_norm}
	Under event $\xi$, it holds that
	\begin{equation}\label{eq:UB_by_square_norm}
		\|M_{\hat \pi^{-1}} - M\|_F^2 \leq 4\sum_{\tau = 0}^{\tau_{\max}-1}\SN(\cP^K_{\tau+1, 2(\tau + 1)}) \enspace .
	\end{equation}
\end{proposition}

Recall that $\bar \gamma$ is defined in \Cref{eq:def_gamma}, and that $\Gamma$ can be taken to be a valid grid with $\bar \gamma$ smaller than a polylogarithm in $n,d,\delta$.
The final proposition states that at any level $u$ and any step $t$, any sequence of subset that can be ordered according to the already constructed graph $\cG_{t,u}$ as in \Cref{eq:packing_stronger} will either have a square norm smaller than the minimax rate $\rho_{\perm}$, defined in \Cref{eq:rho_perm} or almost exponentially decrease its square norm with high probability.

\begin{proposition}\label{prop:UB_on_square_norm}
	Fix any $u \in [0, 2\tau_{\max}]$ and step $t < T$, and assume that $I = (i_1, \dots, i_{L})$ is a sequence of experts that satisfies $P_{t, u}(i_1) \overset{\cG_{t,u}}{\prec} \dots \overset{\cG_{t,u}}{\prec} P_{t, u}(i_{L})$. Then on the intersection of the event $\xi$ (defined in \Cref{eq:condition_phil1}) and an event of probability higher than $1 - 5\delta$, it holds that 
	\begin{equation*}
		\SN(\cP_{t+1, u}(I)) \leq \left[C\bar \gamma^{6}\rho_{\perm}(n, d, \lambda_0)\right]\lor \left[\left(1 - \frac{1}{4\bar \gamma^2}\right)\SN(\cP_{t,u}(I))\right]\enspace ,
	\end{equation*}
	for some numerical constant $C$. 
\end{proposition}

Let us fix $\tau \in \{0, \dots, \tau_{\max}-1\}$. Applying \Cref{prop:UB_on_square_norm} for each $t = K\tau, \dots, K\tau + K-1$ and $u = 2(\tau + 1)$ -the hypothesis of \Cref{prop:UB_on_square_norm} being satisfied by \Cref{eq:packing_stronger}, we obtain with probability $1 - 5(K+T)\delta$ that 
\begin{align*}
	\SN(\cP_{\tau+1, 2(\tau+1)}) &\leq \left[C\bar \gamma^{6}\rho_{\perm}(n, d, \lambda_0)\right]\lor e^{-\frac{T}{4\tau_{\max}\bar \gamma^4}}nd \\
	&\leq CT\bar \gamma^{6}\rho_{\perm}(n, d, \lambda) \enspace,
\end{align*}

if $T$ is larger than $4\bar \gamma^6 \geq 4\log^2(nd)\bar \gamma^4$.
We conclude the proof of \Cref{th:UB} with \Cref{prop:UB_by_square_norm}, using that $4\tau_{\max} \leq \bar \gamma$:
\begin{equation*}
	\|M_{\hat \pi^{-1}} - M\|_F^2 \leq 4\sum_{\tau = 0}^{\tau_{\max}-1}\SN(\cP^K_{\tau+1, 2(\tau + 1)})\leq CT\bar \gamma^{7}\rho_{\perm}(n, d, \lambda) \enspace .
\end{equation*}

\section{Proofs of the lemmas of \Cref{sec:general_analysis} and of \Cref{prop:UB_by_square_norm}}

\subsection{Proof of \Cref{prop:UB_by_square_norm}}
	Let $\hat \pi$ be any arbitrary permutation that is consistent with the largest $\DAG$ $\cG(\cW, \bar \gamma)$, as defined in \Cref{subsec:weigthed_graph_permutation}. Recall that we assume in this proof that $\pi^* = \mathrm{id}$.
	By \Cref{lem:covering_general}, for any $i\in [n]$ there exists $\tau \in [0, \tau_{\max}-1]$ and $i_0 \in I(\tau)$ such that $i \in P^K_{\tau+1, 2\tau + 1}(i_0)$. 
	
	Let us define the interval $[a,b]$ as the maximal interval containing $i_0$ and that is included in the more conservative set $P^{K}_{\tau+1, 2\tau}$. Now, if $j > b$, then by definition there exists $j'$ such that $j \geq j' > b$ and $j' \not \in P^{K}_{\tau+1, 2\tau}$.
	Summarizing the properties, we have
	$j \geq j' \overset{\cG_{\tau+1, 2\tau}}{\succ} i_0$, 
	and that $i$ is in the neighborhood of $i_0$ in the graph $\cG_{\tau+1, 2\tau+1}$.
	Hence, applying the weak-transitivity property (first in $\cC$), holding true on event $\xi$ - see \Cref{lem:consistency}, we obtain that $j$ is $\cG(\cW_{K(\tau+1)}, \gamma_{\min})$-above $i$. By the consistency property (second point in $\cC$), $j$ is also necessarily $\cG(\cW, \gamma_{\min})$-above $i$, and this proves that all the $n - b$ experts $j$ satisfying $j > b$ are $\cG(\cW, \gamma_{\min})$ above $i$. Hence it holds that $\hat \pi(i) \leq b$. By symmetry, we also prove that $\hat \pi(i) \geq a$, so that
	 \begin{equation}\label{eq:subset_3}
		\hat \pi(i) \in [a,b] \subset P^K_{\tau+1, 2\tau}(i_0).
	 \end{equation}
	 
	 Finally, we have 

	 \begin{align*}
		\| M_{\hat \pi^{-1}} - M \|_F^2 &= \sum_{i =1}^n \|M_{\hat \pi(i) \cdot} - M_{i \cdot} \|_F^2 \\
		&\leq \sum_{\tau = 0}^{\tau_{\max}}\sum_{i_0 \in I(\tau)}\sum_{i \in P^K_{\tau+1, 2\tau+1}(i_0)} \|M_{\hat \pi(i) \cdot} - M_{i \cdot} \|^2 \\
		&\leq 2\sum_{\tau = 0}^{\tau_{\max}}\sum_{i_0 \in I(\tau)}\sum_{i \in P^K_{\tau+1, 2\tau+1}(i_0)} \|M_{i \cdot} - \overline m(P^K_{\tau+1, 2\tau}(i_0)) \|^2 + \|M_{\hat \pi(i) \cdot} - \overline m(P^K_{\tau+1, 2\tau}(i_0)) \|^2 \\
		&\leq 4\sum_{\tau = 0}^{\tau_{\max}}\sum_{i_0 \in I(\tau)}\sum_{i \in P^K_{\tau+1, 2\tau}(i_0)} \|M_{i \cdot} - \overline m(P^K_{\tau+1, 2\tau}(i_0)) \|^2 \enspace ,
	 \end{align*}
	 where we used \Cref{lem:covering_general} for the first inequality and \Cref{eq:subset_3} for the last inequality.

\subsection{Proof of the lemmas of \Cref{sec:general_analysis}}

We postpone the proof of \Cref{lem:concentration_l1} to the next subsection.
\begin{proof}[Proof of \Cref{lem:consistency}]
	Recall that we consider the case $\lambda_0 \leq 1$, so that $\lambda_0 \land 1/\lambda_0 = \lambda_0$ in \Cref{eq:update_W}.

	Consider any substep of the whole procedure where the current directed weighted graph is $\cW'$. 
	For the first point, remark that $i$ is $\cG(\cW', \gamma_{\min})$-above $j$ only if there exists a previous substep during which we find out that $\proscal<Y_{i \cdot} - Y_{j \cdot},w> \geq \gamma_{\min}$ on some direction $w \in \bbR^Q$, where $Y$ is the sample used to refine the edges \Cref{eq:updating_edges}. Since $\gamma_{\min} > \phi_{\lone}$, then decomposing
	$Y = \lambda_1 M + \noise$ as in \Cref{eq:two_possible_noises}, we have
	\begin{align}
		\lambda_1\proscal<M_{i \cdot} - M_{j \cdot},w> \geq \proscal<Y_{i \cdot} - Y_{j \cdot},w> - \proscal<\noise_{i \cdot} - \noise_{j \cdot},w> > 0 \enspace,
	\end{align}
	where the last inequality comes from \Cref{eq:condition_phil1}, using the notation \Cref{eq:abuse_notation}.
	Since the coefficients of $w$ are nonegative, we have proven that $i$ is above $j$.
	For the second point, assume that $i$ is $\cG(\cW, \gamma_u)$-above $j$, and take $i' \geq i$. As before, there exists a direction $w$ used during the procedure such that $\proscal<Y_{i \cdot} - Y_{j \cdot},w> \geq \gamma_u$. Now consider any $k \in \cN(\cG(\cW, \gamma_{u+1}), j)$. On the direction $w$, we have under the event $\xi$ defined in \Cref{eq:condition_phil1} that
	\begin{align*}
		\proscal<Y_{i' \cdot} - Y_{k \cdot},w> &\geq \lambda_1\proscal<M_{i' \cdot} - M_{k \cdot},w> - \tfrac{1}{3}\phi_{\lone}\sqrt{\lambda_0 }\\
		&\geq \lambda_1\proscal<M_{i \cdot} - M_{k \cdot},w> - \tfrac{1}{3}\phi_{\lone}\sqrt{\lambda_0 }\\
		&\geq \proscal<Y_{i \cdot} - Y_{j \cdot},w> - \proscal<Y_{k \cdot} - Y_{j \cdot},w> - \phi_{\lone}\sqrt{\lambda_0 } \\
		&\geq (\gamma_u - \gamma_{u+1} - \phi_{\lone})\sqrt{\lambda_0 } \geq \gamma_{\min}\sqrt{\lambda_0 } \enspace ,
	\end{align*}
	where the last inequality comes from the assumption \Cref{eq:valid_gamma}.
	We conclude that $i'$ is $\cG(\cW',\gamma_{\min})$-above $k$.
\end{proof}

\begin{proof}[Proof of \Cref{lem:packing_general}]
	We proceed by induction over $\tau\geq 0$.
	The lemma is trivial for $\tau = 0,1$ since $I(0)$ is empty and $I(1) = (\floor{(n+1)/2})$. 
	Let $\tau \geq 1$ and $i_1,i_2, i_3$ be three experts in $I(\tau)\cup\{0, n+1\}$ such that $i_1 < i_2 < i_3$. Let $A$ be the set of experts that are $\cG^K_{\tau+1, 2\tau+1}$-above $i_1$ and $\cG^K_{\tau+1,2\tau+1}$-below $i_2$, and $A'$ be the set of experts that are $\cG^K_{\tau+1, 2\tau+1}$-above $i_2$ and $\cG^K_{\tau+1,2\tau+1}$-below $i_3$. Assume that both sets $A$ and $A'$ are nonempty, and let $j \in A$ and $j' \in A'$. Let us apply the weak-transitivity of $\cW,(\gamma_u)$ in Property $\cC$ - which holds true under $\xi$ from \Cref{lem:consistency} - with $u=2\tau+1$. Since $j$ is $\cG^K_{\tau+1,2\tau+1}$-below $i_2$, any $k \in P^K_{\tau+1, 2(\tau + 1)}(j)$ is $\pi^*$-below $i_2$. We also prove that any $k' \in P^K_{\tau+1, 2(\tau + 1)}(j')$ is $\pi^*$-above $i_2$. We conclude that $P^K_{\tau+1, 2(\tau + 1)}(j) < P^K_{\tau+1, 2(\tau + 1)}(j')$, and the proof of the lemma follows.
\end{proof}

\begin{proof}[Proof of \Cref{lem:covering_general}]
	We prove that, by construction, any expert $i \in [n]$ is at distance less than $(n+1)/2^{\tau+1}$ of $\bigcup_{\tau' = 0}^{\tau} \bigcup_{i\in I(\tau)} P^K_{\tau+1, 2\tau+1}(i)\cup\{0, n+1\}$. This is obvious for $\tau = 0$ since any expert is at distance less than $(n+1)/2$ of $0$ or $n+1$. 
	Let $(i_1, \dots, i_L)=\bigcup_{\tau' \leq \tau} I(\tau')$ be the collection of experts in the union of all possible $I(\tau')$ that is ordered according to $\pi^*$. If $j$ is any expert in $[n]$, then we let $l \in [0, L]$ be such that $i_l \leq j \leq i_{l+1}$. We can assume that $j \not \in P^K_{\tau+1, 2\tau+1}(i_l)$ and $j \not\in P^K_{\tau+1, 2\tau+1}(i_{l+1})$ because otherwise the distance of $j$ to $\bigcup_{\tau' = 0}^{\tau} \bigcup_{i \in I(\tau)}^{L}P^K_{\tau+1, 2\tau + 1}(i)$ is $0$. Using property $\cC$ holding true from \Cref{lem:consistency}, it holds that the set $A$ of experts that are $\cG^K_{\tau+1, 2\tau+1}$-above $i_l$ but $\cG^K_{\tau+1, 2\tau+1}$-below $i_{l+1}$ contains $j$ and therefore is nonempty. Now, let $m = \floor{(i_{l} + i_{l+1})/2}$ and $i'$ be any expert closest to $m$ in $A$, as defined in the construction of $I(\tau+1)$, and assume without loss of generality that $m \leq i'$. 
	We consider the following cases:
	\begin{itemize}
		\item $m \leq i' \leq j$: In that case, $j$ is at distance less than $(i_{l+1} - m)/2$ of $i'$ or $i_{l+1}$.
		\item $m \leq j < i'$: This case is not possible since $i'$ is the closest expert to $m$ in $A$.
		\item $j < m < i'$: In that case, since $i'$ minimizes the distance to $m$, we necessarily have that $m \in P^K_{\tau+1, 2\tau + 1}(i_l) \cup P^K_{\tau+1, 2\tau + 1}(i_{l+1})$. Hence $j$ is at distance less than $(m - i_l)/2$ of $m$ or $i_{l}$.
	\end{itemize}
	We have proved that the distance of any $j$ to $\bigcup_{\tau' = 0}^{\tau + 1} \bigcup_{i \in I(\tau)}^{L}P^K_{\tau+1, 2\tau + 1}(i)\cup\{0,n+1\}$ is at most $(m-i_l)/2$ or $(i_{l+1}-m)/2$. Using the induction hypothese, we have that $m - i_l$ and $i_{l+1} -m$ are both less than $n/2^{\tau+1}$, which concludes the induction.

	Finally, applying this property with $\tau_{\max}-1 = \floor{\log_2(n)}$ gives a distance strictly smaller than $1$, which proves the result.
\end{proof}

\subsection{Proof of \Cref{lem:concentration_l1}}

Let us start with the following lemma, which gives a concentration bound when $\lambda_0 \leq 1$:

\begin{lemma}\label{lem:concentration_proscal_bernstein}
	 For any $\delta'>0$ and for any matrix $W \in \bbR^{n \times d}$, the following inequality holds with probability at least $1-\delta'$:
	\begin{equation}\label{eq:concentration_bernstein}
		|\proscal<\noise, W>| \leq \sqrt{4e^2\|W\|_F^2\lambda_0\log\left(\frac{2}{\delta'}\right)} + \|W\|_{\infty}\log\left(\frac{2}{\delta'}\right) \enspace .
	\end{equation}
\end{lemma}
Now we apply \Cref{lem:concentration_proscal_bernstein} with the matrix $W$ with $0$ coefficients except at line $i$ where it is equal to the vector $\frac{w}{\|w\|_2}$ as defined in \Cref{eq:updating_edges} and we deduce that
\begin{equation}\label{eq:transformation_B_w}
	|\proscal<\noise_{i, \cdot}, \frac{w}{\|w\|_2}>| \leq \sqrt{4e^2\lambda_0\log\left(\frac{2}{\delta'}\right)} + \frac{\|w\|_{\infty}}{\|w\|_2}\log\left(\frac{2}{\delta'}\right) \leq 11 \sqrt{\lambda_0} \log(2/\delta') \enspace ,
\end{equation} 
where the last inequality comes from Condition \Cref{eq:condition_w} on $w$.
Now choosing $\delta' = \delta/(4Tn^6)$, a union bound over the at most $2n^2T|\cH|(|\Gamma| \land n^2)$ pairs $(Q,w)$ considered during the procedure, we deduce the bound of \Cref{lem:concentration_l1} for all $\lambda_0 \leq 1$. 

\begin{proof}[Proof of \Cref{lem:concentration_proscal_bernstein}]
	Recall that $E, \widetilde E$ are defined in \Cref{eq:two_possible_noises} and that we have in particular 
	\begin{equation*}
		E_{ik} = (B_{ik}-\lambda_1)M_{ik} + \widetilde E_{ik} \enspace .
	\end{equation*}
	
	Let $x>0$. By Cauchy-Schwarz inequality, we have 
	$$\bbE[e^{x \noise_{ik}}] \leq \sqrt{\bbE[e^{2x(B_{ik} - \lambda_1)M_{ik}}]}\sqrt{\bbE[e^{2x \widetilde E_{ik}}]} \enspace ,$$
	where we recall that $\lambda_1 = 1 - e^{-\lambda_0} \leq \lambda_0$.
	We have 
	\begin{align*}
		\bbE[e^{2x(B_{ik} - \lambda_1)M_{ik}}] \leq e^{-2\lambda_1 xM_{ik}}(\lambda_1 (e^{2xM_{ik}}-1) + 1) \leq e^{\lambda_1 e^2 x^2} \enspace ,
	\end{align*}
	and 
	\begin{align*}
		\bbE[e^{2x\widetilde E_{ik}}] \leq \lambda_1 (e^{2x^2}-1) + 1 \leq e^{\lambda_1 e^2 x^2} \enspace ,
	\end{align*}
	where we used the inequalities $e^{2x^2} - 1 \leq e^{2}x^2$ and $e^{2x} - 1 - 2x \leq e^{2}x^2$ for any $x \in [-1, 1]$.

	\medskip
	In particular, if $t>0$, a Chernoff bound with $x = \tfrac{t}{2\|W\|_F^2 \lambda_0 e^2} \land 1$ gives
	$$\bbP(\proscal<W, \noise> \geq t) \leq \exp(-(\tfrac{t^2}{4\|W\|_F^2\lambda_0 e^2 }\land t)) \enspace ,$$
	so that with probability at least $1-\delta'$:

	\begin{equation*}
		|\proscal<W, \noise>| \leq \sqrt{4e^2\|W\|_F^2\lambda_0\log\left(\frac{2}{\delta'}\right)} + \log\left(\frac{2}{\delta'}\right) \enspace .
	\end{equation*}

\end{proof}

\section{Proof of \Cref{prop:UB_on_square_norm}}\label{sec:proof_UB_on_square_norm}

{\bf Step 0 : general definitions}
\medskip 

In this proof, we fix $u \in \{0,\dots, 2\floor{\log_2(n)}+2\}$ and a corresponding threshold $\gamma_u$ in the sequence in $\Gamma$ satisfying $\gamma_u \geq \phi_{\lone}$ - see \Cref{eq:valid_gamma} - and a step $t < T$. We assume that $I = (i_1, \dots, i_{L})$ is a fixed sequence of experts that satisfies $P_{t, u}(i_1) \overset{\cG_{t,u}}{\prec} \dots \overset{\cG_{t,u}}{\prec} P_{t, u}(i_{L})$.

From now on, we ease the notation by omitting the dependence in $t,u,\gamma_u$ and we write $\cG = \cG_{t,u}$, $\cG' = \cG_{t+1, u}$, $\cP =(P_1, \dots, P_{L})$ for $\cP_{t,u}$ and $\cP'$ for $\cP_{t+1,u}$. We denote $\widetilde \cG^h$ for the 
directed graph at threshold $\gamma_u$ of the directed weighted graph $\widetilde \cW^h$ obtained at the end the first update \Cref{line:first_stat_l1} of \Cref{alg:refine_locally_sketch}. We also write $\widetilde P^h_l = \cN(\widetilde \cG^h, i_l)$ and $\widetilde \cP^h = (\widetilde P^h_1, \dots, \widetilde P^h_{L})$ for the corresponding sequence of subsets at height $h \in \cH$.
By monotonicity, it holds for any $h \in \cH$ that 
\begin{equation*}
	P'_l \subset \widetilde P^h_{l} \subset P_l \enspace .
\end{equation*}
\subsection{Step 1: Analysis of the selected set $\widehat Q$}\label{subsec:step_1}
Recall the definition of the neighborhoods \Cref{eq:neighborhood} of the set $P_l$ in the graph $\cG$:
\begin{equation*}
	\cN_a(l) = \bigcap_{i \in P_l} \rk_{\cG, i_l}^{-1}([1, a]) \spaceAnd \cN_{-a}(l) = \bigcap_{i \in P_l} \rk_{\cG, i_l}^{-1}([- 1, - a]) \enspace ,
\end{equation*}
Define for $\kappa > 0$ and $l \in [1, L]$ the population version $\bDelta^*_k$ of the width statistic $\widehat \bDelta_k$ - see \Cref{eq:stat_delta} - as the the difference of the best and worst expert of $P(i_l)$ if $a=0$ and as the difference of the average of the experts in $\cN_a(l)$ and the average of the expert in $\cN_{-a}(l)$: 
\begin{equation}
	\bDelta^*_k(0, l) = \max_{i,j \in P(i_l)}|M_{i,k} - M_{j,k}|\spaceAnd \enspace \bDelta^*_k(a, l) = \overline m_{k}(\cN_a(l)) - \overline m_{k}(\cN_{-a}(l)) \mbox{ if a $\geq 1$} .
\end{equation}

We also define $a^*(h,l)$ as the minimum $a \geq 1$ such that there are at least $\tfrac{1}{\lambda_0h^2}$ experts in $\cN_a(l)$ and in $\cN_{-a}(l)$:
\begin{equation}\label{eq:definition_nu}
	a^*(h, l) = \min \{a \geq 1 ~:~ |\cN_a(l)|\land |\cN_{-a}(l)| \geq \frac{1}{\lambda_0 h^2}\} \enspace .
\end{equation}
Now, define for $\phi \geq 1$:
\begin{align}\label{eq:definition_Q_pop}
	\begin{split}
		Q^{*h}_l(\phi) &:= \{k \in [d] ~:~ \bDelta^*_k(0, l) \in  [\phi h, 2 \phi h]\} \\
		\overline Q^{*h}_l(\phi) &:= \{k \in [d] ~:~ \bDelta^*_k(a^*(\phi^{-1}h, l), l) \geq  h/2 \}\enspace .
	\end{split}
\end{align}

The following lemma states that, for $\phi$ of order $\log(nd/\delta)$, we can sandwich $\widehat Q^h_l$ between the two fixed sets $Q^{*h}_l$ and $\overline Q^{*h}_l$:
\begin{lemma}\label{lem:inclusion_Q}
	Let $l$ be a fixed index in $\{1, \dots, L\}$ and $h$ a fixed height in $\cH$. There exists a numerical constant $\kappa_0 > 0$ such that, with probability at least $1- \delta/(L|\cH|)$, we have
	\begin{equation}
		Q^{*h}_l(\kappa_0\log(nd/\delta)) \subset \widehat Q_l^{h} \subset \overline Q^{*h}_l(\kappa_0\log(nd/\delta)) \enspace .
	\end{equation}
\end{lemma}

\subsection{Step 2 : l1-control of the intermediary sets $\widetilde \cP^h$} Recall that $\gamma_u$ is a threshold corresponding to a sequence in $\Gamma$ as defined in \Cref{eq:valid_gamma}. For any sets $P\subset [n], Q \subset [d]$, we say that $M(P,Q)$ is indistinguishable in $L_1$-norm if it satisfies 
\begin{equation}\label{eq:indistinguishable}
	\max_{i,j \in P} \| M_{i\cdot}(P, Q) - M_{j\cdot}(P,Q) \|_1 \leq 3\gamma_u\sqrt{\frac{|Q|}{\lambda_0}}\enspace . 
\end{equation}

For a fixed $l \in \{1, \dots, L\}$, let $\xi_{\lone}(l,h)$ be the event under which $M(\widetilde P^h_l, \widehat Q^{h}_l)$ is indistinguishable in $L_1$-norm.

\begin{lemma}\label{lem:event_indistinguishable}
	Let $l$ be a fixed index in $\{1, \dots, L\}$ and $h \in \cH$ such that $\lambda_0|Q^{*h}_l|  \geq 1$. The event  $\xi_{\lone}(l,h)$ holds true with probability at least $1-\delta / (L|\cH|)$.
\end{lemma}

Let $\kappa_0$ be a numerical constant given by \Cref{lem:inclusion_Q} and let $\phi_0 = \kappa_0 \log(nd/\delta)$. In what follows, we write for simplicity $(Q^{*h}_l, \widehat Q^{h}_l,\overline Q^{h}_l) = (Q^{*h}_l(\phi_0), \widehat Q^{h}_l(\phi_0),\overline Q^{h}_l(\phi_0))$.
\Cref{lem:event_indistinguishable} provides an upper bound only on the $L_1$ distance between rows of $M$ restricted to the subsets $\widetilde P^h_l$ and $\widehat Q^{h}_l$, while the square norm of a group \Cref{eq:def_square_norm} is defined with the $L_2$ distance. with \Cref{eq:indistinguishable}. The idea is that for any $k$ in $Q^{*h}$, and for any $i \in \widetilde P^h$, we have that $|M_{ik} - \overline m_{k}|^2\leq 2\phi_0 h|M_{ik} - \overline m_{k}|$. In particular, $\| M_{i\cdot}(\widetilde P^h_l, Q^{*h}_l) - \overline m_{\cdot}(\widetilde P^h_l, Q^{*h}_l) \|_2^2 \leq 2\phi_0 h\| M_{i\cdot}(\widetilde P^h_l, Q^{*h}_l) - \overline m_{\cdot}(\widetilde P^h_l, Q^{*h}_l) \|_1$. Hence, it holds from \Cref{lem:inclusion_Q}, \Cref{lem:event_indistinguishable}  and a union bound over all $l \in \{1, \dots, L\}$ and all $h \in \cH$ satisfying $\lambda_0|Q^{*h}_l|  \geq 1$ that
with probability at least $1-2\delta$,

\begin{equation}\label{eq:control_l1_to_l2}
	\sum_{i \in \widetilde P^h_l} \| M_{i\cdot}(\widetilde P^h_l, Q^{*h}_l) - \overline m_{\cdot}(\widetilde P^h_l, Q^{*h}_l) \|_2^2 \leq 6\phi_0 \gamma_u \left[h|\widetilde P^h_l|\sqrt{\frac{|\overline Q^{*h}_l|}{\lambda_0}}\right] \enspace ,
\end{equation}
simultaneously for all $l \in \{1, \dots, L\}$ and $h \in \cH$ satisfying $\lambda_0|Q^{*h}_l|  \geq 1$.
\begin{proof}[Proof of \Cref{lem:event_indistinguishable}]
  	Let $l$ be a fixed index in $\{1, \dots, L\}$ and $h$ be a fixed height in $\cH$. 
	If $a \geq 1$, the subset $P_l$ is disjoint from the sets $\cN_a(l) \cup \cN_{-a}(l)$ so that $\widehat Q^h_l$ is independent of $Y^{(1)}(P_l)$.
	Remark also that condition  \Cref{eq:condition_w} is satisfied since $\lambda_0|Q^{*h}_l|  \geq 1$ and $Q^{*h}_l \subset \widehat{Q}^{h}_l$.
	\medskip 

	Recall that we assume that $\lambda_0 \leq 1$. We write $w = \1_{\widehat Q^h_l}$ and we recall that $B = (B_{ik})$ is the matrix defined in \Cref{eq:def_matrix_B}. Let $i,j \in \widetilde P^{h}_l$ so that, by definition, we have that $\abs{\proscal<Y_{i\cdot} - Y_{j\cdot}, w>} \leq \gamma_u \sqrt{\lambda_0|\widehat Q^h_l|}$. 
	With probability at least $1-\delta /L$, for all $i,j$ in $P_l$ we have that
	\begin{align}
		\lambda_1 \abs{\proscal<M_{i\cdot} - M_{j\cdot}, w>} \leq \abs{\proscal<Y_{i\cdot} - Y_{j\cdot}, w>} + \abs{\proscal<\noise_{i\cdot} - \noise_{j\cdot}, w>} \leq (\gamma_u + \phi_{\lone}/2)\sqrt{\lambda_0|\widehat Q^h_l|} \enspace .
	\end{align}
	where the last inequality comes from \Cref{lem:concentration_proscal_bernstein} applied with $\delta' = \delta/n^3$ and from the definition of $\phi_{\lone}$ \Cref{eq:definition_phil1}. Recalling the two inequalities
	$\lambda_1 = 1- e^{-\lambda_0} \geq \lambda_0/2$ and $\phi_{\lone} \leq \gamma_u$, we obtain the result.
\end{proof}

\subsection{Step 3 : Local square norm reduction}
Henceforth we condition to the sample $Y^{(1)}$ of \Cref{alg:refine_locally_sketch} which allows us to assume that, for any $h \in \cH$, the two sequences of sets $\widetilde \cP^{h}$ and $\widehat \cQ^h$ are fixed.

For $\kappa_1 > 0$, let $\xi_{\loc}(l, h, \kappa_1)$ be the event holding true if the local square norm  of $M(P_l, \widehat Q_l^h)$ has decreased at the end of \Cref{alg:refine_locally_sketch}, that is
\begin{align}\label{eq:local_square_norm_reduction}
	\begin{split}
		\| M(P'_l, \widehat Q_l^h) -  \overline M(P'_l, \widehat Q_l^h)\|_F^2 \leq & ~\kappa_1 \gamma_u^4\left[\frac{1}{\lambda_0}\sqrt{|P_l| |\widehat Q_l^h|} + \frac{|P_l|}{\lambda_0}\right]  \\
		&~\lor \left(1 - \frac{1}{4\gamma_u^2}\right)\| M(P_l, \widehat Q_l^h) -  \overline M(P_l, \widehat Q_l^h)\|_F^2 \enspace .
	\end{split}
\end{align}

The following proposition states that given the fact that the experts in $\widetilde P^{h}_l$ are indistinguishable in $L_1$-norm and  $\lambda_0 (|\widetilde P^h_l|\land |Q_l^{*h}|) \geq 1$, the event $\xi_{\loc}$ holds true simultaneously for all $l$ and $h$ with high probability.

\begin{proposition}\label{prop:local_square_norm_reduction}There exists a numerical constant $\kappa_1$ such that the following holds, for any fixed index $l$ in $\{1, \dots, L\}$, and fixed height $h$ in $\cH$. Conditionally to $Y^{(1)}$, the event $\xi_{\lone}(l)$ and $\lambda_0 (|\widetilde P^h_l|\land |Q_l^{*h}|) \geq 1$, the event $\xi_{\loc}(l, h, \kappa_1)$ holds true with probability at least $1- 3\delta/(L|\cH|)$.
\end{proposition}
\Cref{prop:local_square_norm_reduction} is at the core of the analysis, and its proof contains a significant part of the arguments. This proposition and its proof are similar to Proposition D.5 in \cite{pilliat2022optimal}, but the main difficulty with respect to \cite{pilliat2022optimal} is that we do not achieve the optimal rate in $\lambda_0 \leq 1$ using only the subgaussianity of the coefficients of the noise $\noise$. A key step in the proof of \Cref{prop:local_square_norm_reduction} is \Cref{prop:concentration_bernstein_op}, which implies \Cref{lem:concentration_pca} and gives a concentration inequality of the operator norm of $\noise \noise^T - \bbE[\noise \noise^T]$. \Cref{prop:concentration_bernstein_op} is effective in that case since the coefficients of $\noise$ will be proven to satisfy \Cref{eq:condition_moments_bernstein}.

Then, the idea is that when a group $P'_l$ has a square norm of order at least $\frac{1}{\lambda_0}\sqrt{|P_l| |\widehat Q_l^h|} + \frac{|P_l|}{\lambda_0}$, the PCA-based procedure defined as in \Cref{eq:ACP} will output a vector $\hat v$ that is well aligned with the first left singular vector of $M(\widetilde P^h_l, \widehat Q^h_l)- \overline M(\widetilde P^h_l, \widehat Q^h_l)$. Moreover, the isotonic structure of $M(\widetilde P^h_l, \widehat Q^h_l)- \overline M(\widetilde P^h_l, \widehat Q^h_l)$ implies in fact that its operator norm is greater than a polylogarithmic fraction of its Frobenius norm (see \Cref{lem:structure_isotonic_matrices} or Lemma E.4 in \cite{pilliat2022optimal}], so that $\|\hat v^T(M(\widetilde P^h_l, \widehat Q^h_l)- \overline M(\widetilde P^h_l, \widehat Q^h_l))\|_2^2$ is of the same order as the square Frobenius norm. Hence after updating the edges, we can prove that the experts in $\widetilde P^h_l \setminus P'_l$ were contributing significantly to the Frobenius norm, which enforces the contraction part in the second term of the maximum in \Cref{eq:local_square_norm_reduction}. All the details of the proof can be found in \Cref{sec:proof_of_prop_reduction}.

\subsection{Step 4 : Control of the size of the sets $\overline Q^{*h}_l$}

 For any $p \in [n]\cap \{2^k~:~ k \in \bbZ^+\}$, let $\cL(p)$ be the sets of indices $l = 1, \dots, L$  whose corresponding group size $|P_l|$ belongs to $[p,2p)$. The two upper bounds implied by \Cref{eq:control_l1_to_l2} and \Cref{eq:local_square_norm_reduction} both depend on the selected subset of columns, which is included in $\overline Q_l^{*h}$ under the event of \Cref{lem:inclusion_Q}. The following lemma provides an upper bound on the sum over $l\in \cL(p)$ of the size of the sets $\overline Q_l^{*h}(\phi)$ defined in \Cref{eq:definition_Q_pop}, for any $\phi>0$.
 
\begin{lemma}\label{lem:Control_Q}
	For any $\phi \geq 1$ and any $h\in \cH$, it holds that 
	\begin{equation*}
		\sum_{l \in \cL(p)}|\overline Q^{*h}(\phi)| \leq 12 \phi^2 \left(\frac{1}{p\lambda_0 h^2}\lor 1\right)\frac{d}{h} \enspace .
	\end{equation*}
\end{lemma}

The proof of \Cref{lem:Control_Q} is mainly implied by the fact that the coefficients of $M$ are bounded by $1$.  Then, the idea is that in the case where all the sets $P_{l}$ are of size $p$, it is enough to take a number of group $a$ of order at most $\tfrac{1}{p\lambda_0h^2} \lor 1$ above and below each $P_l$ to ensure that the corresponding neighborhood of $P_l$ has size $|\cN_a(l)| \land |\cN_{-a}(l)| \geq \frac{1}{\lambda_0h^2}$.

\subsection{Step 5 : Conclusion of the previous steps}
We first decompose the square norm $\SN(\cP)$ as defined in \Cref{eq:def_square_norm} into two terms.
Assume that the event of \Cref{lem:inclusion_Q}, $\xi_{\lone}(l)$ and $\xi_{\loc}(l,h, \kappa_1)$ - see \Cref{lem:event_indistinguishable} and \Cref{prop:local_square_norm_reduction} - hold true.
Define $\cL_{-}$ as the sequence of indices $l$ such that the corresponding reduced subsets $P'_l$ have low local square norm for all $h \in \cH$. More precisely, we say that $l \in \cL_{-}$ if for all $h \in \cH$ we have
\begin{align}\label{eq:L_minus}
	\begin{split}
		\| M(P'_l, \widehat Q_l^h) -  \overline M(P'_l, \widehat Q_l^h)\|_F^2 \leq & ~\kappa_1 \gamma_u^4\left[\frac{1}{\lambda_0}\sqrt{|P_l| |\widehat Q_l^h|} + \frac{|P_l|}{\lambda_0}\right] \\
		&~\lor \frac{1}{2|\cH|}\| M(P_l) -  \overline M(P_l)\|_F^2 \enspace .
	\end{split}
\end{align}
We also define the complementary $\cL_{+} = [1, L] \setminus \cL_{-}$ and their corresponding subsets $\cP'_+, \cP'_-$ in $\cP'$. We have the following decomposition:
\begin{equation}\label{eq:decomposition_square_norm}
	\SN(\cP') = \SN (\cP'_+) + \SN(\cP'_-) \enspace .
\end{equation}

\medskip

Let us now give an upper bound of $\SN (\cP'_+)$. For any $l \in \cL_{+}$, there exists by definition an element $h_l \in \cH$ such that $\|M(P'_l,\widehat Q^{h_l}_l) - \overline M(P'_l,\widehat Q^{h}_l)\|_F^2 > \kappa_1 \gamma_u^4\left[\frac{1}{\lambda_0}\sqrt{|P_l| |\widehat Q_l^h|} + \frac{|P_l|}{\lambda_0}\right] \lor \frac{1}{2|\cH|}\| M(P_l, \widehat Q_l^h) -  \overline M(P_l, \widehat Q_l^h)\|_F^2$. Hence applying \Cref{eq:local_square_norm_reduction} with $h=h_l$, we have that, for any $l \in \cL_+$,
\begin{align*}
	\|M(P'_l) - \overline M(P'_l) \|_F^2 &= \|M(P'_l,\widehat Q^{h_l}_l) - \overline M(P'_l,\widehat Q^{h_l}_l)\|_F^2 + \|M(P'_l, [d]\setminus \widehat Q^{h_l}_l) - \overline M(P'_l, [d]\setminus \widehat Q^{h_l}_l)\|_F^2 \\
	&\leq  \| M(P_l) -  \overline M(P_l)\|_F^2 - \frac{1}{4\gamma_u^2}\|M(P_l,\widehat Q^{h_l}_l) - \overline M(P_l,\widehat Q^{h_l}_l)\|_F^2 \\
	&\leq  \left(1 - \frac{1}{\gamma_u^3}\right)\| M(P_l) -  \overline M(P_l)\|_F^2 \enspace ,
\end{align*}

where the third inequality comes from the second term of \Cref{eq:L_minus} together with $P'_l \subset P_l$ and $\gamma_u \geq \phi_{L_1} \geq 8|\cH|$, with $\phi_{L_1}$ defined in \Cref{eq:definition_phil1}. Hence we obtain that 
\begin{equation}\label{eq:ub_var_P_plus}
	\SN(\cP'_+) \leq \left(1 - \frac{1}{\gamma_u^3}\right)\SN(\cP_+) \enspace .
\end{equation}
 
\medskip

Finally, we give an upper bound of $\SN (\cP'_-)$. Let us write $\cD_n = \{2^k~:~ k\in \bbZ^+\}\cap [n]$ for the set of dyadic integer smaller than $n$. Given $p\in \cD_n$, we write $\cL_-(p) = \cL(p)\cap\cL_-$ for the set of indices in $\cL_-$ such that $|P_l|\in [p, 2p)$, and $\cP'_-(p)$ for the corresponding sequence of subsets in $\cP'_-(p)$. Let $\phi_0 = \kappa_0 \log(nd/\delta)$, where $\kappa_0$ is a numerical constant given by \Cref{lem:inclusion_Q}. By definition of $Q^{*h}_l$, the square norm of a group $P'_l$ restricted to questions that do not belong the set $\cup_{h\in \cH}Q^{*h}_l$ is smaller than $\phi_0nd\cdot\min(\cH) \leq \phi_0$. Hence we have that
\begin{equation}\label{eq:var_P_minus_1}
	\SN(\cP'_-) = \sum_{p \in \cD_n}\SN(\cP'_-(p))\leq \phi_0 + \sum_{(p,h) \in \cD_n \times \cH}\sum_{l \in \cL_-(p)}\| M(P'_l,Q^{*h}_l) - \overline M(P'_l,Q^{*h}_l)\|_F^2 \enspace .
\end{equation}

If $\lambda_0|Q^{*h}_l|\leq 1$ then we use the trivial inequality $\| M(P'_l,Q^{*h}_l) - \overline M(P'_l,Q^{*h}_l)\|_F^2 \leq |P'_l||Q^{*h}_l| \leq |P^h_l|/\lambda_0$,since the entries of $M$ are bounded by one.

If $\lambda_0|Q^{*h}_l|\geq 1$ and $|\widetilde P^{h}_l| \lambda_0 \leq 1$, we have that $h|\widetilde P^{h}_l|\sqrt{\frac{|\overline Q^{*h}_l|}{\lambda_0}} \leq \sqrt{\frac{|\widetilde P^{h}_l||\overline Q^{*h}_l|}{\lambda_0^2}}$, using the fact that $h \leq 1$. Hence, since the experts in $P'_l \subset \widetilde P^{h}$ are indistinguishable in $L_1$ norm by \Cref{lem:event_indistinguishable}, \Cref{eq:control_l1_to_l2} holds true and we have
\begin{align*}
	\| M(P'_l,Q^{*h}_l) - \overline M(P'_l,Q^{*h}_l)\|_F^2 &\leq 6\phi_0 \gamma_u \left[h|\widetilde P^h_l|\sqrt{\frac{|\overline Q^{*h}_l|}{\lambda_0}}\right]\\ 
	&\leq 6\phi_0 \gamma_u \left[\sqrt{h^2|\widetilde P^h_l|^2\frac{|\overline Q^{*h}_l|}{\lambda_0}} \land \sqrt{\frac{|\widetilde P^{h}_l||\overline Q^{*h}_l|}{\lambda_0^2}}\right]\\
	&\leq 12\phi_0\gamma_u \left[\sqrt{(h^2p \lambda_0 \land 1)\frac{p|\overline Q_l^{*h}|}{\lambda_0^2}} + \frac{p}{\lambda_0}\right] \enspace .
\end{align*}

Finally, if $\lambda_0(|Q^{*h}_l|\lor |\widetilde P^{h}_l|) \geq 1$, we are in position to apply \Cref{prop:local_square_norm_reduction}. For all $l \in \cL_-(p)$ and $h \in \cH$ that $\| M(P'_l,Q^{*h}_l) - \overline M(P'_l,Q^{*h}_l)\|_F^2$ is either smaller than $\frac{1}{2|\cH|}\| M(P_l) -  \overline M(P_l)\|_F^2$, or it is smaller than $\kappa_1 \gamma_u^4\left[\frac{1}{\lambda_0}\sqrt{|P_l| |\widehat Q_l^h|} + \frac{|P_l|}{\lambda_0}\right]$. From \Cref{eq:control_l1_to_l2}, it is also smaller than $6\phi_0 \gamma_u h|\widetilde P^{h}_l|\sqrt{\frac{|\overline Q^{*h}_l|}{\lambda_0}}$. As a consequence, we obtain the following upper bound:
\begin{align}\label{eq:var_P_minus_2}
	\begin{split}
		\| M(P'_l,Q^{*h}_l) - \overline M(P'_l,Q^{*h}_l)\|_F^2 \leq &\kappa_2\gamma_u^4 \left[\sqrt{(h^2p \lambda_0 \land 1)\frac{p|\overline Q_l^{*h}|}{\lambda_0^2}} + \frac{p}{\lambda_0}\right]\\
		&~\lor\frac{1}{2|\cH|}\| M(P_l) -  \overline M(P_l)\|_F^2 \enspace ,
	\end{split}
\end{align}

with $\kappa_2 = 12(\kappa_0 \lor \kappa_1)$, and using that $\phi_0 \leq \kappa_0\gamma_u$ and $|\widetilde P^h_l| \leq |P_l| \leq 2p$.

By the two previous cases on $l$, the inequality \Cref{eq:var_P_minus_2} is valid for any $l \in \cL_-(p)$.
Now, we decompose \Cref{eq:var_P_minus_1} into two terms, corresponding to the maximum in \Cref{eq:var_P_minus_2}.
First, since each $P_l$ is in at most one $\cP_-(p)$ for $p \in \cD_n$, we have
\begin{equation}\label{eq:var_P_minus_3}
	\sum_{(p,h) \in \cD_n \times \cH}\sum_{l \in \cL_-(p)}\frac{1}{2|\cH|}\| M(P_l) - \overline M(P_l)\|_F^2 \leq \frac{1}{2}\SN(\cP_-)\enspace .
\end{equation}
Secondly, we have that 
 \begin{align*}
	\kappa_2\gamma_u^4 \sum_{(p,h) \in \cD_n \times \cH} \sum_{l \in \cL_-(p)} &\left[\sqrt{(h^2p \lambda_0 \land 1)\frac{p|\overline Q_l^{*h}|}{\lambda_0^2}} + \frac{p}{\lambda_0}\right]\\ 
	&\leq \kappa_2\gamma_u^6\left[\max_{p,h} \sum_{l \in \cL_-(p)} \sqrt{(h^2p \lambda_0 \land 1)\frac{p|\overline Q_l^{h*}|}{\lambda_0^2}} + \frac{p}{\lambda_0} \right]\\
	&\overset{(a)}{\leq} 2\kappa_2\gamma_u^6\max_{p,h} \left[\frac{n}{\lambda_0} + \sqrt{(h^2p\lambda_0 \land 1)\frac{p|\cL(p)|\sum_{l \in \cL(p)}|\overline Q_l^{h*}|}{\lambda_0^2}} \right]\\
	&\overset{(b)}{\leq}4\kappa_2^2\gamma_u^7\max_{p,h} \left[\frac{n}{\lambda_0} + \sqrt{(h^2p\lambda_0 \land 1)\left(\frac{n^2d}{\lambda_0^2p}\land \left(\frac{nd}{p \lambda_0^3h^3} \lor \frac{nd}{\lambda_0^2h}\right)\right)} \right] \\
	&\leq 4\kappa_2^2\gamma_u^7\max_{p,h} \left[\frac{n}{\lambda_0} + nh\sqrt{\frac{d}{\lambda_0}}\land \sqrt{\frac{n^2dh^2}{\lambda_0}\land \frac{nd}{\lambda_0^2h}} \right] \\
	&\overset{(c)}{\leq}	4\kappa_2^2\gamma_u^7\left[\frac{n}{\lambda_0} + n\sqrt{\frac{d}{\lambda_0}} \land \frac{n^{2/3}\sqrt{d}}{\lambda_0^{5/6}} \right] \enspace .
\end{align*}

where in $(a)$ we used the Jensen inequality, in $(b)$ we used \Cref{lem:Control_Q} with $\phi = \phi_0$ together with the trivial inequality $\sum_{l \in \cL(p)}|\overline Q_l^{h*}| \leq nd/p$ and in $(c)$ the fact that $x \land y \leq x^{2/3}y^{1/3}$ and $h \leq 1$.

Finally, combining this last inequality with \Cref{eq:decomposition_square_norm}, \Cref{eq:ub_var_P_plus} and \Cref{eq:var_P_minus_3},

we obtain 
\begin{align*}
	\SN(\cP') &= \SN (\cP'_+) + \SN(\cP'_-) \\
	&\leq \left(1 - \frac{1}{\gamma_u^3}\right)\SN(\cP_+) + 4\kappa_2^2\gamma_u^7\left[\frac{n}{\lambda_0} + n\sqrt{\frac{d}{\lambda_0}} \land \frac{n^{2/3}\sqrt{d}}{\lambda_0^{5/6}} \right] \lor \left[\frac{1}{2}\SN(\cP_-)\right]\\
	&\leq \left[C\bar \gamma^7\left(\frac{n}{\lambda_0} + n\sqrt{\frac{d}{\lambda_0}} \land \frac{n^{2/3}\sqrt{d}}{\lambda_0^{5/6}} \right)\right] \lor \left[\left(1 - \frac{1}{\bar \gamma^3}\right)\SN(\cP)\right] \enspace ,
\end{align*}

where we recall that $\bar \gamma$ is defined in \Cref{eq:def_gamma} and satisfies $\bar \gamma \geq \gamma_u$. This concludes the proof of \Cref{prop:UB_on_square_norm}.
\section{Proof of the lemmas of \Cref{sec:proof_UB_on_square_norm}}\label{sec:proof_lemma_ub_square_norm}
Recall that we can write
\begin{equation}\label{eq:signal_noise_2}
	\noise = (B- \bbE[B]) \odot M + B \odot \widetilde E \enspace .
\end{equation}
where we recall that $\widetilde E = Y - \bbE[Y | B]$ and that $B$ is a matrix of Bernoulli random variables with parameter $\lambda_1$.

\begin{proof}[Proof of \Cref{lem:inclusion_Q}]
		Assume first that $\lambda_0 \leq 1$. Let us fix $l\in \{1, \dots, L\}$ and $h \in \cH$. We omit the dependence in $l$ in this proof to ease the notation and we write  $P$ for $P_l$.
		Let us define 
		\begin{equation}
			\noise'_{k}(a) := \frac{1}{|\cN_a|}\sum_{i \in \cN_a}\noise_{ik} - \frac{1}{|\cN_{-a}|}\sum_{i \in \cN_{-a}}\noise_{ik} \spaceAnd \nu(a) := |\cN_a| \land |\cN_{-a}| \enspace .
		\end{equation}
		Using \Cref{lem:concentration_proscal_bernstein} with a column matrix $W$ with coefficient in $\{0, \tfrac{1}{|\cN_a|}, -\frac{1}{|\cN_{-a}|}\}$ and a union bound over all $k \in [d]$ and $a \in [n]$, we have with probability at least $1-\delta/L$ that:
		\begin{equation}\label{eq:controle_varepsilon_echange}
			\frac{1}{\lambda_0}\abs{\noise'_{k}(a)} \leq \kappa'_0 \log(nd/\delta)\left[\sqrt{\frac{1}{\lambda_0\nu(a)}}  + \frac{1}{\lambda_0\nu(a)}\right] \enspace ,
		\end{equation} 
		for some numerical constant $\kappa'_0$. In what follows, we work under that \Cref{eq:controle_varepsilon_echange} holds true for all $a \in [n]$ and $k \in [d]$.

		{\bf First inclusion. } Let $k \in Q^{*}(\kappa_0\log(nd/\delta)h)$ with numerical constant $\kappa_0$ to be fixed later. Let $a' \geq 1$ be any integer such that $\nu(a')\geq 1/(\lambda_0 h^2)$. We have  \begin{align}
			\frac{1}{\lambda_0}\abs{\noise'_k(a')} \leq 2\kappa'_0\log(nd/\delta)h \enspace ,
		\end{align} 
		since we work under the event defined by \Cref{eq:controle_varepsilon_echange} and since $h^2 \leq h$.
		Then by consistency of the already constructed graph $\cG_{t,u}$ at the beginning of step $t$, $\cN_{a'}$ (resp. $\cN_{-a'}$) contains by definition \Cref{eq:neighborhood} only experts that are $\pi^*$-above (resp. below) all the experts of $P$. Since by assumption $k$ is in $Q^{*h}$, it holds that $\bDelta^*_k(a') \geq \bDelta^*_k(0) \geq \kappa_0\log(nd/\delta)h$ - see the definition \Cref{eq:definition_Q_pop} of $Q^{*h}$.
		Hence, recalling the signal-noise decomposition \Cref{eq:signal_noise_2}, we have that
		\begin{align}
			\frac{1}{\lambda_0}\widehat \bDelta_k(a') &= \frac{\lambda_1}{\lambda_0}\bDelta_k^*(a') + \frac{1}{\lambda_0}\noise'_k(a') \label{eq:noise_div_lambda_0}
			\geq \log(nd/\delta)((1- 1/e)\kappa_0 - 2\kappa'_0)h \enspace .
		\end{align}
		Choosing $\kappa_0 \geq 10\kappa'_0 + 1$, we obtain by definition \Cref{eq:stat_delta} that $\nu(\hat a_k(h)) \leq \frac{1}{\lambda_0 h^2}$ so that $k \in \widehat Q^h$.

		{\bf Second inclusion. } Let $k \in \widehat Q^h$, and $a' = a^*((\kappa_0\log(nd/\delta))^{-1}h)$ be as defined in \Cref{eq:definition_nu}. By definition, it holds that $\nu(a') \geq \kappa_0\log(nd/\delta)/(\lambda_0h^2) \geq \tfrac{1}{\lambda_0h^2}$. Hence, since $k \in \widehat Q^h$, we have by definition \Cref{eq:set_Q} that $\nu(\hat a_k(h)) \leq \frac{1}{\lambda_0h^2} \leq \nu(a')$, which implies in particular that $\hat a_k(h) \leq a'$. Then, by definition \Cref{eq:stat_delta} of $\hat a_k(h)$ we have that $\tfrac{1}{\lambda_0}\widehat \bDelta_k(a') \geq h$. Using the concentration inequality \Cref{eq:controle_varepsilon_echange} with $h' = (\kappa_0\log(nd/\delta))^{-1}h$ and the fact that $\lambda_0 \geq \lambda_1$ we obtain
		\begin{align}
			\bDelta^*_k(a') \geq h - \frac{2\kappa'_0}{\kappa_0}h \enspace ,
		\end{align}
		and we get the second inclusion by also choosing $\kappa_0 \geq 4(\kappa'_0 + 1)$.
	\end{proof}
		\begin{proof}[Proof of \Cref{lem:Control_Q}]
			For simplicity, we renumber $\cL(p) = (1, 2, \dots, L':=|\cL(p)|)$.
			Let us write $\nu(a,l) = |\cN_a(l)|\land |\cN_{-a}(l)|$ and $\Lambda = \floor{\frac{\phi^2}{p\lambda_0 h^2}} + 1$. We let $a^* := a^*(\phi^{-1}h, l)$ be as defined in \Cref{eq:definition_nu} so that for any $l$, $\nu(a^*, l) \geq \tfrac{\phi^2}{\lambda_0 h^2}$.
			
			By assumption of \Cref{prop:UB_on_square_norm}, it holds that $P_1 \overset{\cG}{\prec} P_2 \overset{\cG}{\prec} \dots \overset{\cG}{\prec} P_{|\cL(p)|}$ where we recall $\cG = \cG_{t,u}$ is the already constructed graph - see \Cref{subsec:step_1}. Hence it holds that $\rk_{\cG, i}(j) \geq \Lambda$ for any $i \in P_l$ and $j \in P_{l+\Lambda}$ - see \Cref{eq:rank} for the definition of $\rk$. Since there are at least $p\Lambda \geq \tfrac{\phi^2}{\lambda_0 h^2}$ experts in the union $P_{l+1} \cup \dots \cup P_{l+ \Lambda}$, we conclude that $a^* \leq \Lambda$, and that any expert in $\cN_{a^*}$ (resp. $\cN_{-a^*})$ is below the maximal expert of $P_{l+\Gamma}$ (resp. above) the minimal expert of $P_{l - \Lambda}$. 
			This implies that, upon writing $\overline \bDelta^*_k(l)$ for the difference of these maximal and minimal experts, we have by definition \Cref{eq:definition_Q_pop} of $\overline Q^{*h}$ that $\overline \bDelta^*_k(l)>h/2$ for all $k$ in $\overline Q^{*h}$. This implies in particular that

			\begin{equation}
				\sum_{l \in \cL(p)}  |\overline Q^{*h}_l(h, \phi)| \leq \sum_{k =1}^d \sum_{l \in \cL(p)}\1\{\overline \bDelta^*_k(l) \geq h/2\} \leq \frac{2}{h}\sum_{k =1}^d \sum_{l \in \cL(p)}\overline \bDelta^*_k(l) \leq (2\Lambda+1) \frac{2d}{h} \leq 6\frac{\Lambda d}{h} \enspace ,
			\end{equation}
			where in the last inequality we used the fact that $M_{i,k} \in [0, 1]$ and that the sequence $P_{l-\Lambda}, \dots, P_{l+\Lambda}$ is of length $2\Lambda + 1$, for any $l \in \cL(p)$.
		\end{proof}

\section{Proof of \Cref{prop:local_square_norm_reduction}}\label{sec:proof_of_prop_reduction}

	Let us fix any $l \in \{1, \dots, L\}$ and $h \in \cH$. Since $l, h$ and $\widehat Q^h_l$ are fixed in this proof, we simplify the notation and we write $(P',\widetilde P, Q) = (P'_l, \widetilde{P}^h_l,\widehat Q^h_l)$ and $M := M(\widetilde P,  Q)$ and $M(P') := M(P', Q)$. We also fix $\delta' = \delta/(L|\cH|)$, where we recall that $L \leq n$ is the number of groups.
	
	Let us assume that

\begin{equation}\label{eq:condition_Frobenius_acp}
	\| M(P') -  \overline M(P')\|_F^2 \geq ~\kappa_1 \gamma_u^4\left[\frac{1}{\lambda_0}\sqrt{|\widetilde P| |Q|} + \frac{|\widetilde P|}{\lambda_0}\right]\ , 
\end{equation}
for some constant $\kappa_1$ to be fixed later. In what follows, we show that under assumption \Cref{eq:condition_Frobenius_acp} for some large enough numerical constant $\kappa_1$, we necessarily have that the square norm of $P'$ is a contraction of the square norm of $P$, that is 
\begin{equation}\label{eq:obj_contraction_pca}
	\|M(P') - \overline M(P')\|_F^2 \leq \left(1 - \frac{1}{4\gamma_u^2}\right)\|M - \overline M\|_F^2 \enspace .
\end{equation}

\medskip

{\bf Step 1: control of the vector $\hat v$}

\medskip

First, the following lemma states that the first singular value of $(M - \overline{M})$ is, up to polylogarithmic terms, of the same order as its Frobenius norm. This is mainly due to the fact that the entries of $M$ lie in $[0,1]$ and that $M - \overline M$ is an isotonic matrix. 

\begin{lemma}[Lemma E.4 in \cite{pilliat2022optimal}]\label{lem:structure_isotonic_matrices}
	Assume that $\|M - \overline{M} \|_F \geq 2$. For any sets $\widetilde P$ and $Q$, we have
	\[
		\|M - \overline{M} \|^2_{\op} \geq \frac{4}{\gamma_u^2} \| M - \overline{M} \|^2_F \enspace .
	\]
\end{lemma}

This lemma was already stated and proved as Lemma E.4 in \cite{pilliat2022optimal}, recalling that $\gamma_u > \phi_{\lone} \geq 8\log(nd)$ -- see \Cref{eq:definition_phil1} and \Cref{eq:valid_gamma}.

Now, write
$\hat v = \argmax_{\|v\|_2 \leq 1} \Big[ \|v^T(Y^{(2)} - \overline{Y}^{(2)})\|_2^2 - \frac{1}{2}\| v^T(Y^{(2)} - \overline{Y}^{(2)} - Y^{(3)} + \overline{Y}^{(3)})\|_2^2\Big]$, 
where the argmax is taken over all $v$ in $\widetilde P$. 
\begin{lemma}\label{lem:concentration_pca}
	Assume that $\lambda_0 |\widetilde P| \geq 1$. There exists a numerical constant $\kappa'_0$ such that if
	\begin{equation}\label{eq:condition_pca}
		\|M - \overline{M}\|_{\op}^2 \geq \kappa'_0 \log^2(nd/\delta')  \left(\frac{1}{\lambda_0}\sqrt{|Q||\widetilde P|} + \frac{|\widetilde P|}{\lambda_0} \right) \enspace ,
	\end{equation}
	then, with probability higher than $1-\delta'$, we have  
	$$ \|\hat v^T \left(M - \overline M\right) \|_2^2 \geq \frac{1}{2}\| M - \overline M \|_{\op}^2 \enspace .$$
\end{lemma}
In light of Lemma~\ref{lem:structure_isotonic_matrices} and Condition~\eqref{eq:condition_Frobenius_acp}, the Condition~\eqref{eq:condition_pca}  in Lemma~\ref{lem:concentration_pca} is valid if we choose $\kappa_1$ in \Cref{prop:local_square_norm_reduction} such that $\kappa_1 \geq 16\kappa'_0$. Consequently, there exists an event of probability higher than $1-\delta'$ such that
\begin{equation}\label{eq:pca_signal_with_Frobenius}
	\|\hat v^T \left(M - \overline M\right) \|^2_2 \geq \frac{2}{\gamma_u^2}\|M - \overline M\|^2_F \enspace .
\end{equation}

\medskip

{\bf Step 2: control of the vector $\hat v_-$}

\medskip

Now remark that since $\|\hat v_{i}\|_2 = 1$, there are at most $\frac{1}{\lambda_0}$ of experts $i$ such that $\hat v_{i} > \sqrt{\lambda_0}$. Hence we have that 
\begin{align*}
	\|\hat v_-^T \left(M - \overline M\right) \|^2_2 &\geq \frac{2}{\gamma_u^2}\|M - \overline M\|^2_F - \sum_{i \in \widetilde P}\1_{\hat v_{i} > \sqrt{\lambda_0}} \|M_{i\cdot} - \overline m\|_2^2 \\
	&\overset{(a)}{\geq} \frac{2}{\gamma_u^2}\|M - \overline M\|^2_F - \frac{3\gamma_u}{\lambda_0}\sqrt{\frac{|\widehat Q|}{\lambda_0}}\\
	&\overset{(b)}{\geq} \frac{1}{\gamma_u^2}\|M - \overline M\|^2_F \enspace .
\end{align*}
$(a)$ comes from the fact that any expert in $\widetilde P$ satisfies \Cref{eq:indistinguishable} under the event of \Cref{lem:event_indistinguishable}. $(b)$ comes from Condition \Cref{eq:condition_Frobenius_acp} and the assumption that $\lambda_0|\widetilde P| \geq 1$.

\medskip

{\bf Step 3: control of the vector $\hat w$}

\medskip

Next, we show that a thresholded version of $\hat z= ( Y^{(4)}- \overline Y^{(4)} )^T \hat v_-$ is almost aligned  with $z^*= \lambda_1(M-\overline{M})^T \hat v_-$. We define the sets $S^*\subset Q$ and $\hat S\subset Q$ of questions by
\begin{equation}\label{eq:def_S_star_S_hat}
 S^* = \left\{ k \in Q ~:~ |z^*_k| \geq 2\gamma_u\sqrt{\lambda_0}\right\} \ ; \quad \hat S = \left\{ k \in Q ~:~ |\hat z_k| \geq \gamma_u\sqrt{\lambda_0}\right\} \enspace .
\end{equation}
$S^*$ stands for the collection of questions $k$ such that $z^*_k$ is large whereas $\hat S$ is the collection questions $k$ with large $\hat z_k$.  Finally, we consider the vectors  $w^*$ and $\hat{w}$ defined as theresholded versions of $z^*$ and $\hat z$ respectively, that is $w^*_k= z^*_k1_{k\in S^*}$ and $\hat{w}_k= \hat z_k\1_{k\in \hat S}$. Note that, up to the sign, $\hat{w}$ stands for the active coordinates computed in $\algoRefineLocally$, \Cref{line:direction} of \Cref{alg:refine_locally_sketch}.

\medskip
Recall that we assume that $\lambda_0 \leq 1$.
We write $v$ for any unit vector in $\mathbb{R}^{|\widetilde P|}$.
Let us apply \Cref{lem:concentration_proscal_bernstein} for each column $k \in Q$ of the noise matrix $\noise$ with the matrix $W$ equal to $v - (\tfrac{1}{|\widetilde P|}\sum_{i \in \widetilde P} v_i)\1_{\widetilde P}$ at column $k$ and $0$ elsewhere. We deduce that, for any fixed matrix $M$, any subsets $\widetilde P$ and $Q$, and any unit vector $v \in \bbR^{\widetilde P}$ such that $\|v\|_{\infty} \leq 2\sqrt{\lambda_0}$, we have
\begin{equation}\label{eq:control_noise_v_hat}
\P\left[\max_{k\in Q} \left|(v^T (\noise^{(3)} - \overline \noise^{(3)}))_k\right| \leq 100\log(2|Q|/\delta')\sqrt{\lambda_0}\right]\geq 1-\delta'\enspace .
\end{equation}
Observe that $\hat z=z^* +  (\noise^{(3)} - \overline \noise^{(3)})^T \hat v_-$.
Conditioning on $\hat v_-$, we deduce that, on an event of probability higher than $1-\delta'$, we have
\begin{equation}\label{eq:upper_noisec2}
	\|\hat z-z^* \|_{\infty}\leq 100\log(2|Q|/\delta')\sqrt{\lambda_0}\ \leq \frac{\gamma_u}{2}\sqrt{\lambda_0} \enspace ,
\end{equation}

where the last inequality comes from $\gamma_u > \phi_{\lone}$. Hence it holds that $S^*\subset \hat S$ and for $k\in \hat S$, we have $z^*_k/\hat z_k\in [1/2,2]$. Next, we shall prove that, under this event,  $\lambda_1\hat v_-^T (M - \overline M) \hat w/\|\hat w\|_2$ is large (in absolute value):
\[
	\lambda_1\abs{\hat v_-^T (M- \overline M) \hat w}
	= \abs{(z^*)^T \hat{w}}= \sum_{k\in \hat S} z^*_k \hat z_l\geq \frac{2}{5}\sum_{k\in \hat S} (z^*_k)^2 + (\hat z_l)^2\geq \frac{2}{5} [\|w^*\|_2^2 + \|\hat{w}\|_2^2]\geq \frac{4}{5}\| \hat w \|_2 \| w^* \|_2\ ,
\]
where we used in the first inequality that $z^*_k/\hat z_k\in [1/2,2]$ and in the second inequality that $S^*\subset \hat S$. Thus, it holds that
\begin{equation}\label{eq:lb_estw_truew}
	\lambda_1^2\abs{\hat v_-^T (M - \overline M) \frac{\hat w}{\|\hat w\|_2}}^2 \geq \frac{16}{25}\|w^* \|_2^2 \enspace .
\end{equation}
It remains to prove that $\|w^* \|_2$ is large enough. Writing $S^{*c}$ for the complementary of $S^*$ in $Q$, it holds that
\begin{equation}\label{eq:wstar_on_S}
	\|w^* \|_2^2 = \|z^*\|_2^2 - \sum_{k\in S^{*c}}(z^*_k)^2\ ,
\end{equation}
so that we need to upper bound the latter quantity. Write $z^*_{S^{*c}}= z^*-w^*$. Coming back to the definition of $z^*$,

\begin{align*}
\left[\sum_{k\in S^{*c}}(z^*_k)^2\right]^2&~= \left[\sum_{k\in S^{*c}}\lambda_1[\hat v_-^T(M - \overline M )]_k z^*_k\right]^2\\ 
&~\leq  \|\lambda_1\left(M - \overline M \right)z^*_{S^{*c}}\|^2_2=
\sum_{i \in \widetilde P}\left(\sum_{k \in S^{*c}}\lambda_1(M_{ik} - \overline m_k)z^*_k \right)^2 \\
&\overset{(a)}{\leq}   \frac{4\gamma_u^2}{|\widetilde  P|^2}\lambda_0 \sum_{i \in \widetilde  P}\left(\sum_{k \in S^{*c}}\sum_{j\in \widetilde  P}\lambda_1|M_{ik} - M_{jk}|\right)^2\\
&~\leq  \frac{4\gamma_u^2}{|\widetilde  P|^2}\lambda_0 \sum_{i \in \widetilde P}\left(\sum_{j\in \widetilde P}\lambda_1\|M_{i\cdot} - M_{j\cdot}\|_1\right)^2\\
&\overset{(b)}{\leq}  40 \gamma_u^4 \lambda_0^2|\widetilde P||Q|\\
&~\leq  \left[7\gamma_u^2 \lambda_0\sqrt{|\widetilde{P}||Q|}\right]^2 \leq \left[\frac{1}{2\gamma_u^2}\lambda_0^2\| M -  \overline M\|_F^2\right]^2\enspace  .
\end{align*}
In $(a)$, we used the definition of $S^*$. In $(b)$, we used \Cref{eq:indistinguishable} that holds true since we are under the event \Cref{lem:inclusion_Q} and $\lambda_0 |Q| \geq 1$. The last inequality comes from Condition \Cref{eq:condition_Frobenius_acp}, choosing $\kappa_1 \geq 14$.

\medskip

Recall that  $z^*= \hat v_-^T(M-\overline{M})$. Combining \Cref{eq:pca_signal_with_Frobenius} and \Cref{eq:wstar_on_S}, we deduce that 
\begin{equation}\label{eq:w_star_captures_significant_part}
	\|w^* \|_2^2 \geq \frac{1}{2\gamma_u^2}\lambda_0^2\|M - \overline M\|_F^2\ ,
\end{equation}
which, together with \Cref{eq:lb_estw_truew} and $\lambda_0 \geq \lambda_1$, yields
\begin{equation}\label{eq:energy_hat_w}
	\left\|(M - \overline M) \frac{\hat w}{\|\hat w\|_2} \right\|_2^2 \geq \abs{\hat v_-^T (M - \overline M) \frac{\hat w}{\|\hat w\|_2}}^2 \geq \frac{1}{2\gamma_u^2}\|M - \overline M\|_F^2 \enspace .
\end{equation}

\medskip

Write $\hat{w}^{(1)}$ and $\hat{w}^{(2)}$ the positive and negative parts of $\hat{w}$
respectively so that $\hat{w}= \hat{w}^{(1)}- \hat{w}^{(2)}$ and $\hat{w}^{+}= \hat{w}^{(1)}+ \hat{w}^{(2)}$. We obviously have $\|\hat{w}\|_2= \|\hat{w}^{+}\|_2$. Besides, if the rows of $M$ are ordered according to the oracle permutation, then
$(M - \overline M)\hat{w}^{(1)}$ and $(M - \overline M)\hat{w}^{(2)}$ are nondecreasing vectors with mean zero. It then follows from Harris' inequality that these two vectors have a nonegative inner product. We have proved that
\begin{equation}\label{eq:lower_former_signal}
	\left\|(M - \overline M) \frac{\hat w^+}{\|\hat w^+\|_2} \right\|_2^2\geq \left\|(M - \overline M) \frac{\hat w}{\|\hat w\|_2} \right\|_2^2  \geq  \frac{1}{2\gamma_u^2}\|M - \overline M\|_F^2 \enspace .
\end{equation}

{\bf Step 4: Showing that $\hat w$ satisfies Condition \Cref{eq:condition_w}}

Recall that we assume for simplicity that $\lambda_0 \leq 1$.
First we upper bound $\|w\|_\infty^2$ by using $(a)$ that $\hat z$ is close to $z^*$ with \Cref{eq:upper_noisec2}, $(b)$ that for any $k \in Q$, $v^TM_{\cdot k} \leq \|v\|_1$ and $(c)$ that $\lambda_0 |\widetilde P| \geq 1$:
\begin{align}
	\|\hat w\|_{\infty}^2 \overset{(a)}{\leq} 2\|z^*\|_{\infty}^2 + \gamma_u^2\lambda_0 \overset{(b)}{\leq} 2\lambda_0^2\|\hat v\|_1^2 + \gamma_u^2\lambda_0 \overset{(c)}{\leq} 3\gamma_u^2 \lambda_0^2 |\widetilde P| \enspace .
\end{align}

Secondly, we lower bound $\|w\|_2^2$ by using $(a)$ that $S^* \subset \hat S$ and that $z^*_k/\hat z_k \in [1/2, 2]$, $(b)$ that $\|w^*\|_2^2$ captures a significant part of the $L_2$ norm -see \Cref{eq:w_star_captures_significant_part}, and $(c)$ the Condition \Cref{eq:condition_Frobenius_acp} with $\kappa_1 \geq 24$:

\begin{align}
	\|\hat w\|_2^2 \overset{(a)}{\geq} \frac{1}{4}\|w^*\|_2^2 
	\overset{(b)}{\geq} \frac{1}{8\gamma_u^2}\lambda_0^2 \|M- \overline M \|_F^2 
	\overset{(c)}{\geq} 3\gamma_u^2 \lambda_0 |\widetilde P| \enspace .
\end{align}

We deduce that $\|\hat w\|_{\infty}^2 \leq \lambda_0 \|\hat w\|_2^2$, which is exactly Condition \Cref{eq:condition_w}. This shows that $\hat w^+$ is considered for the update \Cref{eq:update_W} in the final step of the procedure \Cref{line:update_2} of \Cref{alg:refine_locally_sketch}.

\medskip

{\bf Step 5: upper bound of the Frobenius-norm restricted to $P'$}

\medskip 

Equipped with this bound, we are now in position to show that the set $P'$ of experts obtained from $\widetilde{ P}$ when applying the pivoting algorithm with $\hat w^+/\|\hat w^+\|_2$ has a much smaller square norm.
By \Cref{lem:concentration_proscal_bernstein} used with the matrix $W$ equal to $0$ except at line $i$ where it is equal to the vector $\hat w^+/\|\hat w^+\|_2$, there exists an event of probability higher than $1-\delta'$ such that

$$\max_{i,j\in P'}\abs{\proscal<\noise_{i\cdot} - \noise_{j\cdot},\frac{\hat w^+}{\|\hat w^+\|_2}>}  \leq \phi_{\lone}\sqrt{\lambda_0} \leq \gamma_u\sqrt{\lambda_0}\ , $$
where we recall that $\phi_{l_1}$ is defined in \Cref{eq:definition_phil1}.
Hence, since the vector $\hat w$ is considered in the update \Cref{eq:update_W}, we have $\max_{i,j\in P'}\abs{\proscal<Y_{i\cdot} - Y_{j\cdot},\frac{\hat w^+}{\|\hat w^+\|_2}>} \leq \gamma_u\sqrt{\lambda_0}$ and
\begin{equation}
	\max_{i,j\in P'}\abs{\proscal<M_{i\cdot} - M_{j\cdot},\frac{\hat w^+}{\|\hat w^+\|_2}>} \leq 2\gamma_u \sqrt{\frac{1}{\lambda_0}} \enspace .
\end{equation}

By convexity, it follows that 
\[
	\left\|(M(P') - \overline M(P'))\tfrac{\hat w^+}{\|\hat w^+\|_2} \right\|_2^2 \leq 4\gamma_u^2\frac{1}{\lambda_0}|P'| \leq 4\gamma_u^2\frac{1}{\lambda_0}|\widetilde P|\ .
\]
In light of Condition \Cref{eq:condition_Frobenius_acp}, this quantity is small compared to $\| M - \overline{M}\|_F^2$:
\begin{equation}\label{eq:upper_new_signal}
	\| (M(P') - \overline M(P'))\tfrac{\hat w^+}{\|\hat w^+\|_2} \|_2^2 \leq\frac{1}{4\gamma_u^2} \| M - \overline{M}\|_F^2\ ,
\end{equation}
which together with~\eqref{eq:lower_former_signal} leads to
\begin{equation}\label{eq:upper_new_signal_2}
	\|(M - \overline M )\tfrac{\hat w^+}{\|\hat w^+\|_2} \|_2^2 - \|(M(P') - \overline M(P') )\tfrac{\hat w^+}{\|\hat w^+\|_2} \|_2^2\geq \frac{1}{4\gamma_u^2}  \| M - \overline{M}\|_F^2\enspace .
\end{equation}
Since $P'\subset \widetilde P$, we deduce that, for any vector $w'\in \mathbb{R}^q$, we have
$\|(M - \overline M)w' \|_2^2 \geq \|(M(P') - \overline M(P') )w' \|^2$. It then follows from the Pythagorean theorem that
\[
	\|M - \overline M\|_F^2 - \|M(P') - \overline M(P') \|_F^2 \geq \|(M - \overline M )\tfrac{\hat w^+}{\|\hat w^+\|_2} \|_2^2 - \|(M(P') - \overline M(P') )\tfrac{\hat w^+}{\|\hat w^+\|_2} \|_2^2 \enspace .
\]
Then, together with  \eqref{eq:upper_new_signal_2}, we arrive at
$$\|M(P') - \overline M(P')\|_F^2 \leq \left(1 - \frac{1}{4\gamma_u^2}\right)\|M - \overline M\|_F^2 \enspace .$$

We have shown that if \Cref{eq:condition_Frobenius_acp} is satisfied, then there is a contraction in the sense of \Cref{eq:obj_contraction_pca}. This in turn gives the upper bound \Cref{eq:local_square_norm_reduction} and it concludes the proof of \Cref{prop:local_square_norm_reduction}.
\medskip

\begin{proof}[Proof of \Cref{lem:concentration_pca}]

	Recall that we consider the case $\lambda_0 \leq 1$ and that the case $\lambda_0 \geq 1$ is discussed in \Cref{sec:full}.
	We start with the two following lemmas. To ease the notation, we assume in this proof that $\widetilde P = \{1, \dots, p\}$, that $Q = \{1, \dots q\}$. We only consider the matrices restricted to the sets $\widetilde P, Q$ and we write $\noise := \noise(\widetilde P, Q)$.
	Let us define $J = \1\1^T \in \bbR^{p \times p}$ the matrix with constant coefficients equals to $1$ and $A = (\bI_p - \frac{1}{p}J)$ be the projector on the orthogonal of $\1$, so that $\noise - \overline \noise = A\noise \in \bbR^{p \times q}$.
	The two following lemmas are direct consequences of \Cref{prop:concentration_bernstein_op}, and a discussion of the corresponding concenration inequality on random rectangular matrices can be found in \Cref{sec:concentration}. We state weaker concentration inequalities than what is proven in \Cref{prop:concentration_bernstein_op} in order to factorize the polylogarithmic factors and to ease the reading of the proof.
	\begin{lemma}\label{lem:concentration_quad}
		Assume that $\lambda_0 \leq 1$ and that $\lambda_0(p\lor q) \geq 1$. It holds with probability larger than $1-\delta'/4$ that
		\begin{equation*}
			\|\noise\noise^T - \bbE[\noise \noise^T]\|_{\op} \leq \kappa''_0\log^2(pq/\delta')\left[\lambda_0\sqrt{pq} + \lambda_0 p\right]\enspace .
		\end{equation*}
	\end{lemma}

	\begin{lemma}\label{lem:concentration_cross}
		Assume that $\lambda_0 \leq 1$ and that $\lambda_0(p\lor q) \geq 1$. With probability larger than $1-\delta'/4$, one has for any orthogonal projection $\Lambda \in \bbR^{q\times q}$ satisfying $\rank(\Lambda) \leq p$ that
		$$\|\Lambda\noise^T \noise \Lambda\|_{\op} \leq \kappa''_1\log^2(pq/\delta')\left[\lambda_0\sqrt{pq} + \lambda_0 p\right]\enspace ,$$
	\end{lemma}
	\begin{proof}[Proofs of \Cref{lem:concentration_quad} and \Cref{lem:concentration_cross}]
		First, we recall that for any $i,k$, we have that $\noise_{ik} = (B_{ik} - \lambda_1)M_{ik} + \widetilde E_{ik}$, and that $\widetilde E$ is an average of $1$-subGaussian random variables, as described in \Cref{eq:aggregated_noise}
		For any $u \geq 0$ we have
		\begin{equation}\label{eq:moment_bernstein}
			\bbE[\noise_{ik}^{2u}] \leq 3^u \bbE\left[ B_{ik} + \lambda_0^{2u} + \widetilde 	E_{ik}^{2u}\right] \leq 3^u \left(2\lambda_0 +  u! \bbE[e^{\widetilde E_{ik}^2}]	\right)
			\leq \frac{1}{2}u!\lambda_0 1000^u \enspace,
		\end{equation}
		where for the last inequality we used the following inequalities:
		\begin{align*}
			\bbE[e^{\widetilde E_{ik}^2}] \leq \sum_{u \geq 1} e^{-\lambda_0} \frac{\lambda_0^u}{u!} e^{1/u} \leq \lambda_0 e \enspace .
		\end{align*}
		Hence condition \Cref{eq:condition_moments_bernstein} is satisfied with $K= 1000$ and $\sigma^2=\lambda_0$ for the coefficients of $\noise$. We just apply \Cref{prop:concentration_bernstein_op} with $X = \noise$ for \Cref{lem:concentration_quad}. For \Cref{lem:concentration_cross}, we apply \Cref{prop:concentration_bernstein_op} with $X = \noise^T$ and we remark that 
		$\| \Lambda \noise^T \noise \Lambda \|_{\op}^2 \leq 2\| \Lambda \noise^T \noise - \bbE[\noise^T \noise] \Lambda \|_{\op}^2 + 2\| \bbE[\noise^T \noise] \|_{\op}^2$ together with the fact that $\|\bbE[\noise^T \noise] \|_{\op}^2 \leq c'\lambda_0p$ for some numerical constant $c'$.
	\end{proof}
	Remark that since we assume in \Cref{lem:concentration_pca} that $\lambda_0 p \geq 1$, it holds that $\sqrt{\lambda_0 p} \leq \lambda_0 p$ and $\sqrt{\lambda_0 q} \leq \lambda_0^2 \sqrt{pq}$, so that both upper bounds of \Cref{lem:concentration_quad} and \Cref{lem:concentration_cross} reduce - up to logarithmic factors - to $\lambda_0\sqrt{pq} + \lambda_0 p$. We write for short in the following 
	\begin{equation}\label{eq:def_F_short}
		F:= F(p,q,\lambda_0, \delta') = \log^2(pq/\delta')[\lambda_0\sqrt{pq} + \lambda_0 p] \enspace,
	\end{equation} 
	and $\kappa''_2 = 8(\kappa''_0 \lor \kappa''_1)$.
	
Now let us write
	\[
		AY = \lambda_1AM + A\noise\enspace ,
	\]
	so that, for any $v\in\mathbb{R}^p$, recalling that $AY = Y - \overline Y$,
	\begin{align*}
		\|v^TAY\|_2^2 = \lambda_1^2\|v^TAM\|_2^2 +  \|v^TA\noise \|_2^2 + 2\lambda_1\langle v^TA\noise, v^TAM \rangle\enspace ,
	\end{align*}
	which, in turn, implies that 
	\begin{align*}
		\Big|\|v^TAY\|_2^2 - \lambda_1^2\|v^TAM\|_2^2 - \E\big[\|v^TA\noise\|_2^2\big] \Big| 
		&~\leq \Big|\|v^TA\noise\|_2^2 - \E \big[\|v^TA\noise\|_2^2\big] \Big|
		+ 2\lambda_1 |v^TAM \noise^T (Av)|\\
		&\stackrel{(a)}{\leq} \|A(\noise\noise^T - \bbE[\noise \noise^T])A\|_{\op}  + 2 \lambda_1\|AM \noise^T\noise (AM)^T\|^{1/2}_{\op} \\
		&~\leq \|\noise\noise^T - \bbE[\noise \noise^T]\|_{\op} + 2\lambda_1\|AM\|_{\op}\| \Lambda \noise^T\noise \Lambda\|^{1/2}_{\op} \enspace ,
	\end{align*}
	Where we define  $\Lambda \in \bbR^{d\times d}$ as the orthogonal projector on the image of $\ker(AM)^{\perp}$ which is of rank less than $p$. For $(a)$, we used the fact that $A$ is contracting the operator norm as an orthogonal projector so that $\|Av\|_2 \leq 1$.
	We now apply \Cref{lem:concentration_quad} and \Cref{lem:concentration_cross} together with the fact that $\lambda_1 \leq \lambda_0$, and we obtain with probability at least $1-\delta'/2$ that

	\begin{equation}\label{eq:upper_1}
		\sup_{v \in \bbR^p, \|v\|=1}\Big|\|v^TAY\|_2^2 - \lambda_1^2\|v^TAM\|_2^2 - \E\big[\|v^TA\noise\|_2^2\big] \Big| \leq \kappa''_2 F + \lambda_1\|AM\|_{\op}\sqrt{\kappa''_2 F}\enspace .
	\end{equation}
	where $F$ is defined in \Cref{eq:def_F_short}.
In the same way, we have that, with probability larger than $1-\delta'/2$, 
	\begin{align*}
		\sup_{v\in \mathbb R^{p}:\ \|v\|_2\leq 1} \Big|\frac{1}{2}\|v^TA(Y-Y')\|_2^2 - \E\big[\|v^TA\noise\|_2^2\big] \Big|
	&= \frac{1}{2}\sup_{v\in \mathbb R^{p}:\ \|v\|_2\leq 1}\Big|\|v^TA(Y-Y')\|_2^2 - \mathbb E \|v^TA(Y-Y')\|_2^2 \Big| \\ 
	&\leq  \kappa''_3F\enspace,
	\end{align*}
	for some numerical constant $\kappa''_3$.
	Putting everything together we conclude that, on an event  of probability higher than $1-\delta'$, we have simultaneously for all	$v \in \mathbb R^{p}$ with $\|v\|_2\leq 1$ that
	\begin{align*}
		 \Big|\|v^TAY\|_2^2 - \|v^TAM\|_2^2 - \frac{1}{2}\|v^TA(Y-Y')\|_2^2 \Big| \leq \kappa''_4 F + \lambda_1\|AM\|_{\op}\sqrt{\kappa''_4 F} \enspace ,
	\end{align*}
	with $\kappa''_4 = \kappa''_2 \lor \kappa''_3$.
	Choosing the numerical constant $\kappa'_0$ of \Cref{lem:concentration_pca} such that $\kappa'_0 \geq 4\cdot 16 (1-1/e)^{-1}\kappa''_4$ we have
	\begin{equation*}
		\lambda_1^2\|AM\|_{\op}^2 \geq 4\cdot 16 \kappa''_4 F\enspace ,
	\end{equation*}
	since it holds that $\lambda_1 \geq (1-1/e)\lambda_0$.
	We deduce that on the same
	event:
	\[
		\sup_{v\in \mathbb{R}^p: \ \|v\|_2\leq 1} \Big|\|v^TAY\|_2^2 - \|v^TAM\|_2^2 - \frac{1}{2}\|v^TA(Y-Y')\|_2^2 \Big|\leq \frac{1}{4}\|AM\|_{\op}^2\enspace .\]
	Writing  $\psi(v)= \big|\|v^T(Y - \overline{Y})\|_2^2  - \frac{1}{2}\|v^TA(Y-Y')\|_2^2\big|$, we deduce that, for $v$ such that $\|v^TAM\|_{2}^2 = \|AM\|_{\op}^2$, we have $\Psi(v)\geq \tfrac{3}{4}\|AM\|_{\op}^2$, whereas, for $v$ such that $\|v^TAM\|_{2}^2< \frac{1}{2} \|AM\|_{\op}^2$, we have $\Psi(v)<\tfrac{3}{4}\|AM\|_{\op}^2$. We conclude that $\hat v$ satisfies $\|\hat v^TAM\|_{2}^2> \frac{1}{2} \|AM\|_{\op}^2$ with probability at least $1-\delta'$.

\end{proof}

\section{Proof of \Cref{th:UB} when $\lambda_0 \geq 1$}\label{sec:full}

The aim of this section is to provide an extension of the proof of \Cref{th:UB} to the case $\lambda_0 \geq 1$.
Recall that we fix $\delta$ to be a small probability the proof of \Cref{th:UB}, and that $E$ and $\widetilde E$ are the matrices defined in \Cref{eq:two_possible_noises} and \Cref{eq:aggregated_noise} by 
\begin{equation*}
	\widetilde E^{(s)}_{ik} = \sum_{t \in N^{(s)}} \tfrac{\varepsilon_t}{\br^{(s)}_{ik}\lor 1}\1\{x_t = (i,k)\} \spaceAnd \noise^{(s)}_{ik} = (B^{(s)}_{ik} - \lambda_1)M + B^{(s)}_{ik}\widetilde E^{(s)}_{ik}\enspace .
\end{equation*}
In what follows, we consider the two subcases where $\lambda_0 > 16\log(5nd/\delta)$ or $\lambda_0 \leq 16\log(5nd/\delta)$, which essentialy rely on the two following ideas:

\begin{itemize}
	\item If $\lambda_0 \leq 16\log(5nd/\delta)$, we use the fact that the coefficients of $\noise$ defined in \Cref{eq:two_possible_noises} are $5$-subGaussian together with the same signal-noise decomposition $Y = \lambda_1 M + \noise$ as in the proofs when $\lambda_0 \leq 1$. The difference from the case $\lambda_0 \leq 1$ lies in the application of subGaussian inequalities of $E_{ik}$ instead of Bernstein inequalities as in \Cref{eq:concentration_bernstein}.
	\item If $\lambda_0 > 16\log(5nd/\delta)$, we show that the event $\{\br_{ik}^{(s)}\geq \lambda_0 /2\}$ holds true for all $i,k,s$ with high probability. Working conditionally to this event, we use the decomposition $Y = M + \widetilde E$ and we show that the noise $\widetilde E$ has $\frac{2}{\lambda_0}$-subGaussian independent coefficients. The rationale behind using $\widetilde E$ when $\lambda_0$ is large is that $\widetilde E_{ik}$ takes advantage of the mean of $2/\lambda_0$ subGaussian variables with high probability.
\end{itemize}
Let $\br_{\min}^{(s)} = \min_{i,k}\br_{ik}^{(s)}$ be the minimum number of observation at positions $(i,k)$ in $N_s$ - see \Cref{eq:batches}.
In the case $\lambda_0 > 16\log(5nd/\delta)$, the following lemma states that with high probability, we observe all the coefficients for all sample $s$ in the full observation regime. 

\begin{lemma}\label{lem:LB_wmin}
	Assume that $\lambda_0 \geq 16\log(5nd/\delta)$. The event $\{\br^{(s)}_{\min} \geq \lambda_0/2 \}$ holds simultaneously for all sample $s$ with probability at least $1 - 5T\delta$.
\end{lemma}
\begin{proof}[Proof of \Cref{lem:LB_wmin}]
	We apply Chernoff's inequality - see e.g. section 2.2 of \cite{massart2007concentration} - to derive that for any $i,k$
	\begin{equation}
		\bbP(\br^{(s)}_{ik} \leq \lambda_0/2) \leq \exp(-\frac{1}{8} \lambda_0) \leq \delta/(nd) \enspace ,
	\end{equation}
	where we use the inequality $(1- \log(2))/2 \geq 1/8$. We conclude with a union bound over all coefficients in $[n] \times [d]$ and all $5T$ samples.
\end{proof}

Let us now omit the dependence of $E$ and $\widetilde E$ in the sample $s$. In what follows, use that the coefficients of $\noise$ are $5$-subGaussian, which is a consequence of the fact that $\noise_{ik}$ is the sum of a centered variable bounded by $1$ and a $1$-subgaussian random variable $\widetilde E_{ik}$, so that by Cauchy-Schwarz and the Hoeffding inequality we have
\begin{equation}\label{eq:1_subgaussian}
	\bbE[\exp(x \noise_{ik})] \leq \sqrt{\exp(4x^2/8)}\sqrt{\exp(4x^2/2)} = \exp(5/4x^2)\enspace .
\end{equation}

Under the event of \Cref{lem:LB_wmin}, we use that $\widetilde E_{ik}$ is $\lambda_0/2$-subGaussian, as an average of at least $2/\lambda_0$ random variables that are $1$-subGaussians:
\begin{equation}\label{eq:lambda_subgaussian}
	\bbE[\exp(x \widetilde E_{ik})] \leq \exp(\tfrac{1}{\lambda_0} x^2) \enspace ,
\end{equation}

\subsection{Adjustements for the general analysis}

We first make the changes that should be done in \Cref{sec:general_analysis} to have a proper proof in the case $\lambda_0 \geq 1$.

If $\lambda_0 \in [1, 16\log(5nd/\delta)]$, we simply replace $\lambda_0$ by $1/\lambda_0$ in the upper bound of \Cref{eq:condition_phil1} for the event $\xi$ in \Cref{lem:concentration_l1}. In the proof of the restated \Cref{lem:concentration_l1}, we can replace the inequality \Cref{eq:concentration_bernstein} by 

\begin{equation}\label{eq:subgaussian_1}
	|\proscal<\noise, W>| \leq \sqrt{10\|W\|_F^2\log\left(\frac{2}{\delta'}\right)} \enspace ,
\end{equation}
for any matrix $W \in \bbR^{n \times d}$, with probability at least $1-\delta'$.
We can then obtain $1/\lambda_0$ instead of $\lambda_0$ simply by using that $\phi_{\lone}/\sqrt{\lambda_0} \geq \sqrt{\phi_{\lone}}$, recalling that $\phi_{\lone}$ is defined in \Cref{eq:definition_phil1}.

If $\lambda_0 > 16\log(5nd/\delta)$, we say that we are under event $\xi$ if the event of \Cref{lem:LB_wmin} holds and \Cref{eq:condition_phil1} holds for all pairs $(Q,w)$, replacing $\noise$ by $\widetilde E$, and $\lambda_0$ by $1/\lambda_0$.
The proof of the new version of \Cref{lem:concentration_l1} lies in the Hoeffding inequality applied to $\widetilde E$ under the event of \Cref{lem:LB_wmin}, leading to the subsequent equation:
\begin{equation}\label{eq:hoeffding_E_tilde}
	|\proscal<\widetilde E, W>| \leq \sqrt{\frac{4\|W\|_F^2}{\lambda_0}\log\left(\frac{2}{\delta'}\right)} \enspace ,
\end{equation}
for any matrix $W \in \bbR^{n \times d}$, with probability at least $1-\delta'$. This equation then replaces \Cref{eq:concentration_bernstein}.

\subsection{Adjustments to the proofs of \Cref{prop:UB_on_square_norm}}

We now adapt the proofs in \Cref{sec:proof_UB_on_square_norm} of \Cref{prop:UB_on_square_norm} to the case $\lambda_0 \geq 1$.

All the lemmas of \Cref{sec:proof_UB_on_square_norm} can be stated as is for any $\lambda_0 \geq 1$, and the only adjustments concern the proofs of \Cref{lem:inclusion_Q}, \Cref{lem:event_indistinguishable} and \Cref{prop:local_square_norm_reduction}.

\subsubsection{Adjustments in the proofs of \Cref{lem:inclusion_Q} and \Cref{lem:event_indistinguishable}}
Consider the proof of \Cref{lem:inclusion_Q}.
First, if $\lambda_0 \geq 16 \log(5nd/\delta)$, we place ourselves under the event \Cref{lem:LB_wmin} and replace $\lambda_1$ by $1$ and all the $\noise$ by $\widetilde E$. Instead of inequality \Cref{eq:controle_varepsilon_echange}, we use the fact that the coefficients of $\widetilde E$ are $2/\lambda_0$-subGaussian - see \Cref{eq:lambda_subgaussian} - leading to the following inequality with probability at least $1 - \delta$:
\begin{equation}
	\abs{\widetilde E_{k}(a)} := \abs{\frac{1}{|\cN_a|}\sum_{i \in \cN_a}\widetilde E_{ik} - \frac{1}{|\cN_{-a}|}\sum_{i \in \cN_{-a}}\widetilde E_{ik}} \leq \kappa'_0 \log(nd/\delta)\sqrt{\frac{1}{\lambda_0\nu(a)}} \enspace ,
\end{equation}
for some numerical constant $\kappa'_0$. The rest of the proof remains unchanged.

\medskip

If $\lambda_0 \in [1, 16\log(5nd/\delta)]$, we use the fact that $\noise$ has $5$-subGaussians coefficients - see \Cref{eq:1_subgaussian} and we do not divide by $\lambda_0$ in \Cref{eq:noise_div_lambda_0} - see the definition of $\widehat \bDelta$ \Cref{eq:stat_delta}.

Concerning \Cref{lem:event_indistinguishable}, the adjustments are the same as for \Cref{lem:concentration_l1}, namely working under the event of \Cref{lem:LB_wmin} and we replacing $\noise$ by $\widetilde E$, $\lambda_0$ by $1/\lambda_0$ and $\lambda_1$ by $1$ if $\lambda_0 \geq 16\log(5nd/\delta)$, and using the fact that the coefficient of $\noise$ are $5$-subGaussians - see \Cref{eq:1_subgaussian} if $\lambda_0 \in [1, 16\log(5nd/\delta)]$.

\subsubsection{Adjustments in the proof of \Cref{prop:local_square_norm_reduction}}

We now adapt the proofs in \Cref{sec:proof_of_prop_reduction} of \Cref{prop:local_square_norm_reduction} to the case $\lambda_0 \geq 1$.
First, \Cref{lem:concentration_pca} can be stated as is, and its proof when $\lambda_0 \geq 1$ is directly implied by Lemma E.5 in \cite{pilliat2022optimal} with $\Theta := M$ either conditionally on \Cref{lem:LB_wmin} with noise $N:= \widetilde E$ and $\zeta^2 := 2/\lambda_0$ when $\lambda_0 \geq 16\log(5nd/\delta)$ or with noise $N:= \noise$ and $\zeta^2 := 5$ when $\lambda_0 \leq 16\log(5nd/\delta)$.

\medskip

Secondly, remark that if $\lambda_0 \geq 1$, it holds that $\hat v_- = \hat v$ and that Condition \Cref{eq:condition_w} on $\hat w$ is automatically satisfied, so that step 2 and step 4 can be removed from the proof in that case. For Step 3 and 5, we do the following adjustments:

If $\lambda_0 \in [1, 16\log(5nd/\delta)]$, the proof remains unchanged except that we use that the coefficients of $\noise$ are $5$-subGaussian -see \Cref{eq:1_subgaussian}.

If $\lambda_0 \geq 16\log(5nd/\delta)$, we work conditionnally on the event of \Cref{lem:LB_wmin} and we replace $\lambda_1$ by $1$ and $\noise$ by $\widetilde E$. The subgaussian concentration bound on $\widetilde E$ \Cref{eq:hoeffding_E_tilde} allows us to replace $\lambda_0$ by $\frac{1}{\lambda_0}$ in the equations from \Cref{eq:def_S_star_S_hat} to \Cref{eq:w_star_captures_significant_part}.

\section{Proof of Corollaries \ref{cor:reconstruction_UB} and \ref{cor:ub_reco_biso}}
\begin{proof}[Proof of \Cref{cor:reconstruction_UB}]
	Assume that $\pi^* = \mathrm{id}$ for simplicity. Let $P_{\iso}$ be the projector on the set of isotonic matrices, and $E' = Y^{(2)}_{\hat \pi^{-1}} - M_{\hat \pi^{-1}}$ so that $\hat M_{\iso} = P_{\iso} (M_{\hat \pi^{-1}} + E')$. 
	Remark that the loss can be decomposed as 
	\begin{equation*}
		\|(\hat M_{\iso})_{\hat \pi} - M\|_F^2 = \|P_{\iso}M_{\hat \pi^{-1}} - P_{\iso}M_{} + P_{\iso}(M_{} + E') - M_{} + M_{}- M_{\hat \pi^{-1}}\|_F^2\enspace .
	\end{equation*}
	Using the non-expansiveness of $P_{\iso}$ and the triangular inequality as in the proof of proposition 3.3 of \cite{mao2020towards}, we deduce that 
	\begin{equation}\label{eq:triangular_reco}
		\|\hat M_{\iso} - M\|_F^2 \leq 4\|M_{\hat \pi^{-1}} - M\|_F^2 + 2\|P_{\iso}(M + E') - M\|_F^2 \enspace .
	\end{equation}
	Since the projection of $M+E'$ on isotonic matrices is equal to the columnwise projection on isotonic vectors, it holds that $\sup_{M \in \bbC_{\iso}(n,d)}\bbE\|P_{\iso}(M + E') - M\|_F^2 = d \sup_{M \in \bbC(n,1)}\bbE\|P_{\iso}(M_{\cdot 1} + E'_{\cdot 1}) - M_{\cdot 1}\|_F^2 $, where we also use the notation $P_{\iso}$ for the projector on isotonic vectors. The rate of estimation in $L_2$ norm of an isotonic vector with bounded total variation partial observation can be found in \cite{zhang2002risk}, with $V := 1$ and $\sigma^2 := 1/\lambda$. Hence, we obtain that $\sup_{M \in \bbC(n,1)}\bbE\|P_{\iso}(M_{\cdot 1} + E'_{\cdot 1}) - M_{\cdot 1}\|_F^2 \leq C_1 n^{1/3}/\lambda^{2/3}$. Upper bounding the first term in~\eqref{eq:triangular_reco} with a quantity of order $\rho_{\perm} \leq 2\rho_{\reco}$ by \Cref{th:UB} concludes the proof.

\end{proof}

\begin{proof}[Proof of \Cref{cor:ub_reco_biso}]
	We follow the same steps as in \Cref{cor:reconstruction_UB}. Assume that $\pi^* = \eta^* = \mathrm{id}$, $E'=Y^{(3)}_{\hat\pi^{-1}\hat\eta^{-1}} - M$, and let $P_{\biso}$ be the projector on bi-isotonic matrices. We have that 
	\begin{equation}\label{eq:ineq_trig_biso}
		\|(\hat M_{\biso})_{\hat \pi \hat \eta} - M\|_F^2 \leq 4\|M_{\hat \pi^{-1}\hat\eta^{-1}} - M\|_F^2 + 2\|P_{\biso}(M + E') - M\|_F^2 \enspace .
	\end{equation}
$M$ is isotonic in both directions so that we can apply \Cref{th:UB} in rows and columns. After the first two steps of the above procedure, we obtain two estimator $\hat \pi, \hat \eta$ that satisfy
	\begin{equation}\label{eq:ub_liu_moitra_2_perm}
		\sup_{\stackrel{\pi^*, \eta^* \in \Pi_n}{M:\,  M_{\pi^{*-1} \eta^{*-1}}\in \mathbb{C}_{\mathrm{biso}}}} \E\left[\|M_{\hat{\pi}^{-1}\hat{\eta}^{-1}}- M_{\pi^{*-1}\eta^{*-1}}\|_F^2\right] \leq C''\log^{C''}(n)n^{7/6}\lambda^{-5/6} \enspace .
		\end{equation}
	The second term of \Cref{eq:ineq_trig_biso} is the risk of a bi-isotonic regression by least square, and is smaller than $n/\lambda \leq n^{7/6}\lambda^{-5/6}$ - see e.g. \cite{mao2020towards}.
\end{proof}

\section{Proof of the minimax lower bound}

\begin{proof}[Proof of \Cref{th:lower_bound_poisson}]

Since $\rho_{\perm}(n,d,\lambda)$ is nondecreasing with $n$ and $d$, we can assume without loss of generality that both $n$ and $d$ express as a power of $2$. 

The following proof is strongly related to the proof of Theorem 4.1 in \cite{pilliat2022optimal}. While a worst case distribution is defined on the set of matrices that have nondecreasing rows and nondecreasing columns in \cite{pilliat2022optimal}, we aim here at defining a worst case distribution on matrices only have nondecreasing columns. Since the isotonic model is less constrained than the bi-isotonic model studied in \cite{pilliat2022optimal}, the permutation estimation problem is statistically harder, and the lower bound has a greater order of magnitude.

As in \cite{pilliat2022optimal}, the general idea is first to build a collection of prior $\nu_{{\bf G}}$ indexed by some ${\bf G}\in \boldsymbol{\mathcal{G}}$ on $M$, then to reduce the problem to smaller problems and finally to specify the prior in function of the regime in $n,d$ and $\lambda$. By assumption, the data $y_t$ is distributed as a normal random variable with mean $M_{x_t}$ and variance $1$, conditionally on $M$ and $x_t$. We write as in \cite{pilliat2022optimal} $\mathbf{P}_{\bf G}^{(\mathbf{full})}$ and  $\mathbf{E}_{\bf G}^{(\mathbf{full})}$ the corresponding marginal probability distributions and expectations on the data $(x_t,y_t)$.
Our starting point is the fact that the minimax risk~\eqref{eq:minimax_risk_perm} is higher than the worst Bayesian risk:  
\begin{equation}\label{eq:starting_point}
\cR^*_{\perm}(n,d,\lambda) \geq \inf_{\hat{\pi}}\sup_{{\bf G}\in  \mathbfcal{G}}\mathbf{E}_{\bf G}^{\mathbf{full}}\left[\|M_{\hat{\pi}^{-1}}- M_{\pi^{*-1}}\|_F^2\right]\ . 
\end{equation}

\medskip

{\bf Step 1: Construction of the prior distribution on $M$}

\medskip

Let $p \in \{2, \dots, n\}$ and $q \in [d]$ be two powers of $2$ to be fixed later, and $\overline G^{(\iota)} :=[(\iota -1) p + 1, \iota p]$, for $\iota \in \{1, \dots, n/p\}$. The general idea is to build a simple prior distribution on isotonic matrices in $\bbR^{\overline G^{(\iota)} \times d}$, and to derive a prior distribution on isotonic matrices in $\bbR^{n \times d}$ by combining $n/p$ independent simple prior distributions defined on each strip $\bbR^{\overline G^{(\iota)} \times d}$.

Let $w \in \bbR^n$ be a vector that is constant on each group $\overline G^{(\iota)} =[(\iota -1) p + 1, \iota p]$ and that has linearly nondecreasing steps:
\begin{equation}\label{eq:vector_w_lb}
	w_i = \floor{\frac{i}{p}} \frac{p}{4n} \in [0, 1/4] \enspace .
\end{equation}
Letting $\1_{[d]}$ be constant equal to $1$ in $\bbR^d$, we define
\begin{equation} \label{eq:definition_M_lower_bound}
M = w \1_{[d]}^T + \frac{\upsilon}{\sqrt{p\lambda}} B^{(\mathbf{full})}\ , 
\end{equation}
where the random matrix $B^{(\mathbf{full})}\in \{0,1\}^{n\times d}$ is defined as in \cite{pilliat2022optimal}. We recall the definition of its distribution in what follows for the sake of completeness.

Consider a collection $\mathcal{G}$ of subsets of $[p]$ with size $p/2$ that are well-separated in symmetric difference as defined by the following lemma.
\begin{lemma}\label{lem:packing}
There exists a numerical constant $c_0$ such that the following holds for any even integer $p$. 
There exists a collection $\mathcal G$ of subsets of $[p]$ with size $p/2$ which satisfies $\log(|\cG|)\geq c_0 |p|$ and whose elements are $p/4$-separated, that is 
$|G_1 \Delta G_2| \geq p/4$ for any $G_1\neq G_2$. 
\end{lemma}
The above result is stated as is in \cite{pilliat2022optimal} and is a consequence of Varshamov-Gilbert's lemma - see e.g.~\cite{tsybakov}.

\medskip 

For each $\iota \in [n/p]$, we fix a subset $G^{(\iota)}$ from $\mathcal{G}$, and its translation $G^{t(\iota)}=\{(\iota-1)p + x: x\in G^{(\iota)}\} \subset \overline G^{(\iota)}$. The experts of $G^{t(\iota)}$ will correspond the $p/2$ experts in $\overline G^{(\iota)}$ that are above the $p/2$ experts in $\overline G^{(\iota)}\setminus G^{t(\iota)}$. We write ${\bf G}= (G^{t(1)}, \ldots, G^{t(n/p)})$ and $\boldsymbol{\mathcal{G}}$ the corresponding collection of all possible ${\bf G}$. Given any such ${\bf G}$, we shall define a distribution $\nu_{{\bf G}}$ of $B^{(\mathbf{full})}$, and equivalently of $M$ by \Cref{eq:definition_M_lower_bound}.

 For $\iota\in [n/p]$, we sample uniformly a subset $Q^{(\iota)}$ of $q$ questions among the $d$ columns. In each of these $q$ columns, the corresponding rows of $B^{(\mathbf{full})}$ are equal to one. More formally, we have
\begin{equation}\label{eq:definition_B_Full}
B^{(\mathbf{full})} = \sum_{\iota=1}^{n/p}\mathbf 1_{G^{t(\iota)}}\1_{Q^{(\iota)}} \enspace .
\end{equation}
As mentioned above, the definition of $B^{(\mathbf{full})}$ is the same as in \cite{pilliat2022optimal}, if $\tilde d$ is set to be equal to $d$. They define a block constant constant matrix when $\tilde d < d$ to get an appropriate prior distribution for bi-isotonic matrices, but we do not need to do that here since we do not put any constraint on the rows of $M$.

\medskip 
The matrix $M$ defined in~\eqref{eq:definition_B_Full} is isotonic up to a permutation of its rows and has coefficients in $[0,1]$, if the following inequality is satisfied.
\begin{equation}\label{eq:condition_v}
\frac{\upsilon}{\sqrt{p\lambda}} \leq \frac{p}{8n} \ . 
\end{equation}

This constraint is strictly weaker than its counterpart (149) in \cite{pilliat2022optimal}, and this is precisely what makes the lower bound in the isotonic setting larger than the lower bound in the bi-isotonic setting of \cite{pilliat2022optimal}. Our purpose will be to wisely choose parameters $p,q$ and $\upsilon > 0$ to maximize the Bayesian risk \Cref{eq:starting_point} with $\nu_{\bG}$.

\medskip

{\bf Step 2: Problem Reduction} 

\medskip

In what follows, we use the same reduction arguments as in \cite{pilliat2022optimal}. Using the notation of \cite{pilliat2022optimal}, we write
$\mathbf{P}_{\bf G}^{(\mathbf{full})}$ and  $\mathbf{E}_{\bG}^{(\mathbf{full})}$ for the probability distribution and corresponding expectation of the data $(x_t, y_t)$, when $M$ is sampled according to $\nu_{\bG}$. Since the distribution of the rows of $M$ in $\overline G^{t(\iota)}$ only depend on $G^{t(\iota)}$, we write $\nu_{G^{t(\iota)}}$ for the distribution of these rows. We also write $\mathbf{P}^{(\mathbf{full})}_{G^{t(\iota)}}$ and $\mathbf{E}^{(\mathbf{full})}_{G^{t(\iota)}}$ for the corresponding marginal distribution and corresponding expectation of the observations $(x_t,y_t)$ such that $(x_t)_1\in \overline{G}^{t(\iota)}$. By the poissonization trick, the distribution $\mathbf{P}^{(\mathbf{full})}_{\bf G}$ is a product measure of $\mathbf{P}^{(\mathbf{full})}_{G^{t(\iota)}}$ for $\iota=1,\ldots, n/p$.

Let $\tilde \pi$ be any estimator of $\pi^*$. Let us provide more details than \cite{pilliat2022optimal} to prove that $\tilde \pi$ can be modified into an estimator $\hat \pi$ satisfying $\hat \pi(\overline G^{(\iota)}) = \overline G^{(\iota)}$ for all $\iota = 1, \dots, n/p$, and reducing the loss $\|M_{\hat \pi^{-1}}- M_{\pi^{*-1}}\|_F^2 \leq \|M_{\tilde \pi^{-1}}- M_{\pi^{*-1}}\|_F^2$ almost surely, for all possible prior $\nu_{\bG}$. For that purpose, we introduce 
\begin{equation*}
	N(\pi) = \sum_{\iota = 1}^{n/p}\sum_{i \in \overline G^{(\iota)}}\1\{i \not \in \overline G^{(\iota)}\} \enspace .
\end{equation*}

If $N(\tilde \pi) > 0$, then there exists $\iota_0$ and $i_0 \in \overline G^{(\iota_0)}$ such that $\tilde \pi(i_0) \in \overline G^{(\iota_1)}$ with $\iota_1 \neq \iota_0$. Then, $\tilde \pi$ being a permutation, we consider its associated cycle containing $i_0$, which we denote by $(i_1, \dots, i_{K})$. Let $(i'_1, i'_2, \dots, i'_L)$ be the elements of this cycle such that $\tilde \pi(i'_l) \not \in \overline G^{(\iota_l)}$, where $\iota_l$ satisfies $i'_l \in \overline G^{(\iota_l)}$. Then it holds that for any $l = 1, \dots, L-1$, $\tilde \pi(i'_l) \in \overline G^{(\iota_{l+1})}$, and $\tilde \pi(i'_L) \in \overline G^{(\iota_{1})}$. We now define $\tilde \pi'(i) = \tilde \pi(i)$ for all $i$, except on the cycle $(i'_1, \dots, i'_L)$ where we set $\tilde \pi'(i'_l) = \tilde \pi(i'_{l-1})$. Then, we easily check that $N(\tilde \pi') = N(\tilde \pi) - L < N(\tilde \pi)$, and that $\|M_{\tilde \pi'^{-1}}- M_{\pi^{*-1}}\|_F^2 \leq \|M_{\tilde \pi^{-1}}- M_{\pi^{*-1}}\|_F^2$ if condition \Cref{eq:condition_v} is satisfied.

We can therefore restrict ourselves to estimators $\hat \pi$ such that $\hat{\pi}(\overline{G}^{(\iota)})=\overline{G}^{(\iota)}$ for all $\iota$. There is however still another catch to obtain the same lines as in \cite{pilliat2022optimal}. Indeed, the restriction $\hat \pi^{(\iota)}$ of $\hat \pi$ to $\overline{G}^{(\iota)}$ is measurable with respect to the observation $Y$, but not necessarily to $Y(\overline{G}^{(\iota)})$. Still, this restriction can be writen as $\hat \pi^{(\iota)} = \hat \pi^{(\iota)}(Y(\overline{G}^{(\iota)}), Y([n]\setminus\overline{G}^{(\iota)}))$, and, for any $\alpha > 0$, there exists $y^{*(\iota)}(\alpha)$ such that
\begin{equation*}
	\mathbf{E}^{(\mathbf{full})}_{\bf G}\left[\|M_{\hat \pi^{(\iota)-1}} - M_{\pi^{*-1}}\|_F^2\right] \geq \mathbf{E}^{(\mathbf{full})}_{\bf G}\left[\|M_{\bar \pi^{(\iota)-1}(\alpha)} - M_{\pi^{*-1}}\|_F^2\right] - \alpha \enspace ,
\end{equation*}
where $\bar \pi^{(\iota)} := \hat \pi^{(\iota)}(Y(\overline{G}^{(\iota)}), y^{*(\iota)}(\alpha))$ is measurable with respect to $Y(\overline{G}^{(\iota)})$. Since it is possible such a stable estimator for any $\alpha > 0$, we finally obtain the inequality
\begin{align*}
	\cR^*_{\perm}(n,d,\lambda)&\geq \inf_{\hat \pi:\  \hat{\pi}(\overline{G}^{(\iota)})=\overline{G}^{(\iota)} }\sup_{{\bf G}\in \boldsymbol{\mathcal{G}}}  \sum_{\iota=1}^{ n/p }  \mathbf{E}^{(\mathbf{full})}_{\bf G} \left[\|\big(M_{\hat \pi^{-1}} - M_{\pi^{*-1}}\big)_{\overline{G}^{(\iota)}}\|_F^2\right]\\
&\geq  \sum_{\iota=1}^{ n/p } \inf_{\hat{\pi}^{(\iota)}}  \sup_{G^{t(\iota)}}\mathbf{E}^{(\mathbf{full})}_{G^{t(\iota)}} \left[\|\big(M_{\hat \pi^{(\iota)-1}} - M_{\pi^{*-1}}\big)_{\overline{G}^{(\iota)}}\|_F^2\right]\ .
\end{align*}
The problem of estimating the permutation $\pi^*$ is now broken down into the $n/p$ smaller problems of estimating the subsets $G^{t(\iota)} \subset \overline G^{(\iota)}$. The square Euclidean distance between to experts in $\overline G^{(\iota)}$ of experts is $0$ is they are both either in or not in $G^{t(\iota)}$ and it is equal to $\frac{q\upsilon^2}{p\lambda}$ otherwise. Let us focus on the easier problem of estimating the subsets $G^{t(\iota)}$ and define $\hat{G}^{t(\iota)}$ the set of the $p/2$ experts that are ranked above according $\hat \pi^{(\iota)}$. Then, we have that 
\[
	\|\big(M_{\hat \pi^{(\iota)-1}} - M_{\pi^{*-1}}\big)_{\overline{G}^{(\iota)}}\|_F^2=\frac{q\upsilon^2}{p\lambda}\big|\hat{G}^{(\iota)}\Delta G^{t(\iota)}| \geq \frac{q\upsilon^2}{4\lambda}\1\{\hat{G}^{(\iota)}\neq G^{t(\iota)}\}\enspace , 
\]
where the last inequality comes from the construction of the sets $G^{t(\iota)}$ by \Cref{lem:packing}.
Hence, we deduce that

\begin{equation}\label{eq:lower_minimax_1}
	\cR^*_{\perm}(n,d,\lambda)\geq  \frac{q\upsilon^2}{4\lambda}\sum_{\iota=1}^{n/p}\inf_{\hat{\pi}^{(\iota)}}  \sup_{G^{t(\iota)}} \mathbf{P}^{(\mathbf{full})}_{G^{t(\iota)}}\left[\hat{G}^{(\iota)}\neq  G^{t(\iota)}\right] \ ,
\end{equation}

so that by symmetry,

\[	\cR^*_{\perm}(n,d,\lambda)\geq \frac{nq\upsilon^2}{4p\lambda}\inf_{\hat{G}^{(1)}}\sup_{G^{t(1)}}\mathbf{P}^{(\mathbf{full})}_{G^{t(1)}}\left[\hat{G}^{(1)}\neq  G^{t(1)}\right]\enspace . 	
\]
Consider the $p\times d$ matrices $N$ and $Y^{\downarrow}$ defined by 
\[
	N_{ik}= \sum_t \1_{x_t=(i,k)} \ ;\quad \quad  Y^{\downarrow}_{ik} =  \sum_t \1_{x_t=(i,k)} (y_t - w_i) \enspace ,
\]
where $w$ is defined in \Cref{eq:vector_w_lb}.
To simplify the notation, we write henceforth $G$ and $\hat{G}$ for  $G^{t(1)}$ and $\hat{G}^{(1)}$ respectively. Letting $\mathbf{P}_{G}$ for the corresponding marginal distribution of $N$ and $Y^{\downarrow}$, the same sufficiency argument as in \cite{pilliat2022optimal} gives that
\[
\inf_{\hat{G}} \sup_{G}	\mathbf{P}^{(\mathbf{full})}_{G}\left[\hat{G}\neq  G\right]=\inf_{\hat{G}} \sup_{G}	\mathbf{P}_{G}\left[\hat{G}\neq  G\right]\ .  
\]

We finally obtain the following inequality:
\begin{equation}\label{eq:lower_minimax_G_T1}
	\cR^*_{\perm}(n,d,\lambda)\geq  \frac{nq\upsilon^2}{4p\lambda}\inf_{\hat{G}} \sup_{G}	\mathbf{P}_{G}\left[\hat{G}\neq  G\right]\enspace . 
\end{equation}

Let $\mathbf{P}_0$ be the distribution on $N$ and  $Y^{\downarrow}$ corresponding to the case $\upsilon=0$. The entries of $N$ of are independent and follow a poisson distribution of parameter $\lambda$. Conditionally to $N_{ik}$, we have $Y^{\downarrow}_{ik}$ is a gaussian variable with mean $0$ and variance $N_{ik}$. Then, we deduce from Fano's inequality~\cite{tsybakov} that  
\begin{align}\label{eq:KLFano}
	\inf_{\hat{G}} \sup_{G\in \cG} \mathbf P_{G}(\hat{G} \neq G) &\geq 1 - \frac{1 + \max_{G \in \cG} \mathrm{KL}(\mathbf P_{G}||\mathbf P_0)}{\log(|\cG|)}\ ,  
	\end{align}
	where $\mathrm{KL}(.||.)$ stands for the Kullback-Leibler divergence. The following lemma gives an upper bound of these Kullback-Leibler divergences. It can be found in \cite{pilliat2022optimal}, with the slightly stronger assumption that $p\lambda \geq 1$.
\begin{lemma}[Lemma J.2 of \cite{pilliat2022optimal}]\label{lem:KLborn}
There exists a numerical constant $c_1$ such that the following holds true. If $\upsilon^2 \leq 1 \land {p \lambda}$, then for any $G\in \cG$, we have 
\begin{align*}
\mathrm{KL}(\mathbf P_G||\mathbf P_0)&\leq c_1\frac{\upsilon^2q^2}{d}\ \enspace .
\end{align*}
\end{lemma}

The proof of \Cref{lem:KLborn} can be found in \cite{pilliat2022optimal}, with $\tilde n := p$ and $\tilde d = d$. The slighlty stronger assumption that $p\lambda \geq 1$ made in Lemma J.2 in \cite{pilliat2022optimal} is in fact not necessary. Indeed, it is only used to prove that $\cI := \lambda p(e^{\upsilon^2/(\lambda p)} - 1) \leq c_1'\upsilon^2$ in the proof of Lemma J.2 in \cite{pilliat2022optimal}, and this inequality remains valid under the assumption of \Cref{lem:KLborn}, that is $u^2 := \upsilon^2/(\lambda p) \leq 1$.

\medskip

{\bf Step 3: Choice of suitable parameters $p,q$ and $\upsilon$}

\medskip

By combining~\eqref{eq:lower_minimax_G_T1}, \eqref{eq:KLFano}, with Lemma~\ref{lem:KLborn} and the different constraints on the parameters \Cref{eq:condition_v}, we directly obtain the following proposition.
\begin{proposition}\label{prop:lower_minimax_abstract}
There exists a numerical constant $c$ such that if $p \in \{2, \dots, n\}$, $q \in \{1, \dots, d\}$ are dyadic integers, and $\upsilon$ satisfy the following condition:
\begin{eqnarray}\label{eq:cond_1}
\upsilon & \leq& c\left[ 1\wedge \sqrt{p \lambda} \wedge \frac{\sqrt{pd}}{q} \wedge \sqrt{\lambda}\frac{p^{3/2}}{n}\right]\enspace , \label{eq:cond_2}
\end{eqnarray}
then we have
\begin{equation}
\cR^*_{\perm}(n,d,\lambda)\geq   c \frac{n q  \upsilon^2 }{p\lambda } \ .  
\end{equation}
\end{proposition}

The above proposition being a direct consequence of what preceeds it, we consider that it does not require a proof. Let us now apply \Cref{prop:lower_minimax_abstract} for different parameters $p,q$ and $\upsilon$ to conclude the proof of \Cref{th:lower_bound_poisson}.
\medskip

First, using the lower bound in the bi-isotonic case -- see Theorem 4.1 of \cite{pilliat2022optimal}, we have for some constant $c'$ that
\begin{equation}\label{eq:first_lb}
	\cR^*_{\perm}(n,d,\lambda) \geq c' (n/\lambda \land nd) \enspace .
\end{equation}

In what follows, we write $\floor{x}_{\dya}$ for the greatest integer that is a power of two and smaller than $x$.
Let us consider the following inequality:
\begin{equation}\label{eq:assumption_lambda}
	\lambda \geq 1/d \lor n^2/d^3 \enspace .
\end{equation}
In the case where \Cref{eq:assumption_lambda} is not satisfied, then $n\sqrt{d/\lambda}\land n^{2/3}\sqrt{d}\lambda^{-5/6} \leq n/\lambda \land nd$ 
and the lower bound of \Cref{th:lower_bound_poisson} is proven by \Cref{eq:first_lb}.

We subsequently assume that \Cref{eq:assumption_lambda} is satisfied.

\noindent 
{\bf Case 1}: $\lambda n \leq 1$. In this case, we choose $q = \floor{\sqrt{\tfrac{d}{\lambda}}}_{\dya}$ and $p = n/2$. We have that $q \in \{1, \dots, d\}$ since $\lambda \leq 1$ in that case and by assumption \Cref{eq:assumption_lambda}, $\lambda \geq 1/d$. We deduce from \Cref{prop:lower_minimax_abstract} applied with $v/c = \sqrt{p\lambda} = \sqrt{pd}/q$ that 
$$\cR^*_{\perm}(n,d,\lambda) \geq c'' n\sqrt{\tfrac{d}{\lambda}} \enspace .$$

\noindent 
{\bf Case 2}: $\lambda \in [\tfrac{1}{n}, 8n^2]$. In this case, we choose $q = \floor{\tfrac{n^{1/3}\sqrt{d}}{\lambda^{1/6}}}_{\dya}$ and $p = \floor{\tfrac{n^{2/3}}{\lambda^{1/3}}}_{\dya}$. We deduce from \Cref{eq:assumption_lambda} that $q \leq d$. Since $\lambda \in [\tfrac{1}{n}, 8n^2]$, we also necessarily have that that $q \geq 1, p\geq 2$ and $p \leq n$. Applying the above proposition with $\upsilon/c = 1 = \sqrt{pd}/q = \sqrt{\lambda}p^{3/2}/n$, we deduce that 
$$\cR^*_{\perm}(n,d,\lambda) \geq c'' \tfrac{n^{2/3}\sqrt{d}}{\lambda^{5/6}}\enspace .$$

{\bf Case 3}: $\lambda \geq 8n^2$. When $\lambda$ satisfies this condition that is out of the scope of \Cref{th:lower_bound_poisson} but discussed below \Cref{th:lower_bound_poisson}, we choose $q = \floor{\sqrt{d}}_{\dya}$ and $p = 2$. Applying the above proposition with $\upsilon/c = 1$, we deduce that 
$$\cR^*_{\perm}(n,d,\lambda) \geq c'' \tfrac{n\sqrt{d}}{\lambda}\enspace .$$

\medskip 
We have proved that for any $n,d$ and $\lambda$, we have the lower bound 
\begin{equation*}
	\cR^*_{\perm}(n,d,\lambda) \geq c'' \left[n\sqrt{\tfrac{d}{\lambda}} \wedge \tfrac{n^{2/3}\sqrt{d}}{\lambda^{5/6}}\wedge \tfrac{n\sqrt{d}}{\lambda} + n/\lambda \right] \land nd \enspace .
\end{equation*}

This concludes in particular the proof of \Cref{th:lower_bound_poisson}, stated for $\lambda \in [1/d, 8n^2]$.

\end{proof}

\section{Proof of \Cref{prop:concentration_bernstein_op}}

	Let us introduce $\bP_k = \Lambda(X_{\cdot k} X_{\cdot k}^T - \bbE[X_{\cdot k} X_{\cdot k}^T])\Lambda \in \bbR^{p \times p}$, so that

\begin{equation}\label{eq:operator_norm_reduction}
	\Lambda(XX^T - \bbE[XX^T])\Lambda = \sum_{k=1}^q \bP_k \enspace .
\end{equation}

\begin{lemma}\label{lem:prop_concentration_mgf_op}
	There exists a numerical constant $\kappa'''_3$ such that 
	for any \\$x \in [0, (\kappa'''_3(\sigma^2 r_{\Lambda}+ K^2\log(q)))^{-1}]$, we have 
	\begin{equation*}
		\|\bbE[e^{x\bP_k}]\|_{\op} \leq \exp(\kappa'''_3 x^2(\sigma^2 + \sigma^4 p)) + \frac{1}{q} \enspace .
	\end{equation*}
\end{lemma}

Moreover, applying the Matrix Chernoff techniques for the independent matrices $\bP_{k}$ (see lemma 6.12 and 6.13 of \cite{wainwright2019high}), we have for any $t > 0$ that
\begin{align*}
	\log(\bbP(\|\sum_{k=1}^q\bP_k \|_{\op} \geq t)) &\leq \log(\tr\left[\bbE[e^{x \sum_{k=1}^q\bP_k} ]\right]) -xt \\
	&\leq \log\left(\tr\left[\exp\left(\sum_{k=1}^q\log(\bbE[e^{x\bP_k}])\right) \right]\right) -xt \\
	&\leq \log(p) + \sum_{k =1}^q \|\log(\bbE[e^{x \bP_k}]) \|_{\op} - xt \\
	&= \log(p) + \sum_{k =1}^q \log(\|\bbE[e^{x \bP_k}] \|_{\op}) - xt \enspace .\\
\end{align*}

Applying \Cref{lem:prop_concentration_mgf_op}, it holds for any $x \in [0, (\kappa'''_3(\sigma^2 r_{\Lambda}+ K^2\log(q)))^{-1}]$ that
\begin{align*}
	\sum_{k =1}^q \log(\|\bbE[e^{x \bP_k}] \|_{\op}) &\leq  q\log\left(\exp\left(\kappa'''_3 x^2 (\sigma^2 + \sigma^4 p)\right) + \frac{1}{q}\right) \\
	&\leq \kappa'''_3 x^2 (\sigma^2q + \sigma^4 pq) + 1 \enspace .
\end{align*}
where in the last inequality we used the fact that for any $a\geq 1$ and $u> 0$, $\log(a + u) \leq \log(a) + u/a$.

Hence we obtain 
\begin{align*}
	\log(\bbP(\|\sum_{k=1}^q\bP_k \|_{\op} \geq t)) \leq \log(ep) + \kappa'''_3 x^2 (\sigma^2q + \sigma^4 pq) -xt \enspace .
\end{align*}
Hence if $t \leq 2\frac{\sigma^2 q + \sigma^4 pq}{\sigma^2 r_{\Lambda}+ K^2\log(q)}$, we choose $x = \frac{t}{2\kappa'''_3(\sigma^2 q + \sigma^4 pq)}$ and if $t > 2\frac{\sigma^2 q + \sigma^4 pq}{\sigma^2 r_{\Lambda}+ K^2\log(q)}$ we choose $x = \tfrac{1}{\kappa'''_3(\sigma^2 r_{\Lambda}+ K^2\log(q))}$, which gives 

\begin{equation*}
	\bbP(\|\sum_{k=1}^q\bP_k \|_{\op} \geq t) \leq e p \max\left[\exp\left(-\frac{1}{\kappa_3}\left(\frac{t^2}{4(\sigma^2 q + \sigma^4 pq)} \lor \frac{t}{2(\sigma^2 r_{\Lambda}+ K^2\log(q))}\right)\right)\right]\enspace .
\end{equation*}

We deduce that with probability at least $1-\delta$, it holds that
\begin{equation*}
	\|\sum_{k=1}^q\bP_k \|_{\op} \leq \kappa \left[\sqrt{(\sigma^4pq + \sigma^2 q)\log(p/\delta'')} + (\sigma^2r_{\Lambda} + K^2\log(q))\log(p/\delta'')\right]\enspace ,
\end{equation*}
for some numerical constant $\kappa$.

\begin{proof}[Proof of \Cref{lem:prop_concentration_mgf_op}]
	
	Since $\|\Lambda X_{\cdot k}  X_{\cdot k}^T\Lambda \|_{\op} = \|\Lambda X_{\cdot k}\|_2^2$, we state the following lemma controlling the moment generating function of the $L_2$ norm of the projection $\Lambda X_{\cdot k}$:

	\begin{lemma} {\color{red} }\label{lem:prop_concentration_norme_2_bellec}
		There exists a numerical constant $\kappa'''_0$ such that for any $x \leq \tfrac{1}{\kappa'''_0K^2}$ we have
		\begin{equation*}
			\bbE[e^{x \| \Lambda X_{\cdot k} \|_2^2}] \leq e^{\kappa'''_0
			\sigma^2 r_{\Lambda} x} \enspace .
		\end{equation*}
	\end{lemma}
	
	Now we define the event $\xi_{\op} := \{ \max_{k=1,\dots,d} \| \Lambda X_{\cdot k}\|_2^2 \leq \kappa'''_0(\sigma^2 r_{\Lambda} + K^2\log(q^3))\}$, where $\kappa'''_0$ is the numerical constant given by \Cref{lem:prop_concentration_norme_2_bellec}. Applying the same lemma together with the Chernoff bound, a union bound over all $k = 1 \dots d$ gives 
	\begin{equation*}
		\bbP(\xi_{\op}^c) \leq \tfrac{1}{q^2} \enspace .
	\end{equation*}

We consider in what follows the relation order $\preceq$ induced by the cone of nonegative symetric matrices $\bbS_n^+$, namely $X' \preceq X''$ if and only if $X'' - X' \in \bbS_n^+$.
Under the event $\xi_{\op}$, it holds that for any $u \geq 2$,
\begin{align*}
	\bP_k^u &\preceq \|\bP_k \|_{\op}^{u-2} \bP_k^2 \\
	&\preceq \|\Lambda(X_{\cdot k}^T X_{\cdot k} - \bbE[X_{\cdot k}^T X_{\cdot k}]) \Lambda \|_{\op}^{u-2} \bP_k ^2\\
	&\preceq ( \kappa'''_1(\sigma^2 r_{\Lambda} + K^2\log(q)))^{u-2}\bP_k^2 \enspace ,
\end{align*}
for some numerical constant $\kappa'''_1$ (depending on $\kappa'''_0$). In the third inequality we used the definition of $\xi_{\op}$ the fact that $\bbE[\|\Lambda X_{\cdot k}\|_2^2] \leq \kappa'''_0 \sigma^2 r_{\Lambda}$.

We now give an upper bound of $\|\bbE[\bP_k^2]\|_{\op}$, which is the operator norm of the variance of $\bP_k$ as defined in section 6 in \cite{wainwright2019high}.
Remark that since any matrix $U \in \bbR^{q \times q}$ satisfies $U\Lambda U^T \preceq U U^T$, we have that $\bP_k \preceq \Lambda (X_{\cdot k}^T X_{\cdot k} - \bbE[X_{\cdot k}^T X_{\cdot k}])^2 \Lambda$.

Let us compute the expectation of $(X_{\cdot k}^T X_{\cdot k} - \bbE[X_{\cdot k}^T X_{\cdot k}])^2$:
	
	\begin{equation*}
		\bbE[(X_{\cdot k}^T X_{\cdot k} - \bbE[X_{\cdot k}^T X_{\cdot k}])^2]_{ij} = \sum_{l \in P} \bbE[( X_{ik} X_{lk} - \bbE[ X_{ik} X_{lk}])( X_{lk} X_{jk} - \bbE[ X_{lk} X_{jk}])] \enspace .
	\end{equation*}
	The off diagonal terms are zero, and the $i^{th}$ diagonal element satisfies:
	\begin{equation}\label{eq:diagonal_coef_X}
		\bbE[(X_{\cdot k}^T X_{\cdot k} - \bbE[X_{\cdot k}^T X_{\cdot k}])^2]_{ii} = \bbE[( X_{ik}^2 - \bbE[ X_{ik}^2])^2] + \sum_{j \neq i}\bbE[ X_{ik}^2]\bbE[ X_{jk}^2] \enspace .
	\end{equation}
	By assumption \Cref{eq:condition_moments_bernstein}, the first term of \Cref{eq:diagonal_coef_X} satisfies 
	\begin{equation*}
		\bbE[( X_{ik}^2 - \bbE[ X_{ik}^2])^2] \leq 4\bbE[( X_{ik}^4)] \leq   48\sigma^2K^2\enspace .
	\end{equation*}
	The second term of \Cref{eq:diagonal_coef_X} is smaller than $\sigma^4p$, still by assumption \Cref{eq:condition_moments_bernstein}. Hence we have some numerical constant $\kappa'''_2$ that
	
	$$\|\bbE[\bP_k^2]\|_{\op} \leq \|\bbE[( X_{ik}^2 - \bbE[ X_{ik}^2])^2]\|_{\op} \leq \kappa'''_2 (\sigma^2 + \sigma^4 p) \enspace .$$

	Now, by the definition of the exponential of matrices, the triangular inequality and the fact that $\bP_k$ is centered, we have
	\begin{equation}\label{eq:decomposition_evenement}
		\|\bbE[\exp(x\bP_k)]\|_{\op} = 1 + \sum_{u \geq 2} \frac{x^u}{u!}\|\bbE[\bP_k^u\1_{\xi_{\op}}]\|_{\op} + \sum_{u \geq 2} \frac{x^u}{u!}\|\bbE[\bP_k^u\1_{\xi^c_{\op}}]\|_{\op} \enspace .
	\end{equation}

	By definition of $\xi_{\op}$ together with the upper bound of the variance of $\bP^2_k \1_{\xi_{\op}} \preceq \bP^2_k$, it holds for any $x \in [0, (\kappa'''_1(\sigma^2 r_{\Lambda} + K^2\log(q)))^{-1}]$ that 

	\begin{align*}
		\sum_{u \geq 2} x^u\|\bbE[\bP_k^u\1_{\xi_{\op}}]\|_{\op} &\leq x^2 \|\bbE[\bP_k^2]\|_{\op}\sum_{u\geq 2} \frac{x^{u-2}}{u!} (\kappa'''_1(\sigma^2 r_{\Lambda} + K^2\log(q)))^{u-2}\\
		&\leq x^2 \kappa'''_2(\sigma^2 +\sigma^4 p) \sum_{u \geq 0} \frac{x^u}{(u+2)!} (\kappa'''_1(\sigma^2 r_{\Lambda} + K^2\log(q)))^u \\
		&\leq \exp(\kappa'''_3 x^2 (\sigma^2 +\sigma^4 p)) - 1 \enspace , \\
	\end{align*}
	for some numerical constant $\kappa'''_3$.
	We now control the second term of \Cref{eq:decomposition_evenement} under the complementary event $\xi_{\mathrm{\op}}$, for any $x \in [0, (2\kappa'''_0(\sigma^2 r_{\Lambda} + K^2\log(q)))^{-1}]$:
	
	\begin{align*}
		\sum_{u \geq 2} \frac{x^u}{u!}\|\bbE[\bP_k^u\1_{\xi^c_{\op}}] &~\leq \bbE[\exp(x\|\bP_k \|_{\op}\1_{\xi_{\op}^c}]]  \\
		&\overset{(a)}{\leq} \sqrt{\frac{1}{q^2}}\sqrt{\bbE[\exp(2x \|\bP_k \|_{\op})]} \\
		&\overset{(b)}{\leq} \frac{1}{q}\exp(x \kappa'''_0 \sigma^2 r_{\Lambda}) \\
		&~\leq \frac{1}{q}\enspace ,
	\end{align*}
	where in $(a)$ we used the cauchy-schwarz inequality for real random variables and in $(b)$ we applied \Cref{lem:prop_concentration_norme_2_bellec}.
	
	\end{proof}

	\begin{proof}[Proof of \Cref{lem:prop_concentration_norme_2_bellec}]
		We use the result of \cite{bellec2019concentration} which is a generalization of the Hanson-Wright inequality to random variables with coefficients with bernstein's moments.
		
		[Assumption 1 of \cite{bellec2019concentration}] is satisfied with parameters $\sigma^2$ and $K$, and we have the following upper bound on the moment generating function of the quadratic form $\|\Lambda  X_{\cdot k}^T \|_2^2 = | X_{\cdot k} \Lambda  X_{\cdot k}^T |$:
	
		\begin{equation}
			\bbE[e^{x \| \Lambda X_{\cdot k}^T \|_2^2}] \leq e^{x \bbE[\| \Lambda X_{\cdot k}^T \|_2^2]} e^{\kappa'''_0 x^2 K^2\sigma^2\|\Lambda \|_F^2} \leq e^{\kappa'''_1 x \sigma^2 r_{\Lambda}}\enspace ,
		\end{equation}
		for any $x$ satisfying condition $(6)$ of \cite{bellec2019concentration}, that is $128 x\|\Lambda \|_{\op} K^2 \leq 1$. For the last inequality, we used the fact that $\|\Lambda\|_F^2 = \rank(\Lambda)$. We obtain the result by choosing $\kappa'''_2 = \kappa'''_1 \lor 128$.
	\end{proof}

\bibliographystyle{abbrv}
\bibliography{biblio}

\end{document}